\documentclass{amsart}
\usepackage{enumitem}
\usepackage{amssymb}
\usepackage{amsthm}
\usepackage{amsmath}
\usepackage{mathrsfs}
\usepackage{comment}
\usepackage{graphicx}
\usepackage[all]{xy}

\usepackage[pdfencoding=auto, psdextra]{hyperref}

\usepackage[nonewpage]{imakeidx}
\makeindex

\usepackage[columns=1]{idxlayout} 

\usepackage[cal=boondoxo]{mathalfa}

\usepackage[usenames]{color}

\theoremstyle{plain}
\newtheorem{proposition}{Proposition}[section]
\newtheorem{theorem}[proposition]{Theorem}
\newtheorem{Theorem}[proposition]{Theorem}
\newtheorem{lemma}[proposition]{Lemma}
\newtheorem{corollary}[proposition]{Corollary}
\theoremstyle{definition}

\newtheorem{definition}[proposition]{Definition}

\theoremstyle{remark}
\newtheorem{remark}[proposition]{Remark}

\newtheorem{question}[proposition]{Question}

\numberwithin{equation}{section}

\let\oldtocsection=\tocsection

\let\oldtocsubsection=\tocsubsection

\let\oldtocsubsubsection=\tocsubsubsection

\renewcommand{\tocsection}[2]{\hspace{0em}\oldtocsection{#1}{#2}}
\renewcommand{\tocsubsection}[2]{\hspace{1em}\oldtocsubsection{#1}{#2}}
\renewcommand{\tocsubsubsection}[2]{\hspace{2em}\oldtocsubsubsection{#1}{#2}}

\DeclareMathOperator{\Aut}{Aut}

\DeclareMathOperator{\Tr}{tr}

\DeclareMathOperator{\SL}{\mathsf{SL}}
\DeclareMathOperator{\GL}{\mathsf{GL}}

\DeclareMathOperator{\SU}{\mathsf{SU}}
\DeclareMathOperator{\ML}{\mathsf{ML}}
\DeclareMathOperator{\PSL}{\mathsf{PSL}}
\DeclareMathOperator{\PGL}{\mathsf{PGL}}
\DeclareMathOperator{\Hom}{Hom}

\DeclareMathOperator{\id}{id}

\DeclareMathOperator{\Isom}{Isom}

\DeclareMathOperator{\Ad}{Ad}

\DeclareMathOperator{\gL}{\mathfrak{g}}

\def\co{\colon\thinspace}

\DeclareMathOperator{\Cc}{\mathcal{C}}
\DeclareMathOperator{\Ec}{\mathcal{E}}
\DeclareMathOperator{\Fc}{\mathcal{F}}
\DeclareMathOperator{\Gc}{\mathcal{G}}

\DeclareMathOperator{\Lc}{\mathcal{L}}

\DeclareMathOperator{\Tc}{\mathcal{T}}
\DeclareMathOperator{\Uc}{\mathcal{U}}

\DeclareMathOperator{\Pc}{\mathcal{P}}
\DeclareMathOperator{\Sc}{\mathcal{S}}
\DeclareMathOperator{\Xc}{\mathcal{X}}
\DeclareMathOperator{\Yc}{\mathcal{Y}}

\DeclareMathOperator{\Gs}{\mathsf{G}}

\DeclareMathOperator{\Cb}{\mathbb{C}}

\DeclareMathOperator{\Fb}{\mathbb{F}}

\DeclareMathOperator{\Hb}{\mathbb{H}}

\DeclareMathOperator{\Kb}{\mathbb{K}}
\DeclareMathOperator{\Nb}{\mathbb{N}}

\DeclareMathOperator{\Rb}{\mathbb{R}}

\DeclareMathOperator{\Zb}{\mathbb{Z}}
\DeclareMathOperator{\Qb}{\mathbb{Q}}

\DeclareMathOperator{\X}{\mathfrak{X}}
\DeclareMathOperator{\Ks}{\mathsf{K}}

\DeclareMathOperator{\kL}{\mathfrak{k}}

\DeclareMathOperator{\aL}{\mathfrak{a}}

\DeclareMathOperator{\PSLC}{\mathsf{PSL}_2(\Cb)}
\DeclareMathOperator{\MCG}{\mathrm{MCG}}
\DeclareMathOperator{\PSLR}{\mathsf{PSL}_2(\Rb)}
\DeclareMathOperator{\PGLR}{\mathsf{PGL}_2(\Rb)}
\DeclareMathOperator{\Out}{\mathrm{Out}}
\DeclareMathOperator{\Teich}{\mathrm{Teich}}
\DeclareMathOperator{\eu}{\mathrm{eu}}
\DeclareMathOperator{\Xf}{\mathfrak{X}}
\DeclareMathOperator{\Homeo}{\mathrm{Homeo}}
\DeclareMathOperator{\cPSLR}{\widetilde{\mathsf{PSL}_2(\Rb)}}
\DeclareMathOperator{\Ztwo}{\mathbb{Z}/2\mathbb{Z}}

\newcommand{\abs}[1]{\left|#1\right|}
\newcommand{\conj}[1]{\overline{#1}}

\newcommand{\ip}[1]{\left\langle #1\right\rangle}

\newcommand{\vertiii}[1]{{\left\vert\kern-0.25ex\left\vert\kern-0.25ex\left\vert #1 
    \right\vert\kern-0.25ex\right\vert\kern-0.25ex\right\vert}}
	
	\newcounter{notes}



\begin{document}

\title[On character varieties for surface groups]{Dynamics and geometry of character varieties for surface groups}

\author[Chasteen-Boyd D.]{David Chasteen-Boyd}
\address{Department of Mathematics, University of Virginia}
\email{kxk2dr@virginia.edu}

\author[Maloni S.]{Sara Maloni}
\address{Department of Mathematics, University of Virginia}
\email{sm4cw@virginia.edu}
\urladdr{sites.google.com/view/sara-maloni/}

\author[Schlich S.]{Suzanne Schlich}
\address{Dipartimento di Matematica ``Giuseppe Peano'', Università di Torino}
\email{suzanne.schlich@unito.it}
\urladdr{https://suzanne-schlich.github.io/}

\thanks{D.C.B. and S.M. were partially supported by U.S. National Science Foundation grant DMS-1848346 (NSF CAREER). This material is based upon work supported by the National Science Foundation under Grant No. DMS-1928930, while S.M. was in residence at the Simons Laufer Mathematical Sciences Institute in Berkeley, California, during Spring 2026.
S.S. is funded by the European Union (ERC, GENERATE, 101124349). Views and opinions expressed are however those of the author only and do not necessarily reflect those of the European Union or the European Research Council Executive Agency. Neither the European Union nor the granting authority can be held responsible for them.} 

\date{\today}

\begin{abstract}
The problem of classifying geometric structures on manifolds is very much related to the discussion of the automorphism groups actions on character varieties, which are spaces of equivalence classes of representations. In this chapter we survey some results on this topic, mostly focusing on representations of surface groups (both in the orientable and non-orientable cases) and free groups. An important principle in the study of the dynamics on character varieties $\mathfrak{X} = \mathfrak{X}(\pi_1(S), \Gs)$ for surface groups $\pi_1(S)$ is the following dichotomy: when the target group $\Gs$ is compact, $\mathfrak{X}$ has nontrivial homotopy type, and the action of the mapping class group is chaotic; whereas when the target group $\Gs$ is non-compact, $\mathfrak{X}$ contains contractible sets on which the mapping class group acts properly. We will expand on this dichotomy in various cases. After introducing the necessary background, we will discuss representations into $\PSLR$ and $\PGLR$, discussing the number of connected components, the geometric properties (Bowditch question), the dynamics (Goldman conjecture) and some components with an `exotic' behaviour (Deroin-Tholozan representations). We will also underline how the theory for representations of fundamental groups of orientable closed hyperbolizable surfaces needs to be adapted when one considers surfaces with punctures or non-orientable surfaces. This will serve as a guide and motivation for other sections. We will then discuss representations into compact groups, where we will discuss mostly ergodicity results in various settings, and some non-ergodicity results at the end. Thirdly, we will consider representations in $\PSLC$. We will discuss convex-cocompact representations, primitive-stable and Bowditch representations and their relationship, also in the more general setting of representations into rank-one Lie groups. Finally, we will describe how some of the results mentioned can be generalized for representations into higher-rank Lie groups. 
\end{abstract}

\maketitle

\tableofcontents

\section{Introduction}\label{intro}

In this chapter we survey results on dynamics on character varieties.  In particular, we decided to focus on representations of surface groups, discussing in most cases how the theory needs to be adapted when one considers representations from fundamental groups of punctured surfaces or non-orientable surfaces. In addition, we will discuss extensively the dichotomy between representations into compact groups and into non-compact groups. We will try to summarize some ideas behind certain important results in the field, and hope this document will help the readers navigate the literature. Note that there are other surveys on dynamics on character varieties, see for example Goldman \cite{goldman-survey} and Canary \cite{can_dyn}, among others, and we strive to make our work complementary to theirs as much as possible. Given the wide spread of the topic discussed, we will not aim to be exhaustive, but to present a slice of the theory, according to our expertise. 

Given a closed, orientable surface $S$ with genus $g\geq 2$, it is a classical result that the data of the following three objects are equivalent:
\begin{itemize} 
    \item a discrete and faithful representation $\rho:\pi_1(S)\to\PSLR$;
    \item a (marked) hyperbolic structure on $S$;
    \item a (marked) complex structure on $S$.
\end{itemize}
 In addition, the transformation of such a representation by conjugation is equivalent to changing the marking of the corresponding structure by an isomorphism isotopic to the identity. (See, for example, Farb--Margalit \cite{FM} for a discussion of this classical correspondence.) This fact naturally leads to the study of the \textit{character variety}\index{Character variety} $\mathfrak{X}(\pi_1(S),\PSLR) := \mathrm{Hom}(\pi_1(S),\PSLR) / \PSLR$ of conjugation classes of representations. More generally, given a finitely presentable group $\Gamma$ and a Lie group $\Gs$, we consider the character variety $\mathfrak{X}(\Gamma,\Gs) := \mathrm{Hom}(\Gamma,\Gs) / \Gs$.  When the target group $\Gs$, like in this case, is a non-compact group, this conjugation quotient can behave poorly, so we will introduce in Section \ref{character} the polystable quotient, which, restricting to closed orbits, ensures a better algebraic structure that is, for example, Hausdorff. In this chapter we are interested in the topological and geometrical properties of the different connected components of $\mathfrak{X}(\Gamma,\Gs)$ and on the action of the outer automorphism group $\mathrm{Out}(\Gamma)$ on it.

When $\Gamma = \pi_1(S)$ for a connected, orientable, hyperbolizable surface, the group $\Out(\Gamma)$ corresponds, up to finite index, with the mapping class group $\mathrm{MCG}(S)$. An important principle in the study of the dynamics of the mapping class group $\mathrm{MCG}(S)$ on the character variety $\mathfrak{X} = \mathfrak{X}(\pi_1(S), \Gs)$ for surface groups $\pi_1(S)$ is the following dichotomy: when the target group $\Gs$ is compact, $\mathfrak{X}$ has nontrivial homotopy type, and the action of the mapping class group is chaotic; while when the target group $G$ is non-compact, $\mathfrak{X}$ contains contractible sets on which the mapping class group acts properly. For example, when $G= \mathsf{U}(1)$, the character variety can be identified with the Jacobian for a Riemann surface $X$ homeomorphic to $S$, which parametrizes topologically trivial holomorphic complex line bundles over $X$. Using the identification of $\mathfrak{X}$ with the ordinary cohomology group with values in $\Rb/\Zb$, one can see that the action of the mapping class group $\mathrm{Mod}(S)$ on $\mathfrak{X}$ factors through the symplectic representation and can be seen to be chaotic (ergodic). Recall that an action of a group $\Gamma$ on a measurable space $X$ equipped with a measure $\mu$ is called \textit{ergodic}\index{Ergodic action} if the only $\Gamma$--invariant measurable sets $A$ satisfy $\mu(A)=0$ or $\mu(X\setminus A)=0$. On the other hand, when $\Gs= \PSLR$, one can see that in $\mathfrak{X}$ there are two contractible connected components, corresponding to Teichm\"uller spaces, and a result often attributed to Fricke proves that the action of $\mathrm{Mod}(S)$ on these components is properly discontinuous. Recall that an action of a group $\Gamma$ on a space $X$ is called \textit{proper discontinuous}\index{Proper discontinuous action} if for any compact subset $K \subset X$, the cardinality of the set of elements $\gamma \in \Gamma$ such that $K \cap \gamma\cdot K$ is not empty is finite. See Kapovich \cite{Kapovich-proper} for a discussion on different equivalent characterizations. 

In this survey we will explore open cases around this dichotomy for closed and punctured surfaces, both in the orientable and non-orientable case. We will start with a review, in the Background (Section \ref{back}), of a few concepts that will  be important in many of the following sections: an overview of the topology of non-orientable surfaces, the definition of character variety and its symplectic structure, and the notion of Euler class for a representation. In Section \ref{PSL-R} we will discuss representations into $\PSLR$ and $\PGLR$, by discussing first the case of representations from the fundamental group of closed orientable surfaces, and then discuss the case of non-orientable surfaces and punctured surfaces.  We will end the section with a discussion of the ``super-maximal'' representations studied by Deroin and Tholozan. In Section \ref{compact} we will instead focus on representations into compact groups, where, again, we will consider orientable surfaces first, followed by non-orientable surfaces and free groups after. We will end this section with a discussion of dynamics of certain subgroups of the mapping class group. In Section \ref{PSL-C} we will consider representations into $\PSL_2(\Cb)$, where we will start recalling the theory of convex-cocompact representations, before discussing Minsky's primitive-stable representations and Bowditch representations. We will end with a discussion of  the relationship between these definitions. In the last section (Section \ref{higher-rank}), we will give an overview of how to generalize these discussions for representations into higher-rank Lie groups. 

\section*{Acknowledgments}

The authors would like to thank Sean Lawton and Fr\'ed\'eric Palesi for answering many questions on quotients of representations, ergodicity results and results about non-orientable surfaces and Inyoung Ryu for clarification on her work. 

\section{Background}\label{back}

\subsection{Topology of surfaces}\label{sec:top}

We begin with an overview of the topology of surfaces and establish some notational conventions that will be used throughout. We will assume that all surfaces are smooth and connected unless otherwise stated. We will generally refer to an arbitrary surface as $S$ and an arbitrary non-orientable surface as $N$.

\subsubsection{Orientable Surfaces}\label{orient}

Every closed, orientable surface $S$ is diffeomorphic to a connected sum of some number of tori, i.e. $S\cong\mathop{\#}\limits_{i=1}^{g}T_i$ where the number $g$ is its genus. We will use $S_g$ to denote the closed, orientable surface of genus $g$. Its fundamental group has a presentation consisting of $2g$ generators and one relation, $$\pi_1(S_g)=\ip{a_1,b_1,\dots,a_g,b_g\ \big|\ \prod_{i=1}^g[a_i,b_i]}.$$ Similarly, every compact, orientable surface is diffeomorphic to some $S_g$ with a finite number of open disks with disjoint closures removed. We will use $S_{g,n}$ to denote the closed, orientable surface with genus $g$ and $n$ boundary components. The fundamental group of $S_{g,n}$ is a free group on $2g+n-1$ generators and has a presentation similar to that of $\pi_1(S_g)$ as $$\pi_1(S_{g,n}) = \ip{a_1,b_1,\dots,a_g,b_g,c_1, \dots,c_n\ \big|\ \left(\prod_{i=1}^g[a_i,b_i]\right)\left(\prod_{i=1}^nc_i\right)}.$$ We will also consider surfaces with finitely many punctures, that is, points removed from the interior of the surface; we will generally also denote the closed, orientable genus $g$ surface with $n$ punctures by $S_{g,n}$. This is somewhat an abuse of notation, though the two surfaces denoted by $S_{g,n}$ are almost identical in the sense that one is a deformation retract of the other, they have isomorphic fundamental groups, and even the standard presentations for their fundamental groups are identical. Whether $S_{g,n}$ is denoting a surface with $n$ punctures or $n$ boundary components will generally be dictated by the geometric context rather than by anything topological. Later in this section we will highlight some rules of thumb about when to consider $S_{g,n}$ as having punctures or boundary components.

The \textit{mapping class group}\index{Mapping class group} of a closed, orientable surface $S_g$ is the group of orientation-preserving homeomorphisms of the surface taken up to isotopy, i.e. $\MCG(S_g)=\Homeo^+(S_g)/\Homeo_0(S_g)$. Dehn \cite{dehn_papers_1987} showed that this group is finitely generated by Dehn twists about simple closed curves on $S_g$. Let $\MCG^{\pm}(S_g):=\Homeo(S_g)/\Homeo_0(S_g)$ be the \emph{extended mapping class group}, which contains $\MCG(S_g)$ as an index 2 subgroup. The Dehn--Nielsen--Baer Theorem states that for $g\geq 1$ we have an isomorphism $\MCG^{\pm}(S_g)\cong\Out(\pi_1(S_g))$. We will see in Section \ref{character} that $\Out(\Gamma)$ naturally acts on the character variety $\Xf(\Gamma,\Gs)$; the Dehn--Nielsen--Baer Theorem gives a geometric interpretation of this action when $\Gamma$ is the fundamental group of a surface.

For a surface $S_{g,n}$ with boundary, we instead consider the group of orientation-preserving homeomorphisms of the surface that restrict to the identity on the boundary, taken up to isotopy, i.e. $\MCG(S_{g,n})=\Homeo^+(S_g,\partial S_{g,n})/\Homeo_0(S_g,\partial S_{g,n})$. Note that if a homeomorphism fixes the boundary of $S_{g,n}$ then it is automatically orientation-preserving, so there is not a version of the extended mapping class group for $S_{g,n}$. Because elements of $\MCG(S_{g,n})$ must fix the boundary components, $\MCG(S_{g,n})$ is not isomorphic to $\Out(\pi_1(S_{g,n}))$ but is rather isomorphic to the proper subgroup of it that fixes (the conjugacy classes of) the boundary curves $c_1,\dots,c_n$.

If $S_{g,n}$ is instead a punctured surface with punctures $\{p_1,\dots,p_n\}$, we define $\MCG(S_{g,n})$ to be the group of orientation-preserving homeomorphisms of the surface that fix the points $p_i$, taken up to isotopy, i.e. $\MCG(S_{g,n})=\Homeo^+(S_g,\{p_1,\dots,p_n\})/\Homeo_0(S_g,\{p_1,\dots,p_n\})$ (note that this is also often called the \emph{pure} mapping class group of $S_{g,n}$). Similar to the case of surfaces with boundary, $\MCG^{\pm}(S_{g,n})$ is isomorphic to the subgroup of $\Out(\pi_1(S_{g,n}))$ that fixes (the conjugacy class of) each of the peripheral curves $c_i$ around the puncture $p_i$ for $i=1,\dots,n$. See \cite{FM} for more details on mapping class groups and a proof of the Dehn--Nielsen--Baer Theorem.

We finish with a brief discussion of the geometric differences between punctures and boundary components. A similar discussion also applies to non-orientable surfaces. We will denote by $c_i$ both the simple closed curve homotopic to the boundary component and the peripheral simple closed curve around the puncture. For representations $\rho:\pi_1(S_{g,n})\to\PSLR$, the difference will generally come from what type of isometry $\rho(c_i)$ is. First consider when $\rho(c_i)$ is parabolic; specifically, let $\rho(c_i)=\pm\begin{pmatrix}
    1 & 1\\
    0 & 1
\end{pmatrix}$ so that $\rho(c_i)$ fixes the point at $\infty$. Then $\Hb^2/\ip{\rho(c_i)}$ is homeomorphic to an open annulus. For any $y>0$, the set of points $(t,y)_{t\in[0,1]}$ in the upper-half space model of $\Hb^2$ projects to a simple closed curve and as $y$ increases the lengths of these curves decreases and the points on the curves converge towards the point at $\infty$; for this reason we think of the point at $\infty$ as a puncture in the surface $\Hb^2/\ip{\rho(c_i)}$. On the other hand, if $\rho(c_i)=\pm\begin{pmatrix}
    \sqrt{2} & 0\\
    0 & \frac{1}{\sqrt{2}}
\end{pmatrix}$ is hyperbolic, $\Hb^2/\ip{\rho(c_i)}$ is again homeomorphic to an open annulus, but now there is a unique shortest curve in $\Hb^2/\ip{\rho(c_i)}$ homotopic to the image of the points $(0,t)_{t\in[1,2]}$; we want to think of this unique shortest curve as a boundary component of the surface. Finally, if $\rho(c_i)=\pm\begin{pmatrix}
    \cos\theta & -\sin\theta\\
    \sin\theta & \cos\theta
\end{pmatrix}$ is elliptic then $\Hb^2/\ip{\rho(c_i)}$ is not a manifold as $\rho(c_i)$ has a fixed point, and will not even be Hausdorff if $\theta$ is an irrational angle. So, we will most commonly think of $S_{g,n}$ as having boundary if the $c_i$s are restricted to being hyperbolic. On the other hand, if the $c_i$s are restricted to being parabolic we will generally think of $S_{g,n}$ as having punctures.

In addition to these geometric reasons, many of the arguments presented will involve either cutting a surface along a simple closed curve or gluing two surfaces together along boundary components. If no specific boundary conditions are stated, or if we are considering representations into groups without the same classification of elements as $\PSLR$, we will most commonly think of $S_{g,n}$ as having boundary components instead of punctures.

\subsubsection{Non-orientable Surfaces}

Just as every closed, orientable surface is diffeomorphic to the connected sum of some number of tori, every closed, non-orientable surface is diffeomorphic to the connected sum of some number of projective planes; we call this number the \emph{non-orientable genus} of the surfaces and denote the closed non-orientable surface of genus $g$ by $N_g$. Schematically, each copy of $\mathbb{RP}^2$ is represented by the symbol $\otimes$, as below, and is called a cross-cap. A cross-cap can also be thought of as the result of removing an open disk from the surface and then identifying the antipodal points on the resulting circle boundary component, or as the result of removing an open disk and gluing the boundary of a M\"obius band to the resulting circle boundary component.

\begin{figure}
[hbt] \centering
\includegraphics[height=4 cm]{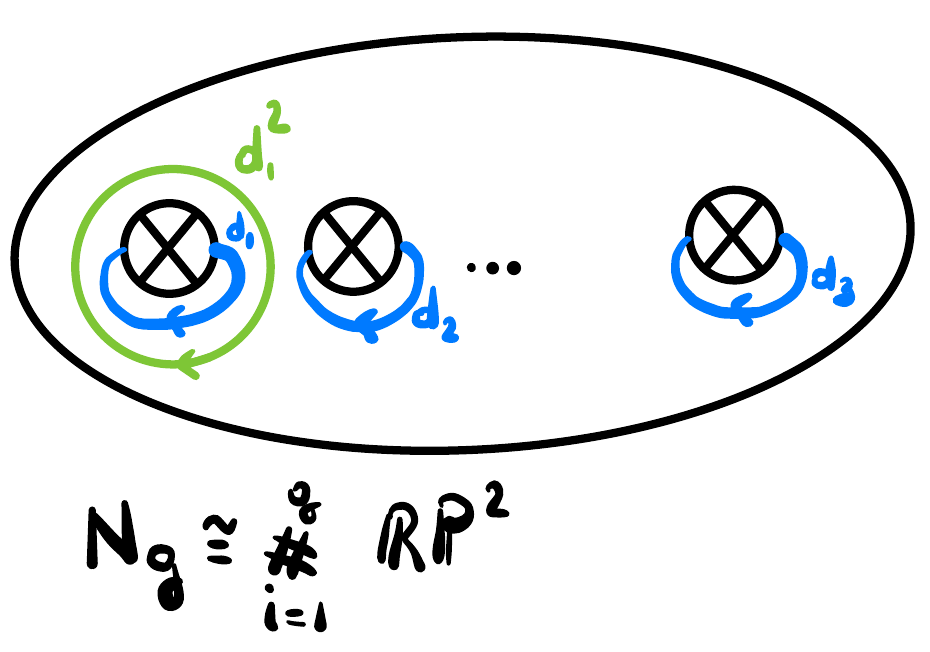}
\caption{The closed surface $N_g$ with $g$ cross-caps. The blue curves $d_1,\dots,d_g$ are one-sided curves that meet the cross-caps at one point. The green curve $d_1^2$ is a two-sided curve that encircles the first cross-cap, and is the boundary of a M\"obius band neighborhood of $d_1$.}
\label{fig:ng}
\end{figure}

With this convention $N_1$ is diffeomorphic to $\mathbb{RP}^2$ and $N_2$ is diffeomorphic to the Klein bottle. It is not always clear, given any non-orientable surface, where it falls in this classification. Take, for example, $N=\mathbb{RP}^2\#S_1$, the connected sum of a projective plane and a torus. This surface is diffeomorphic to $N_3$. More generally, we have 
\begin{align*}
N_{k} \cong N_{k-2m}\#S_m \quad \text{ for } 0\leq m < \left\lfloor  \frac{k}{2} \right\rfloor, 
\end{align*}
that is, two copies of $\mathbb{RP}^2$ can be turned into one copy of a torus, as long as there is at least one additional copy of $\mathbb{RP}^2$. Every non-orientable surface also has an orientable covering surface of index 2, namely, $N_g$ is covered by $S_{g-1}$.

Simple closed curves on a non-orientable surface are classified as either one-sided or two-sided; a curve $\gamma$ is \textit{one-sided} if it has a regular neighborhood that is a M\"obius band and \textit{two-sided} if it has regular neighborhood that is an annulus. Note that there are simple closed curves in a non-orientable surface that can be written as powers of other curves, namely any two-sided curve that is the square of a one-sided curve. Because a one-sided curve is the core curve of a M\"obius band, this means that a two-sided curve is non-primitive if it bounds a M\"obius band in the surface. In Figure \ref{fig:ng}, the curves $d_i$ are one-sided simple closed curves and the curve $d_1^2$ is a non-primitive two-sided simple closed curve that can be thought of as the boundary of a M\"obius band neighborhood of $d_1$.

The fundamental group of $N_g$ has a presentation as $$\pi_1(N_g)=\left\langle d_1,\dots,d_g\ \big|\ \prod_{i=1}^g d_i^2\right\rangle$$ where the generators $d_i$ are the one-sided curves shown above. For a non-orientable surface of odd genus, we can use the fact that $N_{2k+1}\cong N_1\# S_{k}$ to write the fundamental group as $$\pi_1(N_{2k+1})=\left\langle d,a_1,b_1,\dots,a_k,b_k\ \big|\ d^2\prod_{i=1}^k[a_i,b_i]\right\rangle.$$ This contains the subgroup $\langle d^2,a_1,b_1,\dots,a_k,b_k\rangle$, which is isomorphic to the fundamental group of the compact orientable genus $k$ surface with one boundary component; this is the fundamental group of the surface obtained by cutting $N_1\# S_{k}$ along a one-sided curve. As we did above for orientable surfaces, we will denote with $N_{g,n}$ the closed, orientable surface with
genus $g$ and $n$ boundary components. The fundamental group of $N_{g, n}$ is a free group of rank $g + n - 1$, and has the following presentation:
$$\pi_1(N_{g, n}) = \left\langle d_1\ldots,d_g, c_1, \ldots, c_n\ \big|\ \left(\prod_{i=1}^g d_i^2\right) \left( \prod_{i=1}^n c_i\right)\right\rangle.$$

The mapping class group of a non-orientable surface $N$ is the group of homeomorphisms of the surface taken up to isotopy, i.e. $\MCG(N)=\Homeo(N)/\Homeo_0(N)$. Chillingworth \cite{chillingworth} showed that for surfaces with non-orientable genus $g\geq 2$ this group is finitely generated by Dehn twists and one additional mapping class, called a cross-cap slide or $Y$-homeomorphism, defined earlier by Lickorish \cite{lickorish}. This $Y$-homeomorphism is supported on an embedded copy of $N_{2,1}$ and, roughly speaking, is the result of pushing one of the cross-caps along the core of the other and inverting the orientation of the two-sided curve encircling the latter cross-cap. Lickorish \cite{lickorish-note} showed that the subgroup of $\MCG(N)$ generated by the Dehn twists about the two-sided curves is of index 2 in $\MCG(N)$, so for the purposes of the dynamical arguments later in this paper we generally need only consider this subgroup. As with orientable surfaces, $\MCG(N_g)$ is isomorphic to $\Out(\pi_1(N_g))$ and the mapping class groups of surfaces with punctures and boundaries are isomorphic to the corresponding subgroup of $\Out(\pi_1(N_{g,n}))$ \cite{nonor_mcg}. 

\subsection{Character varieties}\label{character}

Given a finitely presentable group $\Gamma$ endowed with the discrete topology and a Lie group $\Gs$, we consider the space $\mathrm{Hom}(\Gamma,\Gs)$ of homomorphisms from $\Gamma$ to $\Gs$ endowed with the compact-open topology. The group $\Gs$ acts on the space $\mathrm{Hom}(\Gamma,\Gs)$ by conjugation, that is, we have $g\cdot\rho:= \iota_g\circ \rho$, where $\iota_g\in \mathrm{Inn}(\Gs)$ is the inner automorphism defined by $\iota_g(x) = g\cdot x \cdot g^{-1}$. When $\Gs$ is a non-compact group, the conjugation quotient can behave poorly, so we will introduce the polystable quotient, which restricts the focus to closed orbits to ensure better algebraic structures. For example, in the cases of representations discussed in this survey, this quotient is Hausdorff, see Florentino--Lawton \cite{flo-law}. 

Let us denote $\mathrm{Orb}_{\Gs}(\rho)$ the $\Gs$-orbit of the representation $\rho$.  We define the subspace $\mathrm{Hom}(\Gamma,\Gs)^*$ of {\it polystable} homomorphisms as follows: 
$$\mathrm{Hom}(\Gamma,\Gs)^*:=\{\rho\in\mathrm{Hom}(\Gamma,\Gs)\ |\ \mathrm{Orb}_{\Gs}(\rho)=\overline{\mathrm{Orb}_{\Gs}(\rho)}\}.$$ 
We then define the {\it polystable} quotient space as follows: $$\mathfrak{X}(\Gamma,\Gs):=\mathrm{Hom}(\Gamma,\Gs)^*/\Gs.$$  We will call this quotient the $\Gs$--{\it character variety of} $\Gamma$, even though in this generality it may be neither a variety, see Casimiro--Florentino--Lawton--Oliverira \cite{CFLO}, nor correspond to characters, see Lawton--Sikora\cite{LaSi}.

Note that when $\Gs$ is compact every homomorphism is polystable, so we have $\Xf(\Gamma,\Gs)=\mathrm{Hom}(\Gamma,\Gs)/\Gs$. On the other hand, consider the case when $\Gamma=\Zb$ and $\Gs=\PSLR$, and define the representation $\rho_1:\Gamma\to\Gs$ given by $\rho_1(1)=\pm\begin{pmatrix}
    1 & 1 \\
    0 & 1
\end{pmatrix}$. For any real number $a>0$, we define $g_a:=\pm\begin{pmatrix}
    \sqrt{a} & 0\\
    0 & \frac{1}{\sqrt{a}}
\end{pmatrix}$ and $\rho_a:=g_a\rho_1 g_a^{-1}$, which is in $\mathrm{Orb}_{\Gs}(\rho_1)$ and has $\rho_a(1)=\pm\begin{pmatrix}
    1 & a \\
    0 & 1
\end{pmatrix}$. Because $a$ can be arbitrarily small, the conjugacy class $[\rho_1]$ contains elements that are arbitrarily close to the trivial representation, but the trivial representation is not in $\mathrm{Orb}_{\Gs}(\rho_1)$, meaning that $\mathrm{Orb}_{\Gs}(\rho_1)\neq\overline{\mathrm{Orb}_{\Gs}(\rho_1)}$.

The $\Gs$-character variety of $\Gamma$ is also often defined in the literature from the point of view of algebraic geometry. In this paper we will focus on the definition of $\Xf(\Gamma,\Gs)$ as the polystable quotient space, but we give here a brief introduction to the GIT quotient definition. Let $\Gs$ be a (complex) affine algebraic group, that is, a subgroup of $\GL_n(\Cb)$ defined by polynomial equations. If $\Gamma$ has a generating set of $n$ elements, then $\Hom(\Gamma,\Gs)$ can be identified with the collection of points in $\Gs^n$ satisfying the relations in $\Gamma$, so it has the structure of an affine algebraic variety over $\Cb$. Let $\Cb[\Hom(\Gamma,\Gs)]$ be the ring of $\Cb$-valued functions on $\Hom(\Gamma,\Gs)$ which are locally given by rational functions. The action of $\Gs$ on $\Hom(\Gamma,\Gs)$ by conjugation gives an action of $\Gs$ on $\Cb[\Hom(\Gamma,\Gs)]$, and we denote by $\Cb[\Hom(\Gamma,\Gs)]^{\Gs}$ the $\Gs$-invariant functions; for example, given an element $\gamma\in\Gamma$, the function $\rho\mapsto\Tr(\rho(\gamma))$ is in $\Cb[\Hom(\Gamma,\Gs)]^{\Gs}$. The \emph{geometric invariant theory quotient}, or GIT quotient, is defined as
$$\Hom(\Gamma,\Gs)/\!\!/\Gs:=\mathrm{Spec}(\Cb[\Hom(\Gamma,\Gs)]^{\Gs}),$$
the spectrum of prime ideals of this ring. When $\Gs$ is a reductive complex algebraic group (for example, when $\Gs=\GL_n\Cb,\SL_n\Cb$) the GIT quotient $\Hom(\Gamma,\Gs)/\!\!/\Gs$ and the polystable quotient $\Xf(\Gamma,\Gs)$ are homeomorphic \cite{flo-law}.

Let $\mathcal{P}$ be a collection of properties invariant by conjugation in the sense that if $\rho\in \mathrm{Hom}(\Gamma, \Gs)$ satisfies the properties $\mathcal{P}$, then $\iota_g\circ \rho$ also satisfies the properties $\mathcal{P}$ for all elements $g\in \Gs$.  In this case, the group $\Gs$ acts by conjugation on the subspace $$\mathrm{Hom}_\mathcal{P}(\Gamma, \Gs):=\{\rho\in \mathrm{Hom}(\Gamma, \Gs)\ |\ \rho \ \mathrm{ satisfies }\ \mathcal{P}\}.$$ 
We can define the polystable quotient as follows:
$$\mathfrak{X}_\mathcal{P} = \{\rho \in \mathrm{Hom}_\mathcal{P}(\Gamma, \Gs)|\ \mathrm{Orb}_{\Gs}(\rho)=\overline{\mathrm{Orb}_{\Gs}(\rho)}\} / \Gs$$
We call $\mathfrak{X}_\mathcal{P}(\Gamma, \Gs)$ the {\it $\mathcal{P}$-relative $\Gs$-character variety of $\Gamma$}.  If the condition $\mathcal{P}$ is vacuous, one gets $\mathfrak{X}(\Gamma, \Gs)$ back.

Let $\mathrm{Aut}_\mathcal{P}(\Gamma)$ denote the subgroup of $\mathrm{Aut}(\Gamma)$ that preserves $\mathrm{Hom}_\mathcal{P}(\Gamma, \Gs)$.  Then the inner automorphism group satisfies $\mathrm{Inn}(\Gamma)\subset \mathrm{Aut}_\mathcal{P}(\Gamma)$ and so the outer automorphism group $\mathrm{Out}_\mathcal{P}(\Gamma):=\mathrm{Aut}_\mathcal{P}(\Gamma)/\mathrm{Inn}(\Gamma)$ acts on $\mathfrak{X}_\mathcal{P}(\Gamma, \Gs)$ by $[\alpha]\cdot[\rho]=[\rho\circ \alpha^{-1}]$. We can see, again, that if $\mathcal{P}=\emptyset$, one obtains an action of $\mathrm{Out}(\Gamma)$ on $\mathfrak{X}(\Gamma, \Gs)$.

When the group $\Gs$ does not have the structure of a Lie group but only that of a topological group (for example, when $\Gs$ is the isometry group of a metric space, endowed with the compact-open topology), one may instead wish to consider the Hausdorff character variety, defined as the Hausdorffization of the quotient $\mathrm{Hom}(\Gamma,\Gs)/G$ of the space of homomorphisms from $\Gamma$ to $\Gs$, endowed with the compact-open topology, by the conjugation action of $\Gs$. The Hausdorffization of a topological space is the largest Hausdorff quotient of that space, it is obtained by identifying points that are identified in every Hausdorff quotient of the space (see Maret \cite{maret-character-variety}). Note that when the closure of every orbit contains a unique closed orbit, the polystable quotient space and the Hausdorff character variety coincide.

\subsection{Symplectic structure and invariant measure}\label{measure}

Many of the results we will discuss concern showing that the mapping class group of a surface $S$ acts ergodically on some character variety $\Xf(\pi_1(S),\Gs)$. The measure with respect to which this action is ergodic is constructed differently depending on the group $\Gs$ and the surface $S$; we briefly describe these constructions in this section.

We first consider the case where $S$ is an orientable surface and $\Gs$ is a Lie group that preserves a nondegenerate symmetric bilinear form $B$ on its Lie algebra $\gL$; for more details on this construction see \cite{goldman-symplectic}. A representation $\rho\in\Hom(\pi_1(S),\Gs)$ defines an action of $\pi_1(S)$ on $\gL$ via the adjoint representation $\Ad : \Gs \to \gL$ by $\gamma\cdot v=\Ad(\rho(\gamma))v$, making $\gL$ into a $\pi_1(S)$-module denoted $\gL_{\Ad\rho}$. The tangent space $T_\rho\Xf(\pi_1(S),\Gs)$ can then be identified with the cohomology group $H^1(\pi_1(S);\gL_{\Ad\rho})$ by checking that tangent directions to paths in $\Hom(\pi_1(S),G)$ correspond to cocycles in $Z^1(\pi_1(S);\gL_{\Ad\rho})$ and tangent directions to paths of the form $\rho_t=g_t^{-1}\rho g_t$ correspond to coboundaries in $B^1(\pi_1(S);\gL_{\Ad\rho})$. Using the pairing of cocycles via the cup product and pairing the coefficients using the bilinear form $B$, this gives a map
$$\omega_\rho^B:H^1(\pi_1(S);\gL_{\Ad\rho})\times H^1(\pi_1(S);\gL_{\Ad\rho})\to H^1(\pi_1(S);\Rb)\cong\Rb$$
at each $\rho\in\Xf(\pi_1(S),\Gs)$. The proofs that this map is nondegenerate and is a 2-form follow from the properties of $B$ and the cup product. To show that $\omega^B$ is nondegenerate, Goldman follows work of Atiyah--Bott \cite{atiyah-bott}, but instead considers $H^1(\pi_1(S);\gL_{\Ad\rho})$ as de Rham cohomology with coefficients in the flat vector bundle $S\times_\rho\gL:=(\widetilde{S}\times\gL)/\pi_1(S)$ where $\pi_1(S)$ acts by $\gamma\cdot(\tilde{x},v)=(\gamma\cdot\tilde{x},\Ad(\rho(\gamma))v)$. This construction gives a symplectic form\index{Symplectic form} $\omega^B$, and in turn a measure, on $\Xf(\pi_1(S),\Gs)$. 

The symplectic structure on $\Xf(\pi_1(S),\Gs)$ provides a way to study the dynamics of the mapping class group $\MCG(S)$ action on the character variety. This is done by relating the Dehn twist $\tau_\alpha$ along a simple closed curve $\alpha$ to some Hamiltonian flow, often called the Goldman twist flow \cite{goldman-flows}. Given a conjugation-invariant function $f:\Gs\to\Rb$, the associated 1-form $df$ is dual via $B$ to some vector field $F:\Gs\to\gL$ in the following sense:
$$B(v, F(x)) =\left.\frac{d}{dt}\right|_{t=0} f(x \exp{t v}),$$
for all $v \in \gL$ and $x \in \Gs$, and we will use the notation $df_x(v) := \left.\frac{d}{dt}\right|_{t=0} f(x \exp{t v})$ in the following.
For each $x\in\Gs$ we then define the one-parameter subgroup $\zeta^t(x)=\exp(tF(x))$ of $\Gs$. Note that $f$ is invariant under conjugation in $\Gs$, so for any $g\in\Gs$, $x \in G$ and $v\in\gL$ we have the following 
\begin{align*}
B\left(\Ad_gF(x),v\right) =B\left(F(x),\Ad_{g^{-1}}v\right)=df_x(\Ad_{g^{-1}}v)=df_{gxg^{-1}}(v)=B\left(F(gxg^{-1}),v\right).
\end{align*}
Since $B$ is nondegenerate, we then have $\Ad_g F(x)=F(gxg^{-1})$ and $F$ is $\Ad$--equivariant. In particular, $\Ad_xF(x)=F(xxx^{-1})=F(x)$ so $F(x)$ is $\Ad_x$ invariant, and therefore $\zeta^t(x)$ centralizes $x$.

We now wish to use this subgroup to define a flow on $\Xf(\pi_1(S),\Gs)$ that extends the action of the Dehn twist about a simple closed curve $\alpha\in\pi_1(S)$. We will consider two cases: first we consider the case that the curve $\alpha$ is nonseparating and then discuss the case of $\alpha$ being separating.

\textbf{Case 1:} Let $\alpha$ be a nonseparating curve. Recall that if $S\setminus\alpha$ denotes the surface obtained by cutting $S$ along $\alpha$ and $\alpha_{\pm}$ denote the corresponding boundary components of $S\setminus \alpha$, then $\pi_1(S)$ can be constructed as an HNN-extension as follows:
$$\pi_1(S)=\Big(\pi_1(S\setminus\alpha)\sqcup\ip{\beta}\Big)\Big/\Big(\beta\alpha_-\beta^{-1}=\alpha_+\Big).$$ The Dehn twist $\tau_\alpha$ acts on $\Hom(\pi_1(S),\Gs)$ by
$$(\tau_\alpha(\rho))(\gamma)= \begin{cases}
\rho(\gamma) & \text{if } \gamma\in\pi_1(S\setminus\alpha), \\
\rho(\gamma)\rho(\alpha_-)^{-1} & \text{if } \gamma=\beta.
\end{cases}$$

For any $\rho\in\Hom(\pi_1(S),\Gs)$ and any $t \in \Rb$ we have:
\begin{align*}
\Big(\rho&(\beta)\zeta^t(\rho(\alpha_-))^{-1}\Big)\rho(\alpha_-)\Big(\rho(\beta)\zeta^t(\rho(\alpha_-))^{-1}\Big)^{-1}\\ &=\rho(\beta)\Big(\zeta^t(\rho(\alpha_-))^{-1}\rho(\alpha_-)\zeta^t(\rho(\alpha_-))\Big)\rho(\beta)^{-1} \\
    &=\rho(\beta)\rho(\alpha_-)\rho(\beta)^{-1} \\
    &=\rho(\alpha_+),
\end{align*}
where the second equality follows because $\zeta^t(\rho(\alpha_-))$ centralizes $\rho(\alpha_-)$, while the third follows because $\rho$ is a representation of $\pi_1(S)$. This means that for any $t$ we can define a new representation $\xi^t_\alpha(\rho)$ by
$$(\xi^t_\alpha(\rho))(\gamma)= \begin{cases}
\rho(\gamma) & \text{if } \gamma\in\pi_1(S \setminus \alpha), \\
\rho(\gamma)\zeta^t(\rho(\alpha_-))^{-1} & \text{if } \gamma=\beta.
\end{cases}$$
If there is some $t$ such that $\zeta^t(\rho(\alpha_-))=\rho(\alpha_-)$, denote the smallest such positive $t$ by $s(\rho(\alpha_-))$. We then have that $\xi_\alpha^{s(\rho(\alpha_-))}(\rho)=\tau_\alpha(\rho)$, so the orbit of $\rho$ under the Dehn twist $\tau_\alpha$ is contained in the orbit of $\rho$ under the \emph{twist flow}  $\xi^t_\alpha$ on $\Hom(\pi_1(S),\Gs)$.

\textbf{Case 2:} Now consider a curve $\alpha$ that separates $S$ into two components $S_1$ and $S_2$. Then $\pi_1(S)=\pi_1(S_1)\sqcup_{\ip{\alpha}}\pi_1(S_2)$, and the Dehn twist $\tau_\alpha$ acts on $\Hom(\pi_1(S),\Gs)$ by
$$(\tau_\alpha(\rho))(\gamma)= \begin{cases}
\rho(\gamma) & \text{if } \gamma\in\pi_1(S_1), \\
\rho(\alpha)^{-1}\rho(\gamma)\rho(\alpha) & \text{if } \gamma\in\pi_1(S_2).
\end{cases}$$

In this case the twist flow is defined as follows:
$$(\xi^t_\alpha(\rho))(\gamma)= \begin{cases}
\rho(\gamma) & \text{if } \gamma\in\pi_1(S_1), \\
\zeta^t(\rho(\alpha))^{-1}\rho(\gamma)\zeta^t(\rho(\alpha)) & \text{if } \gamma\in\pi_1(S_2).
\end{cases}$$
We then have, again, that $\xi_\alpha^{s(\rho(\alpha_-))}(\rho)=\tau_\alpha(\rho)$.

Recall that we initially defined the one-parameter subgroup $\zeta^t(x)$ of $\Gs$ using a conjugation-invariant function $f:\Gs\to\Rb$. In both the separating and nonseparating cases the twist flow descends to the Hamiltonian flow on $\Xf(\pi_1(S),\Gs)$ corresponding to the function $f_\alpha:\Xf(\pi_1(S),\Gs)\to\Rb$ given by $f_\alpha(\rho)=f(\rho(\alpha))$.

While character varieties of orientable surfaces are known to be symplectic manifolds, the same is not known for non-orientable surfaces. In fact, these spaces are often known to \textbf{not} be symplectic manifolds. Take, for example, the Teichm{\"u}ller component of $\Xf(\pi_1(N_g),\PGLR)$; this is homeomorphic to $\Rb^{-3\chi(N_g)}=\Rb^{-3(2-g)}$, so if $g$ is odd then this space has odd dimension, and therefore cannot be symplectic. In the case of a non-orientable surface $N$ and a compact Lie group $\Gs$, we can think of $\Hom(\pi_1(N),\Gs)$ as a closed subspace of $\Gs^k$ where $k$ is the number of generators of $\pi_1(N)$, meaning that $\Hom(\pi_1(N),\Gs)$ is compact. Together with the fact that Dehn twists act by multiplication by elements of $\Gs$, one might expect that there is a finite, mapping class group-invariant measure on $\Hom(\pi_1(N),\Gs)$ related to the Haar measure on $\Gs^k$. Following work  of Witten \cite{witten-quant}, Ho--Jeffrey \cite{HJ}, Mulase--Penkava \cite{Mu-Pen}, Palesi \cite{pal-erg} constructs such a measure in the case of a closed $N_g$ with $g\geq4$ and $\Gs=\SU(2)$. One can show that the generalization of this construction for other compact $\Gs$ reduces to showing the convergence of a sum involving the dimensions of the complex irreducible representations of $\Gs$. 

\subsection{Euler class}\label{sec:euler}

In this section we give an overview of the \emph{Euler class}\index{Euler class} of a representation $\rho:\pi_1(S)\to\PSLR$, which is a crucial tool used in characterizing the connected components of the character variety $\Xf(\pi_1(S),\PSLR)$. For a more detailed discussion, see \cite{goldman-topological}. For a closed, orientable surface $S$, a representation $\rho:\pi_1(S)\to\PSLR$ gives an action of $\pi_1(S)$ on $\mathbb{RP}^1$, and $\pi_1(S)$ naturally acts on $\widetilde{S}$. We define a flat circle bundle over $S$ with total space $S\times_\rho\mathbb{RP}^1:=(\widetilde{S}\times\mathbb{RP}^1)/\pi_1(S)$. The Euler class $\eu(\rho)$ is a characteristic class of this bundle that measures the obstruction to finding a global section of the circle bundle. The Euler class is generally thought of as an integer using the isomorphism between the second cohomology of this bundle and $\Zb$.

If $S$ is instead a punctured surface or a surface with non-empty boundary, then $H^2(S,\Zb)=0$, so the bundle $S\times_\rho\mathbb{RP}^1$ is trivial and admits a global section, independent of $\rho$. Similarly, $\pi_1(S)$ is a free group, so $\Hom(\pi_1(S),\PSLR)=\PSLR^k$, where $k$ is the number of generators of $S$, and is already connected. For these reasons, we must add some constraints in order to define a useful notion of the Euler class for non-closed surfaces. To do this, let $\{c_1,\cdots,c_n\}$ be a set of peripheral simple closed curves on the surface, one for each boundary component or puncture, and consider a representation $\rho$ such that $\rho(c_i)$ is either parabolic or hyperbolic for each $i$. The representation $\rho|_{\partial S}$ has a preferred trivialization, described in more detail later in this section, and the \emph{relative Euler class} of $\rho$ is a class that measures the obstruction to extending this trivialization to the entire surface. 

This definition in terms of cohomology classes is powerful in proving many of the abstract properties of the Euler class, but is not easily computable and so is not the definition used in much of the literature. We now give two alternate definitions that lend themselves better to computation.

The first definition considers lifts of elements of $\PSLR$ to its universal cover $\cPSLR$, see again \cite{goldman-topological}. Let $\rho:\pi_1(S)\to\PSLR$ be a representation of $\pi_1(S)=\ip{a_1,b_1,\dots,a_g,b_g\ |\ \prod_{i=1}^g[a_i,b_i]}$. Given lifts $\widetilde{\rho(a_i)},\widetilde{\rho(b_i)}\in\cPSLR$ of the generators, the product $\prod_{i=1}^g\left[\widetilde{\rho(a_i)},\widetilde{\rho(b_i)}\right]$ projects to the identity in $\PSLR$, meaning that the product is a lift of the identity. Noting that the fundamental group of $\PSLR$ is $\Zb$, the lifts of the identity are $\{z^n\ |\ n\in\Zb\}$ for some lift $z\in\cPSLR$. This then means that $\prod_{i=1}^g[\widetilde{a_i},\widetilde{b_i}]=z^n$ for some $n\in\Zb$. Moreover, for any $h\in\PSLR$, two lifts $\widetilde{h}_1$ and $\widetilde{h}_2$ of $h$ only differ by a lift of the identity, i.e. there is some $n\in\Zb$ such that $\widetilde{h}_2=z^n\widetilde{h}_1$. Because $z$ is central in $\cPSLR$, this means that the product $\prod_{i=1}^g\left[\widetilde{\rho(a_i)},\widetilde{\rho(b_i)}\right]$, being composed of commutators, is independent of the choices of $\widetilde{\rho(a_i)},\widetilde{\rho(b_i)}$, and therefore so is the integer $n$; this integer $n$ is the Euler class of the representation $\rho$.

If $S$ is not closed, we again restrict ourselves to the representations that send the peripheral curves $\{c_1,\cdots,c_n\}$ to parabolic or hyperbolic elements of $\PSLR$. In $\pi_1(S)$ we have the relation $\left(\prod_{i=1}^g[a_i,b_i]\right)\prod_{j=1}^nc_j$ but two different lifts of any of the $\rho(c_j)$ differ by some power of $z$, so a product $\left(\prod_{i=1}^g[\widetilde{\rho(a_i)},\widetilde{\rho(b_i)}]\right)\prod_{j=1}^n\widetilde{\rho(c_j)}$ is not independent of the choices of lifts of the $\rho(c_j)$. We instead need to find a preferred choice of lift for these peripheral curves. If $g\in\PSLR$ is a parabolic or hyperbolic then it has a fixed point in $\mathbb{RP}^1$, meaning that its translation length when acting on $\mathbb{RP}^1$ is 0 and that the translation length of any lift $\widetilde{g}\in\cPSLR$, seen as acting on $\Rb=\widetilde{\mathbb{RP}^1}$, is an integer. We then define the preferred lift of $g$ to be the lift with zero translation length. This choice makes the product $\left(\prod_{i=1}^g[\widetilde{\rho(a_i)},\widetilde{\rho(b_i)}]\right)\prod_{j=1}^n\widetilde{\rho(c_j)}=z^n$ well-defined and again the relative Euler class of the representation is taken to be $\eu(\rho)=n$.

The second definition is more geometric in nature and relates to what is known as the \emph{volume} of a representation. Recall that the data of a (marked) hyperbolic structure on $S$ is equivalent to the data of a (discrete, faithful) representation $\rho:\pi_1(S)\to\PSLR$ and a $\rho$-equivariant homeomorphism $D_\rho:\widetilde{S}\to\Hb^2$ called a \emph{developing map}. Generalizing this idea, given any representation $\rho:\pi_1(S)\to\PSLR$, a \emph{pseudo-developing map} associated to $\rho$ is a piecewise-smooth, $\rho$-equivariant $D_\rho:\widetilde{S}\to\Hb^2$, which is not necessarily a homeomorphism. The Euler class of the representation can then be defined as 
$$\eu(\rho)=\dfrac{1}{2\pi}\int_{S}(D_\rho)^*\omega,$$
where $\omega$ is the area form on $\Hb^2$. Using hyperbolic geometry, one can show that $\eu(\rho) \in \Zb$. Note that if $\rho$ is discrete and faithful, then $\int_{S}(D_\rho)^*\omega$ is the area of the hyperbolic surface $S$, and the Gauss--Bonnet theorem implies that $\eu(\rho)=\pm\chi(S)$. If $S$ is not closed and the peripheral curves $\{c_1,\cdots,c_n\}$ are sent to parabolic elements of $\PSLR$ then this definition can be used without modification because the hyperbolic area of $S$ is still finite. However, if any of the peripheral curves around punctures are sent to hyperbolic elements then a fundamental domain for $S$ in $\Hb^2$ has infinite hyperbolic area. We can instead integrate over the same fundamental domain but with the `funnel' removed; this domain has a finite hyperbolic area and the integral above gives the Euler class of the representation. In practice this integral is sometimes computed by taking an ideal triangulation of the surface $S$ and splitting the integral over the fundamental domain of $S$ into an integral over each of these ideal triangles, see, for example, \cite{yan_ont} and \cite{MPY}.

Similar obstruction classes exist for studying representations $\rho:\pi_1(N)\to\PGLR$ of a nonorientable surface group, though the situation is more delicate in this case; see \cite{palesi-no} for a detailed discussion. For example, consider the definition using lifts of the generators of $\pi_1(N)$ to $\cPSLR$. The fundamental group $\pi_1(N_{g,k})$ has a presentation of the form $\ip{a_1,\dots,a_g,c_1,\dots,c_k\ |\ \prod_{i=1}^ga_i^2\prod_{j=1}^kc_j}$. Given a representation $\rho:\pi_1(N_{g,k})\to\PSLR$ that sends each element to an orientation-preserving isometry, choosing two different lifts of any of the $\rho(a_i)$ results in changing the exponent on $z$ by a multiple of 2, so this construction is only well-defined if we alter the definition so that it takes values in $\Ztwo$. Now consider the definition via integrating the pullback of the area form on $\Hb^2$. Let $\widehat{N}$ be the orientable double-cover of $N$. A representation $\rho:\pi_1(N)\to\PGLR$ determines a representation $\widehat{\rho}:\pi_1\widehat{N}\to\PSLR$; using this we can define $\eu(\rho)=\dfrac{1}{2}\eu(\widehat{\rho})$, yielding an integer rather than an element of $\Ztwo$. Another obstruction class useful in studying the $\PGLR$-character variety of a non-orientable surface is the \emph{first obstruction class} $o_1$, defined as follows. The group $\PGLR$ has two connected components, i.e. $\pi_0(\PGLR)\cong\Ztwo$, one of which is a copy of $\PSLR$ and the other of which consists of orientation-reversing isometries of $\Hb^2$. Given a representation $\rho:\pi_1(N_g)\to\PGLR$, for each one-sided generator $a_i,\ i=1,\dots,g$ of $\pi_1(N_g)$, we can assign to $\rho(a_i)$ either a $0$ if $\rho(a_i)$ is orientation-preserving or a $1$ if $\rho(a_i)$ is orientation reversing. This assignment is invariant under conjugation by $\PGLR$, so it defines a map $o_1:\Xf(\pi_1(N_g),\PGLR)\to(\Ztwo)^g$ which is constant on connected components of the character variety.

\section{Representations in \texorpdfstring{$\PSLR$}{PSL(2,R)} and \texorpdfstring{$\PGLR$}{PGL(2,R)}}\label{PSL-R}

In this section, we will focus on representations from a discrete group $\Gamma$ into the groups $\mathsf{PSL}_2(\Rb)$ and $\mathsf{PGL}_2(\Rb)$. We will start by describing, in Section \ref{o-surface-PSL-R}, the case of representations from surface groups (that is, $\Gamma = \pi_1(S)$ for $S$ a closed orientable surface with genus $g \geq 2$), and then explain how the discussion in that section motivates the discussion in Section \ref{no-surface-PSL-R} for fundamental groups of non-orientable surfaces and in Section \ref{type-preserving} for fundamental groups of surfaces with punctures and type-preserving representations. We will discuss results concerning the topology of the space of such representations, the geometric properties of these representations, the dynamics of the action of $\Out(\pi_1(S))$ on this space, and generalizations in the cases when the surface is non-orientable or has punctures.

\subsection{Closed orientable surfaces}\label{o-surface-PSL-R}

\subsubsection{Connected components} 

The space $\mathfrak{X}(\pi_1(S_g),\PSLR)$ contains two connected components consisting of discrete, faithful representations. The representations in these components correspond to (marked) hyperbolic structures on $S_g$ with opposite orientations; equivalently, any orientation-reversing $h\in\PGLR$ induces 
a homeomorphism of $\mathfrak{X}(\pi_1(S),\PSLR)$ given by $\rho\mapsto\iota_h\circ\rho$ that swaps these two components. These components correspond to the \emph{Teichm\"uller space of $S_g$}, denoted $\Teich(S_g)$ (or sometimes $\Teich(S_g)$ and $\Teich(\conj{S_g})$ to distinguish the two components) and are each homeomorphic to $\Rb^{6g-6}=\Rb^{-3\chi(S_g)}$. 

The character variety $\mathfrak{X}(\pi_1(S_g),\PSLR)$ contains finitely many connected components which Milnor \cite{mil_ont}, Wood \cite{wood_bundles} and Goldman \cite{goldman-topological} showed are classified by the \emph{Euler class} of the representation, denoted $\eu(\rho)$, see Section \ref{sec:euler}.

\begin{Theorem}[Goldman \cite{goldman-topological}] \label{mwg_cc}
    Let $S_g$ be a closed, orientable surface with $\chi(S_g)<0$. The space $\mathfrak{X}(\pi_1(S_g),\PSLR)$ has $2\abs{\chi(S_g)}+1$ connected components, indexed by the Euler class $\chi(S_g)\leq\eu(\rho)\leq-\chi(S_g)$. Moreover, $\rho$ is discrete and faithful if and only if it is in the maximal components, that is, if $\abs{\eu(\rho)}=\abs{\chi(S_g)}$.
\end{Theorem}

The proof of this theorem goes by induction on the Euler characteristic $\chi(S_g)$. The surface $S_g$ can be decomposed into $g$ one-holed tori and $g-2$ pairs of pants, which are the two orientable surfaces with Euler characteristic $-1$. We wish to prove the equivalent theorem for these two surfaces and then show that we can decompose $S_g$ in such a way that the Euler class of the representation of $\pi_1(S_g)$ is the sum of the Euler classes of the representations obtained by restricting to these subsurfaces. To state the theorem for surfaces with boundary, recall the notation from Section \ref{sec:top} where in the surface $S_{g,n}$ we denoted $c_1, \ldots, c_n$ the peripheral curves. We introduce the set
$$W(S_{g,n})=\{\rho\in\Hom(\pi_1(S_{g, n}),\PSLR)\ |\ \rho(c_i)\ \text{ is not elliptic for }\ i=1,\dots,n\}.$$
We then have the following result.
\begin{Theorem}[Goldman \cite{goldman-topological}] \label{mwg_ccb}
    Let $S_{g,n}$ be an orientable surface with boundary such that $\chi(S_{g,n})<0$. The space $W(S_{g,n})$ has $2\abs{\chi(S_{g,n})}+1$ connected components, indexed by the relative Euler class $\chi(S_{g,n})\leq\eu(\rho)\leq-\chi(S_{g,n})$. Moreover, $\rho \in W(S_{g,n})$ is discrete and faithful if and only if it is in the maximal components, that is, if $\abs{\eu(\rho)}=\abs{\chi(S_{g,n})}$.
\end{Theorem}

To show this, Goldman first proves the case where $S_{g,n}$ is either a one-holed torus or a pair of pants by explicitly computing what the commutator of two elements in $\cPSLR$ can be (for the one-holed torus) or what the product of two elements in $\cPSLR$ can be (for the pair of pants) to determine what the Euler class of such a representation can be. This shows that $\eu(\rho)\in\{-1,0,1\}$ for all such $\rho$. Moreover, Goldman shows that if two such representations have the same Euler class then there is a continuous path between them in $W(S_{1,1})$ or $W(S_{0,3})$.

Next we decompose $S_{g,n}$ into surfaces of characteristic $-1$ such that the dual graph to the decomposition is a tree; this condition is necessary for the inductive step that Goldman shows. Such a decomposition of $S_{g,n}$ contains $g$ one-holed tori and $g+n-2$ pairs of pants. It is possible to find a continuous path in $W(S_{g,n})$ from a given representation $\rho$ to a representation $\rho'$ such that the curves defining this decomposition are sent to hyperbolic elements. This means that for each subsurface $S$ in the decomposition, the restriction $\rho'|_{S}$ is in $W(S)$, and we can compute the relative Euler class of $\rho'|_{S}$. The results for $S_{1,1}$ and $S_{0,3}$ then show that if two representations have the same relative Euler class on each subsurface $S$, then they are in the same connected component of $W(S_{g,n})$.

Finally, we must show that if two representations in $W(S_{g,n})$ have the same Euler class, then there are paths of representations that have the same relative Euler class on each of the subsurfaces. This can be done first by considering each of the surfaces of Euler characteristic $-2$, i.e. $S_{2}$, $S_{1,2}$, and $S_{0,4}$. Inducting along the tree dual to the decomposition gives the final result.

\subsubsection{Dynamics of the mapping class group and Bowditch's question}

The outer automorphism group $\Out(\pi_1(S))$ is (virtually) isomorphic to $\MCG(S)$, the mapping class group of the surface, and acts on the space $\mathfrak{X}(\pi_1(S),\PSLR)$ by precomposition. A result often attributed to Fricke showed that the action on the Teichm\"uller components $\Teich(S)$ and $\Teich(\conj{S})$ is properly discontinuous. (The interested reader can see a proof of this result in Canary \cite{can_dyn} or Farb-Margalit \cite[Theorem 12.2]{FM}.) A consequence is that the action has finite point stabilizers, so the quotient $\Teich(S)/\MCG(S)$ is an orbifold called the moduli space of $S$.

Goldman conjectured in \cite{goldman-survey} that the action of $\MCG(S)$ on the non-maximal components of $\mathfrak{X}(\pi_1(S),\PSLR)$ should be ergodic with respect to the symplectic measure on $\mathfrak{X}(\pi_1(S),\PSLR)$ defined in Section \ref{back} (except for the $\eu(\rho)=0$ component for technical reasons which we discuss below), in stark contrast to the proper discontinuity of the action on the maximal components. This question remains open in general, but March\'e and Wolff \cite{mar-wol, mar-wol2} showed that it is true in the genus--$2$ case. They also solidified the equivalence between Goldman's conjecture and the following question of Bowditch \cite{bow_mar}, namely: \textit{Does almost every non-maximal representation $\rho\in\mathfrak{X}(\pi_1(S),\PSLR)$ send a non-trivial simple closed curve to an elliptic element of $\PSLR$?} They proved that a positive answer to Bowditch's question implies a positive answer to Goldman's conjecture \cite{mar-wol} and in the case of genus $g=2$ they answered Bowditch's question in the positive. 

In order to state their result more formally, we introduce the notation $\Xf^{\text{ne}}_k(S_2)$ for the connected component of $\mathfrak{X}(\pi_1(S_2),\PSLR)$ containing \textit{non-elementary} representations with $\eu(\rho)=k$, that is  representations with Zariski dense image in $\PSLR$. The result below was proven in two papers.
\begin{Theorem}[March\'e--Wolff \cite{mar-wol, mar-wol2}] \label{mw_ergodic}
    Let $S_2$ be the closed, oriented surface of genus 2. Then:
    \begin{enumerate}
        \item The action of $\MCG(S_2)$ on $\Xf^{\text{ne}}_k(S_2)$ is ergodic if $|k|=1$.
        \item The space $\Xf^{\text{ne}}_0(S_2)$ is the union of two disjoint, $\MCG(S_2)$-invariant open subsets, denoted $\Xf^{\text{ne}}_{0,+}(S_2)$ and $\Xf^{\text{ne}}_{0,-}(S_2)$.
        \item The action of $\MCG(S_2)$ on each of $\Xf^{\text{ne}}_{0,+}(S_2)$ and $\Xf^{\text{ne}}_{0,-}(S_2)$ is ergodic.
    \end{enumerate}
\end{Theorem}
Note that the action is still `nearly' ergodic even in $\Xf^{\text{ne}}_{0}(S_2)$. In fact, the space $\Xf^{\text{ne}}_{0}(S_2)$ consists of two disjoint open sets and the action is ergodic on each one. March\'e and Wolff also showed that this decomposition of the component with euler class zero is unique to the genus $2$ case, see Theorem 1.6 in \cite{mar-wol}.

Let $S_g$ denote the genus $g$ closed surface and $\Xf^{\text{ne,NH}}_{k}(S_g)$  denote the subset of $\Xf^{\text{ne}}_{k}(S_g)$ consisting of representations that send some non-trivial simple closed curve to a non-hyperbolic element of $\PSLR$. March\'e and Wolff showed the following: 
\begin{Theorem}[March\'e-Wolff \cite{mar-wol}] \label{mw_bowditch}
    Let $g\geq 2$ and let $\chi(S_g)+1\leq k\leq -\chi(S_g)-1$. If $(g,k)\neq(2,0)$ then the action of $\MCG(S_g)$ on $\Xf^{\text{ne,NH}}_{k}(S_g)$ is ergodic.
\end{Theorem}
Note that Bouilly \cite[Theorem C]{bouilly} proved an analogue of this result for punctured surfaces under certain conditions. 

This shows that Goldman's conjecture would be solved if one could find a positive answer to Bowditch's question and show that $\Xf^{\text{ne,NH}}_{k}(S_g)$ is of full measure in $\Xf^{\text{ne}}_{k}(S_g)$. Note that we originally phrased Bowditch's question as looking for a non-trivial simple closed curve whose image is elliptic. In fact, roughly speaking, this is equivalent to looking for when the image is non-hyperbolic because the parabolic elements have measure zero in $\PSLR$. In addition, this is further equivalent to looking for when the image is an elliptic element with irrational rotation number because the elliptic elements with rational rotation number have measure zero in the full set of elliptic elements.

\subsubsection{Strict domination of representations}

We can consider the function $l \co \PSLR \to \Rb_{\geq 0}$ defined by
$$ l(g) := \mathrm{inf}_{ x \in \mathbb{H}^2} d(x, g\cdot x),$$
that is the translation length of $g \in \PSLR$ in the hyperbolic plane $\mathbb{H}^2$. Note that this function $l$ is invariant under conjugation. 

Let $S = S_g$ be any closed, connected, oriented surface of genus $g \geq 2$, and let $\rho_1, \rho_2 \in \mathfrak{X}(\pi_1(S), \PSLR)$. We say that $\rho_1$ \textit{strictly dominates} $\rho_2$ if $$ \mathrm{sup}_{\gamma \in \pi_1(S) \setminus \{1\}} \frac{l\left(\rho_2(\gamma)\right)}{l\left(\rho_1(\gamma)\right)} < 1.$$

Note that Thurston \cite{T-min} proves that there are no Fuchsian representations~$\rho_1, \rho_2$ where $\rho_1$ strictly dominates $\rho_2$. 

Gueritaud--Kassel--Wolff \cite{GKW} and Deroin--Tholozan \cite{DT-domi} independently proved the following result using different methods: Kassel--Gueritaud--Wolff used a more constructive proof using folded (or pleated) surfaces, while Deroin--Tholozan used a more analytic proof. In fact, Deroin--Tholozan have a more general statement about surface group representations into the isometry group of any complete, simply connected Riemannian manifold with sectional curvature bounded above by $-1$. See also Martin--Baillon for a more general related result into CAT(-1) metric spaces \cite{MB-domination}. Deroin--Tholozan also used this result to obtain various corollaries in other contexts, such as for closed anti-de Sitter $3$--manifolds. Let $\mathfrak{X}_{DF}(\pi_1(S), \PSLR)$ be the set of (conjugacy classes of) discrete and faithful representations and let $\mathfrak{X}_{NDF}(\pi_1(S), \PSLR)$ be its complement, that is $\mathfrak{X}_{NDF}(\pi_1(S), \PSLR) = \mathfrak{X}(\pi_1(S), \PSLR)\setminus \mathfrak{X}_{DF}(\pi_1(S), \PSLR).$ 

\begin{theorem}[Gueritaud--Kassel--Wolff \cite{GKW}, Deroin--Tholozan \cite{DT-domi}]
    Let $S = S_g$ be any closed, connected, oriented surface of genus $g \geq 2$. Any representation in $\mathfrak{X}_{NDF}(\pi_1(S), \PSLR)$ is strictly dominated by some representation in $\mathfrak{X}_{DF}(\pi_1(S), \PSLR)$. In addition, any representation $\mathfrak{X}_{DF}(\pi_1(S), \PSLR)$ strictly dominates some representation in $\mathfrak{X}_{NDF}(\pi_1(S), \PSLR)$ whose Euler class can be prescribed.
\end{theorem}
Note that the last statement of the theorem above is only due to Gueritaud--Kassel--Wolff.

\subsection{Non-orientable surfaces}\label{no-surface-PSL-R}

Let $N = N_g$ be the non-orientable closed surface of genus $g$, see Section \ref{sec:top} for a discussion of the topology of non-orientable surfaces.

While in the case of orientable surfaces we considered representations into $\Isom^+(\Hb^2)\cong\PSLR$, in this section we consider representations from the group $\pi_1(N_g)$ into $\Isom(\Hb^2)\cong\PGLR$. The group $\PGLR$ contains $\PSLR$ as a subgroup and is homeomorphic to two copies of $\PSLR$. We will denote the identity component of $\PGLR$ as $\Gs_+$ and the other component as $\Gs_-$. The component $\Gs_+$ is exactly $\PSLR$, while $\Gs_-$ consists of the orientation-reversing isometries of $\Hb^2$. The elements of $\Gs_-$ are all glide-reflections in $\Hb^2$, that is, they consist of a hyperbolic translation (possibly with zero translation length) along a geodesic composed with a reflection about that geodesic. This means that the square of an element in $\Gs_-$ is either a hyperbolic element of $\Gs_+$ or is the identity.

A hyperbolic structure on $N_g$ is a $(\PGLR,\Hb^2)$--structure, that is, a maximal atlas $\{(U_\alpha,\phi_\alpha)\}_{\alpha\in\mathcal{A}}$ of charts $\phi_\alpha:U_\alpha\to\Hb^2$ on $N_g$ such that for each connected component $C\subset U_\alpha\cap U_\beta$ the transition map $(\phi_\alpha\circ\phi_\beta^{-1})|_C:\phi_\beta(C)\to\phi_\alpha(C)$ is the restriction of an element of $\PGLR$. Similarly, a complex structure on $N_g$ is a maximal atlas of charts to $\Cb$ such that the transition maps are either holomorphic or anti-holomorphic; a surface with such a structure is sometimes called a Klein surface.

With these definitions, given a closed, non-orientable surface $N$ of genus $g\geq 3$, the following data are equivalent:
\begin{itemize} 
    \item a discrete, faithful representation $\rho:\pi_1(N)\to\PGLR$;
    \item a (marked) hyperbolic structure on $N$;
    \item a (marked) complex structure on $N$.
\end{itemize}
See for example Palesi \cite{palesi-no} for a more detailed discussion of this case. 

Note that in order to define a hyperbolic structure on $N$, such a representation must send one-sided curves to elements of $\Gs_-$. The space $\Xf(\pi_1(N),\PGLR)$ contains one connected component consisting of discrete, faithful representations which is homeomorphic to $\Rb^{-3\chi(N)}$. As in the orientable case, the action of the mapping class group $\MCG(N)$ on this component is properly discontinuous. Palesi \cite{palesi-no} showed that $\Xf(\pi_1(N),\PGLR)$ contains finitely many connected components:
\begin{Theorem} [Palesi  \cite{palesi-no}] \label{palesi-cc-pgl}
    The representation space $\Hom(\pi_1(N_k),\PGLR)$ has $2^{k+1}+2k-5$ connected components. As a consequence, the character variety $\X(\pi_1(N_k),\PGLR)$ has $2^{k+1}+k-3$ connected components.
\end{Theorem}

Although the count is a bit more complicated than in the orientable case, it still comes down to considering characteristic classes associated to the representations, along with considering whether the one-sided generators of $\pi_1(N_k)$ are sent into $\Gs_+$ or $\Gs_-$. We can rewrite the quantity in the above theorem as follows:
$$2^{k+1}+2k-5 \;\;=\;\; 2(2^k-1)+2k-3\;\;=\;\;2\sum_{l=1}^k\begin{pmatrix}
    k \\ l
\end{pmatrix}+2\abs{\chi(S_{0,k})}+1$$ 
in order to better see how we arrive at this count. Fixing a choice of $l>0$ one-sided generators of $\pi_1(N_k)$ that are sent into $\PSLR$, given such a $\rho$ we can cut along the remaining one-sided generators to get a representation $\rho':\pi_1(N_{l,k-l})\to\PSLR$. Palesi showed in \cite{palesi-psl} that there are two components of the character variety $\Xf(\pi_1(N_{l,k-l}),\PSLR)$, which results in the $2\sum\limits_{l=1}^k\begin{pmatrix}
    k \\ l
\end{pmatrix}$ term above. When all of the one-sided generators are sent into $\Gs_-$, i.e. when $l=0$, the representation induces a representation of the subgroup $\pi_1(S_{0,k})<\pi_1(N_k)$, generated by the squares of the generators, so we have a representation of an orientable surface with $k$ boundary components into $\PSLR$. Moreover, each one-sided generator was mapped into $\Gs_-$, so the square of its image is either hyperbolic or the identity, and therefore the induced representation of $\pi_1(S_{0,k})$ has hyperbolic or identity boundary conditions. Milnor \cite{mil_ont}, Wood \cite{wood_bundles} and Goldman's \cite{goldman-topological} work then shows that there are $2\abs{\chi(S_{0,k})}+1$ such components and are indexed by the Euler class. Finally, conjugation by an element of $\PGLR$ preserves the components $\Gs_+$ and $\Gs_-$ of $\PSLR$ and sends the Euler class to its opposite. Therefore, the $2k-4$ components corresponding to the case $l=0$ and having non-zero Euler class are identified two by two, yielding the count of the number of connected components of $\X(\pi_1(N_k),\PGLR)$.

Note that in the discussion above it seems that orientable and non-orientable surfaces behave `similarly'. On the other hand, this is not always the case, as one can see, for example, by considering the question of the growth of simple closed curves. A well-known result of Mirzakhani \cite{mir_gro} shows that for an orientable surface $S$ the number of closed geodesics of a given topological type of length less than $L$ is asymptotically equivalent to a positive constant times $L^{\mathrm{dim}(\ML(S))}$, where we denote $\ML(S)$ the space of measured laminations of $S.$ In \cite{hua_sim, mag_cou, EGPS} it was proved that this is not true any longer for non-orientable surfaces. In fact, as Gendulphe \cite{gen_wha} explains, the difference comes from the different dynamics of the action of $\mathrm{Mod}(S)$ on $\ML(S)$. 

\subsection{Punctured surfaces}\label{type-preserving}

If we consider a punctured orientable  surface $S$, the character variety $\Xf(\pi_1(S),\PSLR)$ is connected and homeomorphic to $\PSLR^{-\chi(S)}$, where $\chi(S)$ is the Euler characteristic of $S$, because $\pi_1(S)$ is a (non-Abelian) free group of rank $1-\chi(S)$. On the other hand, there are many distinct punctured surfaces with isomorphic fundamental group. Hence, in order to `distinguish' between different surfaces, it is natural to consider the subspaces of $\Xf(\pi_1(S),\PGLR)$ with suitable constraints on the peripheral curves. Goldman \cite{goldman-topological} studied representations such that all the peripheral curves are sent to hyperbolic elements, and used this study to discuss the case of closed hyperbolizable surfaces. Another natural constraint is to ask all peripheral curves to be sent to parabolic elements. This is the notion we will focus on in this section.

\begin{definition}
A representation $\rho\co\pi_1(S) \to \PGL_2(\Rb)$ is called \textit{type-preserving} if:
\begin{itemize}
\item all the peripheral elements are mapped to parabolic isometries;
\item all $1$-sided [resp. $2$–sided] elements are mapped to orientation reversing [resp. preserving] isometries. 
\end{itemize}
\end{definition}

Recall that an element of $\pi_1(S)$ is called $2$-sided if it is represented by a curve which admits an orientable regular neighborhood, and $1$-sided otherwise. For an orientable surface, all curves are $2$-sided, so type-preserving representations are representations in $\PSLR$.

In this section we will consider the space $$\mathfrak{X}_{tp}(S) = \{ \mbox{type-preserving } \rho \co \pi_1 (S) \to \PGLR \} / \PSLR$$ of type-preserving representations up to $\PSLR$--conjugation. One can consider this as a subset of the $\PGLR$--character variety of $S$ by considering a further quotient which identifies certain connected components.  The reason to consider representations up to $\PSLR$ is because this allows a unified treatment with the well known case of representations of orientable punctured surfaces in $\PSLR$. In particular, the notion of euler classes and the lengths coordinates  extend naturally from $\PSLR$--character varieties to $\mathfrak{X}_{tp}(S).$ We are interested in three topics:
\begin{itemize}
\item Kashaev conjecture on the number of connected components of $\mathfrak{X}_{tp}(S)$;
\item Bowditch question on the existence of representations such that the image of all $2$-sided simple closed curves is hyperbolic but which are not discrete and faithful; and
\item Goldman conjecture on the ergodicity of the action of the mapping class group $\mathrm{Mod}(S)$ on $\mathfrak{X}.$
\end{itemize}
We will discuss some (possibly partial) answers to these questions. 

\subsubsection{Number of connected components}\label{conn-comp}

For each type-preserving representation $\rho$, one can define its Euler class $\eu(\rho)$ as its representation area divided by $2\pi$, see Section \ref{sec:euler} for a more detailed discussion. It is known that the Euler class satisfies the Milnor--Wood inequality $\chi(S)\leqslant \eu(\rho)\leqslant -\chi(S)$, see \cite{mil_ont, wood_bundles}, and Goldman\,\cite{goldman-topological} (in the orientable case) and Palesi\,\cite{palesi-no} (in the non-orientable case) proved that the equality holds if and only if $\rho$ is Fuchsian, that is, the representation is discrete and faithful. The case of closed orientable and non-orientable surfaces is discussed in the previous sections. For a punctured orientable surface $S_{g,n},$ the number of connected components of $\mathfrak{X}(S_{g,n})$ requires a more refined description, since for an integer $e$ with $|e|\leqslant -\chi(S),$ the spaces $\mathfrak X_e(S)$ of (conjugacy classes of) type-preserving representations of Euler class $e$ can either be empty or non-connected, see \cite{kas_coor}. 

In the case of orientable surfaces,  Kashaev\,\cite{kas_coor} conjectured that the number of connected components of the character variety $\mathfrak{X}_{tp}(S_{g,n})$ is determined by the Euler class and an extra invariant, the `sign of the punctures', which corresponds to the $\PSLR$--conjugacy classes of the holonomy representations of the peripheral elements. More precisely, for a type-preserving representation $\rho\co\pi_1 (S) \rightarrow \PGLR,$ we say that the \emph{sign} of a puncture $v$ is  \emph{positive} (resp. \emph{negative}), denoted by $s(v)=+1$ (resp. $s(v)=-1$), if $\rho$ sends a peripheral element around this puncture into a parabolic element conjugated to an upper-triangular matrix with the nondiagonal element positive (resp. negative). For $s\in\{\pm 1\}^n,$ we denote by $\mathfrak{X}_{e, s}(S)$ the space of conjugacy classes of type-preserving representations with Euler class $e$ and signs of the punctures given by $s$. Kashaev \cite{kas_coor} conjectured that each $\mathfrak{X}_{e, s}(S)$ is either empty or connected. 

Yang \cite{yan_ont} described the case of the four holed sphere $S_{0,4}$ and Maloni--Palesi--Yang \cite{MPY} described the case of the thrice-punctured projective plane $N_{1,3}$. The main tool used for these proofs are the length coordinates for the decorated nonelementary character variety, which were originally defined by Penner\,\cite{Penner} for the decorated Teichm\"uller space of orientable punctured surfaces and later generalized by Kashaev\,\cite{kas_coor} for the full character variety of type-preserving representations of orientable surfaces. Maloni--Palesi--Yang \cite{MPY} generalized these coordinates to the case of non-orientable surfaces. One can then prove a formula that  calculates the trace of curves in terms of these coordinates, following ideas of Roger--Yang \cite{roger-yang}, see \cite{yan_ont, MPY}. Another main tool in the proofs is the choice of a good ideal triangulation of the surface, which we call `balanced' triangulations. Recall that $\mathfrak{X}_{\text{ne}}(S)$ is the set of non-elementary representations of $\pi_1(S)$. Recall also that all elementary representations have Euler class zero. 

\begin{Theorem}[Yang \cite{yan_ont}] \label{conncomp-04}
  Let $s \in \{\pm 1\}^4.$
  \begin{enumerate}
    \item $\mathfrak{X}_{0, s}(S_{0,4}) \cap \mathfrak{X}_{\text{ne}}$ is nonempty if and only if $s$ contains exactly two $-1$'s and two $+1$'s.
    \item $\mathfrak{X}_{+ 1, s}(S_{0,4})$ is nonempty if and only if $s$ contains at most one $-1$'s.
    \item $\mathfrak{X}_{- 1, s}(S_{0,4})$ is nonempty if and only if  $s$ contains at most one $+1$'s.
    \item All the nonempty spaces above are connected.
  \end{enumerate}
\end{Theorem}

\begin{Theorem}[Maloni--Palesi--Yang \cite{MPY}]\label{conncomp-13}
  Let $s \in \{\pm 1\}^3.$
  \begin{enumerate}
    \item $\mathfrak{X}_{0,s}(N_{1, 3}) \cap \mathfrak{X}_{\text{ne}}(N_{1, 3})$ is nonempty if and only if $s$ contains exactly one or two $+1$'s.
    \item $\mathfrak{X}_{+ 1, s}(N_{1, 3})$ is nonempty if and only if $s$ contains exactly two or three $+1$'s.
    \item $\mathfrak{X}_{- 1, s}(N_{1, 3})$ is nonempty if and only if  $s$ contains exactly two or three $-1$'s.
    \item All the nonempty spaces above are connected.
  \end{enumerate}
\end{Theorem}

As a consequence, $\mathfrak{X}_{0}(S_{0,4})$ and $\mathfrak{X}_{0}(N_{1, 3})$ have six connected components of non-elementary representations, while $\mathfrak{X}_{+ 1}(S_{0, 4})$ and $\mathfrak{X}_{- 1}(S_{0, 4})$ each have five connected components and $\mathfrak{X}_{+ 1}(N_{1, 3})$ and $\mathfrak{X}_{- 1}(N_{1, 3})$ each have four connected components. A surprising fact that we will discuss below is that different connected components will have different geometric properties.

Recently, for orientable punctured surfaces, Ryu--Yang \cite{RY} classified the number of connected components for type-preserving character varieties using a different method, which instead generalizes ideas used by Goldman \cite{goldman-topological} in the classification of connected components for closed orientable surfaces. In order to state the results, we need to define for a sign $s \in \{\pm 1\}^{n}$, the quantity $\mathrm{p}_+(s)$ as the number of $+1$'s and 
$\mathrm{p}_{-}(s)$ as the number of $-1$'s in the components of $s$. Note that $\mathrm{p}_{+}(s) + \mathrm{p}_{-}(s) = n$.

\begin{Theorem}[Ryu--Yang \cite{RY}] \label{cc-RY-1}
Let $S= S_{g, n}$ be a connected, oriented punctured surface with genus $g \geq 1$ and $n \geq 1$ punctures. Let $e \in \mathbb{N}$ and $s \in \{\pm 1\}^{n}$. Then the space $\mathfrak{X}_{e, s}$ of type-preserving representations with relative Euler class $e$ and sign $s$ is non-empty if and only if the pair $(e, s)$ satisfies the following inequality:
$$\chi(S) + \mathrm{p}_{+}(s) \leq e \leq -\chi(S) - \mathrm{p}_{-}(s).$$
In addition, every non-empty space $\mathfrak{X}_{e, s}$ above is connected.
\end{Theorem}

\begin{Theorem}[Ryu--Yang \cite{RY}] \label{cc-RY-2}
Let $S= S_{0, n}$ be a punctured sphere with $n \geq 3$ punctures. Let $e \in \mathbb{N}$ and $s \in \{\pm 1\}^{n}$. Then the space $\mathfrak{X}_{e, s}$ of type-preserving representations with relative Euler class $e$ and sign $s$ is non-empty if and only if the pair $(e, s)$ satisfies one of the following conditions:
\begin{enumerate}
    \item $$\chi(S) + \mathrm{p}_{+}(s) \leq e \leq -\chi(S) - \mathrm{p}_{-}(s);$$
    \item $e = 0$, and either $\mathrm{p}_{-}(s) = 1$ or $\mathrm{p}_{+}(s) = 1$;
    \item $e = 1$ and $\mathrm{p}_{-}(s) = 0$ or $e = -1$ and $\mathrm{p}_{+}(s) = 0$.
\end{enumerate}
In addition, every non-empty space $\mathfrak{X}_{e, s}$ above is connected.
\end{Theorem}

The inequalities in Theorem \ref{cc-RY-1} and \ref{cc-RY-2} can be considered as ``generalized Milnor--Wood inequalities''. Ryu--Yang also show that the representations satisfying condition $(2)$ in the above theorem are abelian (and hence elementary), see Section 7.1 in \cite{RY}.

The representations which satisfy conditions $(2)$ and $(3)$ in Theorem \ref{cc-RY-2} are the super-maximal representations described by Deroin and Tholozan and that we will introduce in Section \ref{super-maximal}. In particular these representations send every element of the fundamental group represented by a simple closed curve in $S_{0,n}$ to a non-hyperbolic element of $\PSLR$ and one can see that the representations in $(2)$ are abelian, while the representations in $(3)$ are irreducible.

\subsubsection{Hyperbolicity of simple closed curves}\label{geom}

One can generalize Bowditch's \cite{bow_mar} question discussed in Section \ref{o-surface-PSL-R} to the case of type-preserving representations: \textit{Given a non-elementary type-preserving representation $\rho\co\pi_1(S)\to \PGLR$, is it true that if $\rho$ sends every non-peripheral $2$--sided simple closed curve to a hyperbolic element of $\PSLR$, then $\rho$ is Fuchsian?}  As already discussed, March\'e and Wolff \cite{mar-wol} answered affirmatively for the closed genus-$2$ surface $S_{2,0}$, but Yang \cite{yan_ont} and Maloni--Palesi--Yang \cite{MPY} showed that the answer is no for $S_{0,4}$ and $N_{1,3}$, respectively. Moreover, they characterized exactly for which connected components in $\mathfrak{X}_{tp}(S)$ the answer is yes and for which ones the answer is no. Ryu \cite{ryu} partially generalized these results and showed that the answer to the question above is generally no for all punctured orientable surfaces. 

\begin{Theorem}[Yang \cite{yan_ont}] 
\label{hyp-04}\noindent
Let $s \in \{\pm 1\}^4$ and $e \in \mathbb{N}$.
\begin{enumerate}
\item If either $e=1$ and $p_{-}(s) = 1$, or $e=-1$ and $p_{+}(s) = 1$, then there exists a full measure subset $A$ of $\mathfrak X_{e, s}(S_{0,4})$ in which any $\rho \in A$ sends every non-peripheral simple closed curve to a hyperbolic element.
\item Let $s_+=(+1,+1,+1, +1)$ and $ s_-=(-1,-1,-1, -1)$, and let $s \in \{\pm 1\}^4$ be such that $p_{+}(s) =  p_{-}(s) = 2$. Then every representation in $\mathfrak{X}_{0, s}(S_{0,4})$, $\mathfrak{X}_{1, s_+}(S_{0,4})$ and $\mathfrak{X}_{-1, s_-}(S_{0,4})$ sends some non-peripheral simple closed curve to a non-hyperbolic element. 
\end{enumerate}
\end{Theorem}

\begin{Theorem}[Maloni--Palesi--Yang \cite{MPY}] 
\label{hyp-13}\noindent
Let $s \in \{\pm 1\}^3$ and $e \in \mathbb{N}$.
\begin{enumerate}
\item If either $e=1$ and $p_{-}(s) = 1$, or $e=-1$ and $p_{+}(s) = 1$, then there exists a full measure subset $A$ of $\mathfrak X_{e, s}(N_{1,3})$ in which every $\rho \in A$ sends every non-peripheral simple closed curve to a hyperbolic element.
\item Let $s_+=(+1,+1,+1)$ and $ s_-=(-1,-1,-1)$, and let $s \in \{\pm 1\}^3$ such that $p_{+}(s) \in \{1,2\}$. Then every representation in $\mathfrak{X}_{1, s+}(N_{1, 3})$, $\mathfrak{X}_{-1, s_-}(N_{1, 3})$ and $\mathfrak{X}_{0, s}(N_{1, 3})$ sends some non-peripheral $2$-sided simple closed curve to a non-hyperbolic element. 
\end{enumerate}
\end{Theorem}

Note that the representations described in condition $(2)$ of Theorems \ref{hyp-04} and \ref{hyp-13}, that is, type-preserving representations in $\mathfrak{X}_{1, s_+}(S_{0,4})$, $\mathfrak{X}_{-1, s_-}(S_{0,4})$, $\mathfrak{X}_{0, s}(S_{0,4})$, $\mathfrak{X}_{1, s_+}(N_{1, 3})$, $\mathfrak{X}_{-1, s_-}(N_{1, 3})$ and $\mathfrak{X}_{0, s}(N_{1, 3})$, are not Fuchsian, so Theorems \ref{hyp-04} and \ref{hyp-13} gives a negative answer to the generalization of Bowditch's question for punctured surfaces. On all the other components, the answer to Bowditch's question is affirmative.

Ryu generalized part of this result and proved this result for general punctured orientable surfaces.

\begin{Theorem}[Ryu \cite{ryu}] 
\label{gen-bow}\noindent
Let $S = S_{g, n}$ with $\chi(S) \leq -2$ and $n \geq 1$, and assume $e \in \mathbb{N}$ and $s \in \{\pm 1\}^n$ satisfy either $e = -\chi(S)-1$ and $p_{-}(s) = 1$ or $e = \chi(S)+1$ and $p_{+}(s) = 1$. Then there exists a full measure subset $A$ of $\mathfrak X_{e, s}(S_{g,n})$ in which every $\rho \in A$ sends every non-peripheral simple closed curve to a hyperbolic element.

In particular, there exist uncountably many non-Fuchsian type-preserving representations in $\mathfrak{X}_{e, s}(S_{g, n})$ of relative Euler class $e$ and sign $s$ that are totally hyperbolic, that is, they send every non-peripheral element represented by a simple closed curve to a hyperbolic element. 
\end{Theorem}

\subsubsection{Ergodicity of the mapping class group action}

The pure (extended) mapping class group $\mathrm{Mod}(S)$ is the group of isotopy-classes of homeomorphisms of $S$ fixing the boundary components point-wise. It naturally acts on $\mathfrak X(S)$ preserving the Euler class $e$ and the sign of the boundary holonomy $s.$ In the case of closed oriented surfaces, Goldman\,\cite{goldman-survey} conjectured that this action is ergodic on each non-extremal and non-zero component. March\'e and Wolff\,\cite{mar-wol} proved that a positive answer to Bowditch's question implies Goldman's conjecture and used this to prove Goldman's conjecture for $S_2.$ In the case of punctured surfaces, since Bowditch's conjecture is no longer true for all the connected components, the proof of Goldman's result will require a new strategy.  Yang \cite{yan_ont} proved it for the four-punctured sphere and Maloni--Palesi--Yang proved it in most cases of the thrice-punctured projective plane.

\begin{Theorem}[Yang \cite{yan_ont}, Maloni--Palesi--Yang \cite{MPY}]
\label{ergodicity}\noindent
  \begin{enumerate}
  \item Let $s \in \{\pm 1\}^4$ and $e \in \mathbb{N}.$ The mapping class group $\mathrm{Mod}(S_{0,4})$ acts ergodically on all the connected components $\mathfrak{X}_{e, s}(S_{0,4})$ described in Theorem \ref{conncomp-04}.
    \item Let $s_+=(+1,+1,+1)$ and $ s_-=(-1,-1,-1)$, and let $s \in \{\pm 1\}^3$ such that $p_{+}(s) \in \{1,2\}$. Then the mapping class group $\mathrm{Mod}(N_{1,3})$ acts ergodically on the connected component $\mathfrak{X}_{1, s_+}(N_{1, 3}),$ $\mathfrak{X}_{-1, s_-}(N_{1, 3})$ and $\mathfrak{X}_{0, s}(N_{1, 3}).$  
  \end{enumerate}
\end{Theorem}

The ergodicity of the action of $\mathrm{Mod}(N_{1,3})$ on the components $\mathfrak{X}_{1}^{s}(N_{1, 3})$ with $s \in \{\pm 1\}^3$ containing exactly two $+1$'s,  and on the components $\mathfrak{X}_{-1}^{s_-}(N_{1, 3})$ with $s \in \{\pm 1\}^3$ containing exactly two $-1$'s, is still unknown and deserves further study. It is expected that the action would still be ergodic.

\subsubsection{Domination of representations} 

As a corollary of the results discussed above, Yang \cite{yan_ont} and Maloni--Palesi--Yang \cite{MPY} extend to punctured surfaces (orientable and non-orientable, respectively) the result about domination of representations obtained by Gueritaud--Kassel-Wolff \cite{GKW} and Deroin--Tholozan \cite{DT-domi} for orientable closed surfaces and mentioned in Section \ref{o-surface-PSL-R}.

Recall that a representation $\rho$ is said to be \textit{dominated} by another representation $\rho'$ if the traces of the simple closed curves of $\rho$ are less than or equal to those of $\rho'$, in absolute value.

\begin{Theorem}[Yang \cite{yan_ont}, Maloni--Palesi--Yang \cite{MPY}]
\label{domination} \,\\[-0.1in]

\begin{itemize}
    \item Given an orientable punctured surface $S_{g, n}$, every non-Fuchsian type-preserving representation is dominated by a Fuchsian one and for almost every Fuchsian representation $\rho$ and every integer $e$ with $|e|<-\chi(S_{g,n})$, there exists at least one representation with Euler class $e$ which is dominated by $\rho$.
    \item Given a non-orientable punctured surface $N_{k, n}$, every non-Fuchsian type-preserving representation is dominated by a Fuchsian one and for almost every Fuchsian representation $\rho$ and for every integer $e$ with $|e|<-\chi(N_{k,n})$, there exists at least one representation with Euler class $e$ which is dominated by $\rho$. 
\end{itemize} 
\end{Theorem}

\subsection{Deroin--Tholozan representations}\label{super-maximal}

In this section we summarize work of Maret \cite{maret} which studies the action of the mapping class group on some special components of the relative $\PSLR$--character variety $\Xf(\pi_1(S_{0,n}),\PSLR)$ of a punctured sphere $S_{0,n}$. These components were first identified by Deroin--Tholozan \cite{DT}\index{Deroin--Tholozan (or super-maximal) representation} and consist of representations with non-hyperbolic boundary conditions. They have the property that all simple closed curves are sent to elliptic elements of $\PSLR$. Type-preserving Deroin--Tholozan representations already appeared in Section \ref{conn-comp}, while Maret only considers representations with elliptic boundary conditions. Maret's method will be similar to (and is in fact inspired by) ideas described in Section \ref{compact}, where we will describe various ergodicity results for the dynamics on spaces of representations into compact groups. Deroin--Tholozan referred to these representations as \emph{super-maximal representations} because their definition of the relative Euler class exceeds the Euler characteristic of $S_{0,n}$, breaking the classical Milnor--Wood inequality. Note that the definitions of Euler class introduced in Section \ref{sec:euler} do not extend to these representations because they have elliptic boundary conditions; Deroin--Tholozan use a modified version of the definition involving the volume of the representation closely related to another invariant called the Toledo number, see e.g. \cite{BIW}.

\begin{figure}
[hbt] \centering
\includegraphics[height=3cm]{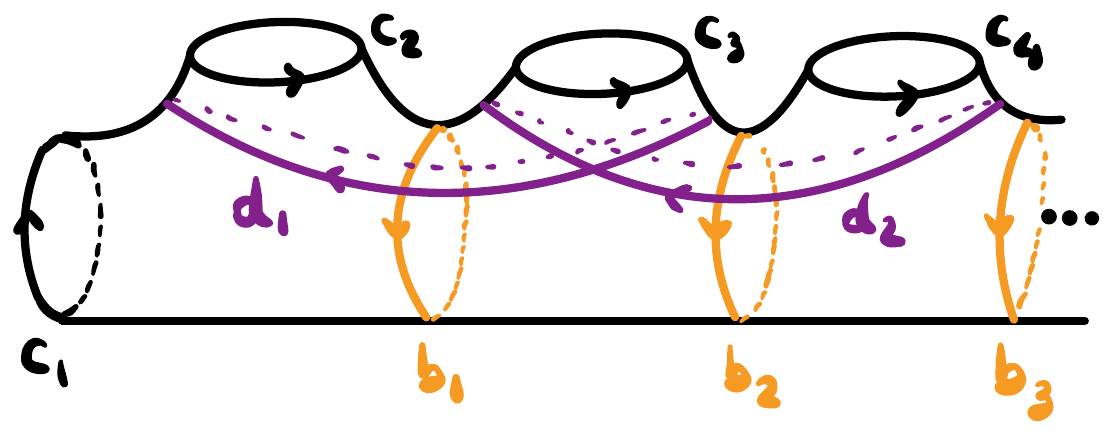}
\caption{The simple closed curves $b_1, \cdots, b_{n-3}$ and $d_1, \cdots, d_{n-3}$ and the peripheral curves $c_1, \cdots, c_{n}$ in $S_{0,n}$, see Figure 5 in \cite{maret}.} 
\label{fig:maret}
\end{figure}

For $n\geq 3$, let $(c_1,\dots,c_n)$ be the generators of $\pi_1(S_{0,n})$, and let $\alpha=(\alpha_1,\dots,\alpha_n)\in(0,2\pi)^n$. The relative character variety $\Xf_\alpha(\pi_1(S_{0,n}),\PSLR)$ consists of conjugacy classes of representations $\rho:\pi_1(S_{0,n})\to\PSLR$ such that $\rho(c_i)$ is an elliptic rotation of angle $\alpha_i$ for each $i$. Deroin--Tholozan showed that if $\alpha_1+\cdots+\alpha_n>2\pi(n-1)$, then there is a nonempty, compact, connected component of $\Xf_\alpha(\pi_1(S_{0,n}),\PSLR)$ of dimension $2(n-3)$, which will be denoted $\Xf_\alpha^{\rm{DT}}(\pi_1(S_{0,n}),\PSLR)$, or just $\Xf_\alpha^{\rm{DT}}(S_{0,n})$ for brevity. These components have the property that every representation sends every simple closed curve to an elliptic element in $\PSLR$. The main theorem of this section is the following:
\begin{theorem}[Maret \cite{maret}] \label{thm:ergodic DT}
    Let $\alpha=(\alpha_1,\dots,\alpha_n)\in(0,2\pi)^n$. Then the action of $\MCG(S_{0,n})$ on $\Xf_\alpha^{\rm{DT}}(S_{0,n})$ is ergodic with respect to the Goldman symplectic measure.
\end{theorem}

\begin{proof}[Proof sketch.]
Similar to Goldman--Xia \cite{GX} and Pickrell--Xia \cite{PX1, PX2}, Maret \cite{maret} shows that the action of the mapping class group, or more precisely the twist flows (see Section \ref{measure}) of a collection of simple closed curves on $\Xf_\alpha^{\rm{DT}}(S_{0,n})$ is locally transitive. Let $c_1,\dots,c_n$ denote the peripheral simple closed curves around the punctures, so that these generate the group $\pi_1(S_{0,n})$. Then let
$$b_i:= c_1^{-1}\cdots c_i^{-1}c_{i+1}^{-1},\ \ d_{i}:= c_{i+1}^{-1}c_{i+2}^{-1},\ \ i=1,\dots,n-3.$$ 
(Recall that in this chapter we are using the standard convention of writing the concatenation in the fundamental group from left to right, differently from Maret.) The $b_i$ are curves that separate $c_1,\dots,c_{i+1}$ from $c_{i+2},\dots,c_n$, while the $d_i$ are curves that separate $c_{i+1}$ and $c_{i+2}$ from the rest of the punctures. See Figure \ref{fig:maret}. This means that all of the curves $d_i$ have the same topological type as the curve $b_1$, and therefore these curves can be mapped to each other by permuting the punctures.

For a nontrivial $\gamma\in\pi_1(S)$, let $\theta_\gamma\co\Xf_\alpha^{\rm{DT}}(S_{0,n})\to(0,2\pi)$ be the map defined by $[\rho]\mapsto2\cos^{-1}\left(\frac{\Tr(\rho(\gamma))}{2}\right)$, which maps $[\rho]$ to the rotation angle of the elliptic element $\rho(\gamma)$. Deroin--Tholozan showed that the action of the Hamiltonian flows of the functions $\theta_{b_1},\dots,\theta_{b_{n-3}}$ on $\Xf_\alpha^{\rm{DT}}(S_{0,n})$ is a maximal effective torus action, meaning that the orbits of the Hamiltonian flows are either circles or fixed points. Since the curves $d_i$ can be mapped to $b_1$ by permuting the punctures, the orbits of the twist flows associated with the curves $d_i$ have the same structure as that of $b_1$. Following the work of Delzant \cite{delzant}, Deroin--Tholozan also show that the moment map $\mu=(\theta_1,\dots,\theta_{n-3})$ maps $\Xf_\alpha^{\rm{DT}}(S_{0,n})$ to a convex polytope $\Delta\subset\Rb^{n-3}$. The interior $\overset{\circ}{\Delta}$ of this polytope consists of the images of representations $[\rho]$ whose orbits under the twist flows of $\theta_{b_1},\dots,\theta_{b_{n-3}}$ are homeomorphic to an $(n-3)$--torus. This means that, in a neighborhood of any representation $[\rho]\in\mu^{-1}(\overset{\circ}{\Delta})$, the Hamiltonian vector fields $X_{b_i}$ of the functions $\theta_{b_i}$ span a $(n-3)$-dimensional (Lagrangian) subspace of the $2(n-3)$--dimensional tangent space. The goal now is to find $n-3$ more curves whose Hamiltonian vector fields, along with the vector fields of the functions $\theta_{b_i}$, span the entirety of the tangent space at a representation $[\rho]$.

To characterize such vectors, we will use the Poisson bracket on $\Xf_\alpha^{\rm{DT}}(S_{0,n})$ associated to the Goldman symplectic form $\omega$, see Section \ref{measure}. Given two smooth functions $\zeta_1,\zeta_2\co\Xf_\alpha^{\rm{DT}}(S_{0,n})\to\Rb$ with Hamiltonian vector fields $X_{\zeta_1},X_{\zeta_2}$, their Poisson bracket is
$$\{\zeta_1,\zeta_2\}:=\omega(X_{\zeta_1},X_{\zeta_2})=d\zeta_2(X_{\zeta_1}).$$
Because the fibers of $\mu$ over $\overset{\circ}{\Delta}$ are Lagrangian tori, if for some $[\rho]\in\mu^{-1}([\rho])$ and some $a\in\pi_1(S)$ we have $\{\theta_{b_1},\theta_{a}\}([\rho])\neq 0$, then the vector $X_a$ sits outside of the span of the $X_{b_i}$ in $T_{[\rho]}\Xf_\alpha^{\rm{DT}}(S_{0,n})$. Motivated by this, consider the following set
$$\Ec:=\{[\rho]\in\mu^{-1}(\overset{\circ}{\Delta})\ |\ \forall i=1,\dots,n-3,\ \exists m_i\in\Zb, \{\theta_{b_1},\theta_{(\tau_{b_i})^{m_i}d_i}\}([\rho])\neq 0\},$$
or equivalently
$$\Ec=\mu^{-1}(\overset{\circ}{\Delta})\cap\bigcap\limits_{i=1}^{n-3}\bigcup\limits_{m_i\in\Zb}\{\theta_{b_1},\theta_{(\tau_{b_i})^{m_i}d_i}\}^{-1}(\Rb\setminus\{0\}).$$
This set is open and measurable in $\Xf_\alpha^{\rm{DT}}(S_{0,n})$. Given a $[\rho]\in\Ec$, the Hamiltonian vector fields $X_{b_i},X_{(\tau_{b_i})^{m_i}d_i},\ i=1,\dots,n-3$ generate the tangent spaces in some neighborhood $\Uc$ of $[\rho]$ contained in $\Ec$. The Dehn twists $\tau_{b_i},\tau_{(\tau_{b_i})^{m_i}d_i}$ act by rotations on the circular orbits of the Hamiltonian flows. The representations for which the Dehn twists $\tau_{b_i},\tau_{(\tau_{b_i})^{m_i}d_i}$ have irrational rotation angle have full measure in $\Xf_\alpha^{\rm{DT}}(S_{0,n})$, and irrational rotations act ergodically on the circle, so any $\MCG(S_{0,n})$--invariant function $f\co\Xf_\alpha^{\rm{DT}}(S_{0,n})\to\Rb$ is constant almost everywhere on almost every orbit of the Hamiltonian flows, and therefore $f$ is constant almost everywhere on $\Uc$. This means that $f$ is locally constant on $\Ec$. If $\Ec$ is connected, then $f$ is constant almost everywhere on $\Ec$, and if $\Ec$ has full measure in $\Xf_\alpha^{\rm{DT}}(S_{0,n})$, then $f$ is constant almost everywhere on $\Xf_\alpha^{\rm{DT}}(S_{0,n})$, which then implies that the action of $\MCG(S_{0,n})$ on $\Xf_\alpha^{\rm{DT}}(S_{0,n})$ is ergodic. The proofs of both connectedness and full measure come down to showing that all of the orbits of the Hamiltonian flows of points in $\mu^{-1}(\overset{\circ}{\Delta})$ must intersect $\Ec$ and that the orbits of the points for which the $\theta_{b_i}$ are all irrational are entirely contained in $\Ec$. 
\end{proof}

\vspace{.1in}

Bouilly--Faraco--Maret \cite{BFM-minimalDT} expand on Theorem \ref{thm:ergodic DT}, proving that, except for finite mapping class group orbits, the orbits of the points in $\Xf_\alpha^{\rm{DT}}(S_{0,n})$ under the action of $\MCG(S_{0,n})$ are dense in $\Xf_\alpha^{\rm{DT}}(S_{0,n})$. Recently, Bronstein-Maret \cite{BM} classified the finite orbits of the mapping class group in the components of the relative character varieties of punctured spheres consisting of (conjugacy classes of) Deroin-Tholozan representations, positively answering a conjecture of Tykhyy \cite{Tykhyy}. In particular, when $n \geq 7$, Bronstein--Maret \cite{BM} proved that every orbit is infinite. Therefore, the combination of this result with \cite{BFM-minimalDT} proves that when $n\geq 7$, the action of the mapping class group on $\Xf_\alpha^{\rm{DT}}(S_{0,n})$ is minimal, that is, every orbit is dense.

\section{Representations in compact groups}\label{compact}

In this section we consider the character variety $\Xf(\pi_1(S),\Gs)$ where $S$ is a (not necessarily orientable) surface and $\Gs$ is a compact Lie group. Unlike the case of $\Xf(\pi_1(S),\PSLR)$, where we had components on which the mapping class group  $\MCG(S)$ was acting properly discontinuously, it is expected that the action of $\MCG(S)$ on each connected component of $\Xf(\pi_1(S),\Gs)$ is ergodic and a lot is known in this context.

\subsection{Orientable surface groups}\label{compact-orientable}

The main result of this section is that if $\Gs$ is a compact Lie group and $S$ is a compact, orientable surface with $\chi(S)<0$, then the action of $\MCG(S)$ on each connected component of $\Xf(\pi_1(S),\Gs)$ is ergodic. This is shown in various cases by the following authors:
\begin{enumerate}
    \item Goldman \cite{gol-erg} for any compact surface with $\chi(S)<0$ when $\Gs$ is locally isomorphic to a product of copies of $\SU(2)$ and $\mathsf{U}(1)$.
    \item Goldman--Xia \cite{GX} for any compact surface with $\chi(S)<0$ when $\Gs=\SU(2)$ using a different proof.
    \item Goldman--Lawton--Xia \cite{GLX} for the surface $S_{1,1}$ and $\Gs=\SU(3)$, using a proof similar to Goldman--Xia.
    \item Pickrell--Xia \cite{PX1,PX2} for any compact surface $S_{g,n}$ with $g\geq1$ and any compact, connected Lie group $\Gs$.
\end{enumerate}

In this section we will outline the proofs in Goldman--Xia and Pickrell--Xia. We will delay a discussion of Goldman's original proof until the Section \ref{compact-non-orientable} as it is similar to the proof of a different result discussed there.

\subsubsection{Ergodicity in \texorpdfstring{$\SU(2)$}{SU(2)} Character Varieties}\hfill\\

For this proof, if $S=S_{g,n}$ is the compact genus $g$ surface with $n$ boundary components $c_1, \ldots, c_n$ and $\Pc\in[-2,2]^n$, we will denote by $\Xf_{\Pc}$ the relative character variety consisting of classes of representations with $\mathrm{Tr}(\rho(c_i))={\Pc}_i$ for all $i$.

\textbf{Step 1:} The first step in this proof is to relate the action of the mapping class group to the symplectic structure of the character variety. This is done by describing $\Rb$-valued functions on the character variety whose Hamiltonian flows in a sense give a continuous action that extends the discrete action of Dehn twists around simple closed curves. That is, given a simple closed curve $\alpha$, the action of the Dehn twist $\tau_{\alpha}$ on the character variety is the same as flowing along the corresponding Hamiltonian flow for some amount of time; for the general construction of this see Section \ref{measure}.

Concretely, let $f:\SU(2)\to \Rb$ be the trace function $x\mapsto\mathrm{Tr}(x)$. Each element $x\in\SU(2)$ is conjugate to a matrix of the form $\begin{pmatrix}
    e^{i\theta} & 0 \\ 0 & e^{-i\theta}
\end{pmatrix}$
for some $\theta\in[0,2\pi)$, and $f(x)=2\cos(\theta)$. For $x=g\begin{pmatrix}
    e^{i\theta} & 0 \\ 0 & e^{-i\theta}
\end{pmatrix}g^{-1}$ we get $\zeta^t(x)=g\begin{pmatrix}
    e^{2i\sin(\theta)t} & 0 \\ 0 & e^{-2i\sin(\theta)t}
\end{pmatrix}g^{-1}$. Setting $s(x)=\frac{\theta}{2\sin(\theta)}$ gives a time at which $\zeta^{s(x)}(x)=x$, so the subgroup contains $x$, and $\xi_\alpha^{s(\rho(\alpha_-))}(\rho)=\tau_\alpha(\rho)$. The twist flow descends to the Hamiltonian flow on $\Xf_{\Pc}(\pi_1(S),\SU(2))$ corresponding to the function $f_\alpha:\Xf_{\Pc}(\pi_1(S),\SU(2))\to\Rb$ given by $f_\alpha(\rho)=\mathrm{Tr}(\rho(\alpha))$.

\textbf{Step 2:} Now we want to analyze the action of the Dehn twist on the orbits of this flow. 

We first note that, other than when $x=\pm\mathrm{Id}$, the one-parameter subgroup
$$\zeta^t(x)=g\begin{pmatrix}
    e^{2i\sin(\theta)t} & 0 \\ 0 & e^{-2i\sin(\theta)t}
\end{pmatrix}g^{-1}$$
is a circle. Moreover, if $\theta=\cos^{-1}\left(\frac{\mathrm{Tr}(x)}{2}\right)$ is an irrational multiple of $\pi$, then the subgroup $\ip{x}$ of $\{\zeta^t(x)\}_{t\in\Rb}$ acts ergodically on this circle. Translating this to the language of Dehn twists and twist flow, we have that, for almost every $b\in[-2,2]$, if $f_\alpha(\rho)=b$, then the twist $\tau_\alpha$ acts ergodically on the orbit $\{\xi_\alpha^t([\rho])\}_{t\in\Rb}$. We then have the following:
\begin{proposition}
    Let $\alpha$ be a simple closed curve with twist flow $\xi_\alpha$ and Dehn twist $\tau_\alpha$. Let $\psi:\Xf_{\Pc}\to\Rb$ be a measurable function invariant under the action of $\ip{\tau_\alpha}$. Then there exists a nullset $\mathcal{N}$ of $\Xf_{\Pc}$ such that the restriction of $\psi$ to the complement of $\mathcal{N}$ is constant on each orbit of the twist flow $\xi_\alpha$.
\end{proposition}

The idea of the proof is below:
\begin{proof}
    The fibers of the quotient map $\Xf_{\Pc}\to\Xf_{\Pc}/\xi_\alpha$ are the orbits $\{\xi_\alpha^t([\rho])\}_{t\in\Rb}$, almost all of which are circles. Disintegrating the symplectic measure on $\Xf_{\Pc}$ over this map gives a measure on almost all of these circles. The set $\mathcal{N}=f_\alpha^{-1}(2\cos(\Qb\pi)$ has measure zero in $\Xf_{\Pc}$, and as above the action of $\ip{\tau_\alpha}$ is ergodic on each circle in the complement of $\mathcal{N}$, meaning that $\psi$ is constant on each of these circles.
\end{proof}

\textbf{Step 3:} In Step 2 we saw that a $\MCG(S)$--invariant function $\psi:\Xf_{\Pc}\to\Rb$ is constant on each orbit of the twist flow $\tau_\alpha$ for each simple closed curve $\alpha$. The last step of this proof is to show that all of the twist flows together act transitively on $\Xf_{\Pc}$. This is the point at which the proof diverges from the original proof in \cite{gol-erg}. In Step 1 we identified the twist flow $\xi_\alpha$ with the Hamiltonian flow of the trace function $f_\alpha$. These trace functions sit in the \emph{character ring} $\Cb[\Xf(F_N,\SL_2\Cb)]$, which is the ring of conjugation-invariant functions $\SL_2\Cb^N\to\Cb$ that are locally given by polynomials. The main tool in this step will be to find a finite collection $\mathcal{S}\subset\pi_1(S)$ of simple closed curves on $S$ whose trace functions generate the entire character ring. These curves are defined explicitly in \cite{GX}.

The maximal ideals in the character ring are in correspondence with the cotangent spaces at points in the character variety through taking the differential. If the entire character ring is generated by trace functions then their differentials span the cotangent space at a point. These differentials are dual to the Hamiltonian vector field for the trace functions, meaning the Hamiltonian vector fields also span the tangent space; in this case the action of the Hamiltonian flows is said to be \emph{infinitesimally transitive} at that point. Because this is true at every point and because $\Xf_{\Pc}$ is connected, the flows act transitively on the character variety. Together with the conclusion of Step 2 this means that a function $\psi:\Xf_{\Pc}\to\Rb$ invariant under the action of $\MCG(S)$ is constant almost everywhere on $\Xf_{\Pc}$, and therefore the action of $\MCG(S)$ is ergodic.

\subsubsection{Ergodicity in Compact Lie Group Character Varieties}\hfill\\

In \cite{PX1,PX2}, Pickrell--Xia show that for a compact surface $S=S_{g,n}$ of genus $g\geq 1$ and any compact, connected Lie group $\Ks$, the mapping class group $\MCG(S)$ acts ergodically on each connected component of $\Hom(\pi_1(S),\Ks)$, with respect to the measure associated to the symplectic structure defined in Section \ref{measure}). Here we sketch the proof of the closed surface case. The key steps in this proof are the following:
\begin{enumerate}
    \item Find a continuous group action $\mathcal{G}$ extending the action of $\MCG(S_{1,1})$ on $\Hom(\pi_1(S_{1,1}),\Ks)$. This step also appeared in Goldman--Xia's proof, though the tools used here come from operator theory rather than symplectic geometry.
    \item Show that this action is infinitesimally transitive at all points outside of a set of codimension at least $2$. This means that the set of points on which the action is infinitesimally transitive is connected, and that the orbit of any point in this set under the action of $\mathcal{G}$ is a connected, full measure subset of a connected component of $\Hom(\pi_1(S_{1,1}),\Ks)$, and therefore a function invariant under this action is almost everywhere constant on the entire connected component.
    \item Prove a `Sewing Lemma' (Lemma 1.3 in \cite{PX1}) which allows one to study the action on $\Hom(\pi_1(S),\Ks)$ by cutting $S$ along a separating curve $\alpha$ into two smaller surfaces. This allows us to reduce to the study of the action in smaller surfaces. The case of $S=S_{1,n}$ follows from this and the previous step.
    \item Use the Sewing Lemma to show that the action of $\MCG(S)$ on $\Hom(\pi_1(S),\Ks)$ is ergodic for $S=S_{g,1}$ by inducting on $g$, and conclude the proof.
\end{enumerate}
Note that the last two steps resemble the end of Goldman's original proof in \cite{gol-erg} and Palesi's work in \cite{pal-erg} in the case of non-orientable surfaces, which will be discussed in the next section.

\textbf{Step 1:} The fundamental group of $S=S_{1,1}$ is $\pi_1(S_{1,1})\cong F_2=\ip{a,b}$, and any representation $\rho\in\Hom(\pi_1(S_{1,1}),\Ks)$ is determined by the images of the generators. So we have $\Hom(\pi_1(S_{1,1}),\Ks)\cong \Ks\times \Ks$. Choosing $a$ and $b$ to be the standard generators of $\pi_1(S)$, the boundary component $c$ of $S$ is given by $c=[a,b]$, and the commutator map $p:\Ks\times \Ks\to \Ks$ can be thought of as a map $\Hom(\pi_1(S_{1,1}),\Ks)\to \Ks$ given by $\rho\mapsto\rho(c)$. For $k\in \Ks$, let $\Hom(\pi_1(S),\Ks;k)$ denote the space of representations $\rho$ with $\rho(c)=k$; then $\Hom(\pi_1(S_{1,1}), \Ks;k)\cong p^{-1}(k)$.

Using these identifications, the Dehn twists about $a$ and $b$ act on $\Ks\times \Ks$ as maps $T_j:\Ks\times \Ks\to \Ks\times \Ks$ given by
$$T_1(g,h)=(gh^{-1},h)\ \text{ and }\ T_2(g,h)=(g,hg^{-1}).$$

On the other hand, these maps fix $p$, i.e. $p(T_1(g,h))=p(g,h)$, so they preserve the image of the boundary element under any $\rho$. They are also volume-preserving with respect to the Haar measure on $\Ks\times \Ks$, so they induce unitary transformations of $L^2(\Ks\times \Ks)$ by pre-composition. Let $T$ be the unitary map on $L^2(\Ks_1\times \Ks_2)$ corresponding to $T_2$.

To study the ergodicity of the mapping class group $\MCG(S_{1,1})=\ip{T_1,T_2}$ on $\Hom(\pi_1(S_{1,1}),\Ks;k)$, we want to consider $\MCG(S_{1,1})$--invariant functions $F\in L^2(\Ks\times \Ks)$. Our goal in this step is to find a continuous group $\mathcal{G}$ of transformations of $L^2(\Ks\times \Ks)$ such that $F$ is $\mathcal{G}$-invariant if and only if it is $\MCG(S_{1,1})$--invariant. To do this, we will use the Peter--Weyl theorem (see, e.g. \cite{peter-weyl} Chapter 2) to consider the elements of $L^2(\Ks\times \Ks)$ as linear operators instead. For notational clarity, we will introduce copies $\Ks_1$ and $\Ks_2$ of $\Ks$. The Peter-Weyl Theorem states that there is an isomorphism
\begin{align*}
	P\co&\bigoplus_{\mu}\Lc(V_\mu)\to L^2(\Ks)\\
	&(L_\mu)_\mu \mapsto f,
\end{align*}
where $$f(g)=\sum_{\mu}\dim(\mu)^{\frac{1}{2}}\text{tr}_\mu(L_\mu\pi_\mu(g^{-1})),$$
and where the sums are taken over all irreducible representations $\pi_\mu:\Ks\to \Lc(V_\mu)$, with $\Lc(V_\mu)$ the space of linear maps on the vector space $V_\mu$. The element $(g_l,g_r)\in \Ks\times \Ks$ acts on these spaces as follows:
\begin{align*}
    L_\mu &\mapsto \pi_\mu(g_l)L_\mu\pi_\mu(g_r)^{-1} \\
    f(g) &\mapsto f(g_lgg_r^{-1}).
\end{align*}

We can now apply this isomorphism to $L^2(\Ks\times \Ks)$ in the following way:
\begin{align*}
    L^2(\Ks_1\times \Ks_2)& \to L^2(\Ks_1;L^2(\Ks_2))\to \bigoplus_\mu L^2(\Ks_1;\Lc(V_\mu)) \\
    f & \mapsto  \left(g_1 \mapsto h(g_1) \right)\mapsto \bigoplus_{\mu}\left(g_1 \mapsto P^{-1}(h(g_1))\right)_{\mu},
\end{align*} where $h(g_1)=f(g_1,\cdot)$. 
Under this identification, the formula from the Peter--Weyl theorem gives that $T$ becomes $\rm{diag}(T_\mu)$, where $T_\mu$ is the multiplication operator 
\begin{align*}
	T_\mu: &\;L^2(\Ks_1;\Lc(V_\mu))\to L^2(\Ks_1;\Lc(V_\mu))\\
	&\; (g \mapsto F_\mu(g)) \mapsto (g \mapsto F_\mu(g)\pi_\mu(g)^{-1}).
\end{align*}

We have the following result:
\begin{lemma}
    The elements of $L^2(\Ks_1;\Lc(V_\mu))$ fixed by $T_\mu$ are   $$L^2(\Ks_1;\Lc(V_\mu))^{T_\mu}=\{F_\mu\ :\ F_\mu(g)\big|_{(V_{\mu}^g)^\perp}=0\ \text{ a.e. } g\}$$
    where $(V_{\mu}^g)^\perp$ is the subspace $\mathrm{Im}(\pi_\mu(g)-\id)$.

    If $F_\mu$ is $T_\mu$--invariant, then (using the isomorphism described above, we can think of the sequence $(F_\mu)$ as defining an element of $L^2(\Ks_1\times\Ks_2)$ that we denote $\oplus F_\mu$, and) we have:
    $$(\oplus F_\mu)(g,h)=(\oplus F_\mu)(g,ha(g)^{-1})$$
    for any measurable $a:\Ks\to \Ks$ such that $[a(g),g]=1\text{ a.e. }g$. 
\end{lemma}

If $F_\mu$ is instead $T_1$--invariant, then we have $(\oplus F_\mu)(g,h)=(\oplus F_\mu)(ga(h)^{-1},h)$ for any measurable $a:\Ks\to \Ks$ such that $[a(h),h]=1\text{ a.e. }h$.

Define
$$A:=\left\{a:\Ks\to\Ks \mid [g,a(g)]=1\ \forall g\in \Ks\right\}.$$
Even if $A$ is not necessarily a Lie group, we will call
$$\aL=\{x:\Ks\to\kL\ |\ {\rm Ad}_g(x(g))=x(g)\ \forall g\in \Ks\}$$
its Lie algebra because it has the property that $\exp(\aL)\subset A$. We then define two different actions $A_1$ and $A_2$ of $A$ on $\Ks_1\times \Ks_2$, given by
$$A_1(a)\cdot(g,h)=(ga(h)^{-1},h)\ \ \text{ and }\ \ A_2(a)\cdot(g,h)=(g,ha(g)^{-1}).$$
With this definition, the Dehn twists $T_i$ correspond to $A_i(a)$ for the map $a(k)=k$ and the identity element corresponds to $A_i(a)$ for the map $a(k)=e$. We then define $\Gc$ to be the group generated by $A_1(A)$ and $A_2(A)$ inside the group of all volume-preserving diffeomorphisms of $\Ks\times \Ks$. The group $\Gc$ will serve as our continuous extension of the action of $\MCG(S_{1,1})$ on $\Ks\times\Ks$. Again, $\Gc$ is not necessarily a Lie group, but we will define the sets:
\begin{itemize}
    \item $\gL_0=\{(x(\cdot),y(\cdot)) : \Ks \times \Ks \to \kL \times \kL \ |\ x,y\in\aL\}$ and
    \item $\gL$ is the Lie algebra generated by $\{{\rm Ad}_\sigma \xi \ |\ \xi \in \gL_0, \sigma \in A_1(A) \ \text{ or } A_2(A)\}$, where  ${\rm Ad}_\sigma \xi \co K \times K \to \mathfrak{k} \times \mathfrak{k}$  is defined by ${\rm Ad}_{\sigma(g, h)} \xi(g, h).$
\end{itemize}
Heuristically speaking, this means that $\gL_0$ is the Lie algebra of the group generated by the identity components of $A_1(A)$ and $A_2(A)$, and that $\gL$ is the Lie algebra of $\Gc$. 

\textbf{Step 2:} We now wish to show that the action of $\Gc$ is infinitesimally transitive at all points outside of a set of codimension $2$. The relative representation varieties are the fibers of the commutator map $p$, which is submersive almost everywhere, so the tangent space at a point $(g,h)\in\Hom(\pi_1(S_{1,1}),\Ks; k)$ is $\ker(dp|_{(g,h)})$. We want to prove the following:
\begin{lemma}
    For $(g, h)$ in the complement of a set of codimension strictly greater than $1$ in $\Ks \times \Ks$, the evaluation map
    $${\rm eval}|_{(g,h)}:\gL\to\ker(dp|_{(g,h)})\ ,\ (x,y)\mapsto(x(h),y(g))$$
    is surjective.
\end{lemma}

The proof of this statement is rather involved, so we will only highlight some of the tools used here. First, one can compute that 
$$dp|_{(g,h)}:(\xi,\eta)\mapsto \Ad_{hgh^{-1}}\xi-\Ad_{hg}\xi+\Ad_{hg}\eta-\Ad_{h}\eta.$$
 This can be used to quickly reduce to the semisimple case, and to show that $p$ is regular off a set of codimension $2$, so that the dimension of $\ker(dp)$ is equal to $\dim(\Ks)$ off a set of codimension $2$. Further, the computation of $dp$ allows one to write an exact sequence involving $\ker(dp)$, by looking at inclusion of one of the factors and projection onto the other, decomposing it into two summands. Careful analysis of the image of $\gL$ in these summands then gives the desired result.

\textbf{Step 3:} The main ingredient in this step is the Sewing Lemma (Lemma 1.3 in \cite{PX1}) which allows one to reconstruct the representation space $\Hom(S,\Ks)$ from the representation varieties $\Hom(S^+,\Ks,k)$ and $\Hom(S^-,\Ks,k)$, where $S^+$ and $S^-$ are the two surfaces (with boundary) obtained by cutting the surface $S$ along a separating curve $\alpha$ and $k \in \Ks$ is the value of the image of the curve $\alpha$ seen in $S^\pm$ as a boundary component. Here we also need to briefly remark on a detail that has been omitted up to this point. Pickrell--Xia give an additional structure to surfaces with boundary: a basepoint, a path from the basepoint to each boundary component, and a sign on each boundary component corresponding to whether the orientation of the boundary curve agrees or disagrees with the orientation of the surface. This means that, for surfaces with boundary, $\Aut(S)$ differs from $\Homeo(S)$, and therefore $\pi_0({\rm Aut}(S))$ is not the mapping class group of $S$, but it is instead the subgroup generated by Dehn twists about curves which do not cross the paths from the basepoint to the boundary. However, for closed surfaces the two definitions coincide. The upshot of this is that one can prove the result for $S_{1,n}$ as a corollary of the above step and the Sewing Lemma in this context. 

\textbf{Step 4:} The goal of this step is to prove the following theorem:
\begin{theorem}
    If $S = S_{g, 1}$ is a once-punctured surface with genus $g>1$, then for every group element boundary condition $k\in \Ks$, the action
$$\pi_0(\Aut(S))\times\Hom(S,\Ks;k)\to\Hom(S,\Ks;k)$$
is ergodic on each connected component. 
\end{theorem}
Taking the boundary condition $k= 1 \in \Ks$, then gives the desired result for closed surfaces. To show this, consider first the case where $S=S_{2,1}$; the more general case of $S=S_{g,1}$ follows similarly by induction. We further will consider a decomposition of $S$ into a one-holed torus and a two-holed torus by cutting along a curve $\beta$, i.e. $S=S_{1,1}\sqcup_\beta S_{1,2}$ (and in the general case $S_{g,1}=S_{g-1,1}\sqcup_\beta S_{1,2}$). For any boundary condition $k\in\Ks$, let $F:\Hom(S,\Ks;k)\to\Rb$ be a measurable $\pi_0(\Aut(S))$--invariant function. In particular, $F$ is invariant under the subgroup of $\pi_0(\Aut(S))$ that fixes $\beta$. Now fix some $k_{\beta}\in\Ks$ and consider the set of representations $\rho\in\Hom(S,\Ks;k)$ such that $\rho(\beta)=k_{\beta}$; this can be identified with the product
$$\Hom(S_{1,1},\Ks;k_{\beta})\times\Hom(S_{1,2},\Ks;(k_{\beta},k))$$
by cutting $S$ along $\beta$. Applying the results of Steps 2 and 3 we have that $F$ is a.e. constant on this set for a.e. choice of $k_{\beta}$, meaning that the value of $F(\rho)$ depends only on the image of $\beta$ under $\rho$. In order to show that $F$ is a.e. constant, we want to choose a curve $\alpha$ that crosses $\beta$ (but does not cross the chosen path from the basepoint to the boundary component) and consider the action of the Dehn twist $\tau_\alpha$. We can then write the curve $\tau_\alpha(\beta)$ explicitly in terms of the chosen generators of $\pi_1(S)$. The function $F$ is $\pi_0(\Aut(S))$--invariant so $F(\tau_\alpha\cdot\rho)=F(\rho)$, but the image $\rho(\beta)$ is in general different from $(\tau_\alpha\cdot\rho)(\beta)$ because $\alpha$ was chosen to cross $\beta$. Using this fact, along with some computations similar to those in Step 2, gives the result that $F$ is a.e. constant and therefore the action of $\pi_0(\Aut(S))$ on $\Hom(S,\Ks;k)$ is ergodic.

\subsection{Non-orientable surface groups}\label{compact-non-orientable}

In this section we discuss Palesi's proof from \cite{pal-erg} that if $N$ is a compact, non-orientable surface with $\chi(N)\leq -2$ then the action of the mapping class group $\MCG(N)$ on the relative character variety $\Xf_C(\pi_1(N),\SU(2))$ is ergodic. Throughout this section we will denote $\Gs=\SU(2)$, $\Xf(N)=\Xf(\pi_1(N),\Gs)$ for a closed surface, and for a punctured surface with $m$ boundary components, we let $C = (C_1, \ldots, C_m) \in \Gs^m$ and $\Xf_C(N)$ will denote the relative character variety consisting of classes of representations $\rho$ with $\rho(c_i)=C_i$ for all $1 \leq i \leq m$. Many of the elements of this proof resemble Goldman's original proof for orientable surfaces.

As in the work of Goldman--Xia, the action of a Dehn twist about a two-sided curve $\alpha$ can be extended to a twist flow $\xi_\alpha^t$ on $\Hom(\pi_1(N),\Gs)$. A particular case to highlight here is the case where $N\cong N_{2g+2,m}\cong S_{g,m}\# N_2$ and $\alpha$ is a (nonseparating) two-sided curve that passes through each of the two cross-caps in $S_{g,m}\# N_2$ once. In this case $N\setminus \alpha$ is an orientable surface homeomorphic to $S_{g,m+2}$. Denote $A=N \setminus\alpha$ and the function $f$ defined as follows.
 \begin{align*}
	 f:&\Gs\to[-2,2]\\
	 &g\mapsto\cos^{-1}\left(\frac{\Tr(g)}{2}\right).
 \end{align*}
The Dehn twist $\tau_{\alpha}$ acts on a generic fiber of the map $\phi:\Xf_C(N)\to\Xf_{C'}(A)$ by a rotation of angle $f(\rho(\alpha))$, where $C'$ is determined by $C$ and the trace of the image of $\alpha$. This means that the map $\phi$ is $\tau_\alpha$-invariant. Moreover, the set of representations for which $f(\rho(\alpha))$ is rational has measure zero in $\Xf_C(N)$, meaning that $\tau_\alpha$ acts ergodically on the circle $\phi^{-1}([\rho])$ for almost all $[\rho]\in\Xf_{C'}(A)$. Let $h:\Xf_C(N)\to\Rb$ be a $\tau_\alpha$-invariant function. By ergodic decomposition (see \cite{erg-decomp} Theorem 5.8), there is an $H:\Xf_{C'}(A)\to\Rb$ such that $h=H\circ\phi$ almost everywhere. So, a mapping class group invariant function is almost everywhere equal to a function depending only on $\Xf_{C'}(A)$. 

The mapping class group $\MCG(A)$ of $A$ can be embedded as a subgroup of $\MCG(N)$. Because $A$ is orientable, Goldman's result in \cite{gol-erg} shows that a $\MCG(A)$--invariant function $f:\Xf_{C'}(A)\to\Rb$ is almost everywhere equal to a function depending only on the traces of the boundary components. Letting $\alpha_+$ and $\alpha_-$ denote the boundary components resulting from cutting $N$ along $\alpha$, we have that any $[\rho_A]\in\phi(\Xf_{C}(N))$ must have $\Tr(\rho(\alpha_+))=\Tr(\rho(\alpha_-))$. Letting $C_i=\Tr(\rho(c_i))$ and $x=\Tr(\rho(\alpha))$, we have the following:
\begin{proposition}\label{palesi-coords}
    Let $f:\Xf_C(N)\to\Rb$ be a $\MCG(N)$--invariant function. There exists a function $G:[-2,2]^{m+1}\to\Rb$ such that $f([\rho])=G(x,C_1,\dots,C_m)$ almost everywhere.
\end{proposition}
Because of the initial assumption that $N$ is non-orientable with even genus and $\chi(N)\leq-2$, the surface $S$ must contain an embedded two-holed Klein bottle. This assumption turns out to be necessary, otherwise the mapping class group of $S$ would be too small to find a Dehn twist that acts nontrivially on $x$. In order to conclude the proof of the theorem in the case of an even genus surface, we need only study the particular case of $N_{2,2}$.

Here we give only an outline of the case of $N_{2,2}$ and refer the reader to \cite{pal-erg} for the complete details. We first write the trace coordinates of Magnus \cite{magnus-coords} on the space $\Xf(N_{2,2})$. Then, an appropriate curve $U$ in $N_{2,2}$ is chosen and the action of the Dehn twist $\tau_{U}$ is written in terms of these trace coordinates. The chosen curve is one that intersects $\alpha$ non-trivially so that it acts non-trivially on the trace coordinate of $\alpha$. Using a few changes of coordinates and an identification of $\Xf(F_2,\Gs)$ as a subset of $\Rb^3$ (see Lemma 5.3.1 in \cite{pal-erg}), the variety $\Xf(N_{2,2})$ can be decomposed as a family of ellipses parameterized by these trace coordinates. Moreover, the action of $\tau_U$ preserves each of these ellipses and acts by a rotation of angle $\theta_U=\cos^{-1}(\frac{1}{2}\Tr(\rho(U)))$ on each of these ellipses. For almost all boundary traces and trace coordinates this angle is an irrational multiple of $\pi$, and the action of $\tau_U$ is ergodic on each ellipse for which $\theta_U$ is irrational. This gives that $f([\rho])$ is equal to a function depending only on the traces of the generators of $\pi_1(N_{2,2})$, while Proposition \ref{palesi-coords} gives that $f([\rho])$ is equal to a function depending only on the traces of the boundary components and the curve $\alpha$. Therefore $f$ must depend only on the traces of the boundary components, and must be constant on each $\Xf_C(N_{2,2})$. This then finishes the proof of the main theorem for non-orientable surfaces of even genus.

The work to prove the theorem for non-orientable surfaces of odd genus is similar. One first studies the particular cases of $N_{1,3}$ and $N_{3,1}$ using the same trace coordinates and shows that the Dehn twists around properly chosen curves act as rotations about ellipses in the character varieties. Any non-orientable surface $N_{g,m}$ with $\chi(N_{g,m})\leq -2$ must either have $g\geq 3$, or $g=1$ and $m\geq 3$, meaning that $N_{g,m}$ contains either an embedded $N_{1,3}$ or $N_{3,1}$.

\subsection{Free groups}\label{compact-free}

In this section we discuss the work of Gelander \cite{gelander} showing that for the free group $F_n$ with $n\geq3$, $\Out(F_n)$ acts ergodically on $\Xf(F_n,\Gs)$ for a compact Lie group $\Gs$. The following lemma is used in the proof:
\begin{lemma}
    The set $U$ of all pairs $(x_1,x_2)\in \Gs\times \Gs$ for which the group $\ip{x_1,x_2}$ is dense in $G$ is open, dense, and of full Haar measure in $\Gs\times \Gs$.
\end{lemma}

With this, we can prove the desired theorem:
\begin{theorem}
    Let $\Gs$ be a compact, connected, semisimple Lie group, and let $n\geq 3$. The action of $\Aut(F_n)$ on $\Gs^n=\Hom(F_n,\Gs)$ is ergodic.
\end{theorem}

\begin{proof}
    For the sake of contradiction assume that $A\subseteq \Gs^n$ is an invariant measurable subset that does not have full measure or zero measure. For each $g\in \Gs$ and $1\leq i\leq n$, define $\varphi_{g,i}\in\Aut(\Gs^n)$ as
    $$\varphi_{g,i}(g_1,\dots,g_n)=(g_1,\dots,gg_i,\dots,g_n),$$
    that is, $\varphi_{g,i}$ acts by left translation by $g$ in the $i$-th factor. Because the action of $\Gs$ on itself by left translations is ergodic, $A$ cannot be invariant under the action of $\{\varphi_{i,g}\ |\ g\in \Gs\}$ for all $i$; assume without loss of generality that it is not invariant under $\{\varphi_{1,g}\ |\ g\in \Gs\}$. Then, for a set $B\subset \Gs^{n-1}$ of positive measure, we have that for any $b\in B$, the set $\{g\in \Gs\ |\ (g,b)\in A\}$ is neither null nor conull in $\Gs$. Fix a point $(g_2,\dots,g_n)\in B$ such that $(g_2,g_3)\in U$, the set defined in the previous lemma. Note that $g_2$ can be thought of as an element of $\Aut(F_n)$ by identifying it with the map $(g_1,g_2,\dots,g_n)\mapsto(g_2g_1,g_2,\dots,g_n)$, and similarly for $g_3$. Because $\ip{g_2,g_3}$ is dense in $\Gs$, it acts ergodically on $\Gs$. The set
    $$A_1=\{g\in \Gs\ |\ (g,g_2,\dots,g_n)\in A\}$$
    is neither null nor conull in $\Gs$, but is also invariant under the action of $\ip{g_2,g_3}$, giving a contradiction.
\end{proof}

\subsection{Ergodicity and non-ergodicity results for subgroups}\label{compact-non-ergodic}

The previous sections have focused on the action of the mapping class group $\MCG(S)$ on the $\Gs$-character variety of the surface $S$. Another question many authors have considered is whether the full mapping class group is needed in order to have an ergodic action on the $\Gs$-character variety; in this section we highlight some of these results.

We begin with results that find a proper subgroup of $\MCG(S)$ that still acts ergodically on $\Xf(\pi_1(S),\Gs)$. A well-known subgroup of $\MCG(S)$ is the \textit{Torelli} subgroup $\mathcal{I}(S) < \MCG(S)$, which is the subgroup acting trivially on the surface's first homology group, see \cite{FM}. One can also consider the \textit{Johnson} subgroup $\mathcal{K}(S)$, which is the subgroup generated by Dehn twists along null homologous simple closed curves (or equivalently along separating simple closed curves). Note that $\mathcal{K}(S)$ is a normal subgroup of infinite index of $\mathcal{I}(S)$. Goldman--Xia \cite{GX-Torelli} studied the twice-punctured torus $S_{1,2}$ and considered the subgroup $\mathcal{J}(S_{1,2})<\MCG(S_{1,2})$. Fixing some $\Pc\in[-2,2]^2$, let $\Xf_{\Pc}(\pi_1(S_{1,2}),\SU(2))$ be the relative character variety consisting of representations $[\rho]$ such that the traces of the two peripheral curves around the punctures are $(\Tr(\rho(c_1)),\Tr(\rho(c_2)))=\Pc$. The action of the Hamiltonian flows of the trace functions of null homologous simple closed curves is locally transitive on an open, dense subset $U\subseteq\Xf_{\Pc}(\pi_1 ,\SU(2))$. However, the complement $V$ of this subset is of codimension 1, meaning it could cut $U$ into multiple connected components. To show that this does not happen, Goldman--Xia show, by explicit computation, that, for generic $\Pc\in[-2,2]^2$ and outside of a measure zero subset of $V$, there is a Hamiltonian flow of a trace function that is transverse to $V$, meaning that the flows can move across it and $V$ does not separate the character variety into distinct connected components. In a similar direction, Funar--March\'e \cite{funar-marche} showed that for $g\geq 2$ the subgroup $\mathcal{K}(S)$ acts ergodically on $\Xf(\pi_1(S_{g}),\SU(2))$ by computing the Taylor expansions of trace functions and showing that their differentials span the cotangent spaces $T_{[\rho]}^*\Xf(\pi_1(S_{g}),\SU(2))$, similar to the argument in \cite{GX}. Bouilly \cite{bouilly} gave a different proof of Funar--March\'e's result which also allows him to generalize the result to representations into general semi-simple compact Lie groups. Note that from the ergodicity of the action of the (first) Johnson subgroup $\mathcal{K}(S)$ and the fact that $\mathcal{K}(S)$ is a subgroup of the Torelli subgroup $\mathcal{I}(S)$, one can deduce the ergodicity of the action of the Torelli subgroup on the various character varieties. 

On the other hand, it is known that some subgroups of the mapping class group are in some sense `too small' to give an ergodic action on the character variety. Brown \cite{brown-anosov}, for $\Gs=\SU(2)$, and subsequently Forni--Goldman--Lawton--Matheus \cite{FGLM-nonerg}, for $\Gs=\SU(3)$, showed that there is a choice of boundary condition $\Pc$ and a pseudo-Anosov $[f]\in\MCG(S_{1,1})$ such that $[f]$ acts non-ergodically on $\Xf_{\Pc}(\pi_1(S_{1,1}),\Gs)$. To do this, they use KAM theory, see \cite{kam-theory}, to show that there is an open set which is neither null nor full measure in $\Xf_{\Pc}(\pi_1(S_{1,1}),\Gs)$ that is invariant under the action of the cyclic group $\ip{[f]}$. Saadi \cite{saadi-nonerg} studied a similar question for the $\SU(2)$-character varieties of two different surfaces, $N_4$ and $S_2$. For $N_4$, he shows that there is a family of filling curves whose Dehn twists generate a subgroup $\Gamma<\MCG(N_4)$ acting non-ergodically on $\Xf(\pi_1(N_4),\SU(2))$, and, moreover, that this group contains a pseudo-Anosov element. Similarly, for $S_2$ he shows that there is a pair of filling multi-curves whose Dehn twists generate a subgroup $\Gamma<\MCG(N_4)$ acting non-ergodically on $\Xf(\pi_1(S_2),\SU(2))$. Both of these statements are proven by finding examples of non-constant rational functions on the character variety which are invariant under the action of the subgroup $\Gamma$.

\section{Representations in \texorpdfstring{$\mathsf{PSL}_2(\Cb)$}{PSL(2,C)} and other rank one Lie groups}\label{PSL-C}

\subsection{Convex-cocompact representations}\label{CC}

In this section, we will introduce convex-cocompact representations in semi-simple Lie groups of real rank one, and then restrict to $\PSLC$ for most of the discussion. We fix $\Gs$ a rank one semi-simple Lie group and $X$ its associated symmetric space, that is, $X=\Gs/\Ks$, where $\Ks$ is a maximal compact subgroup of $\Gs$. The space $X$ is a Riemannian symmetric space of negative curvature. An example of particular importance is when $X$ is the $n$--dimensional real hyperbolic space $\Hb^n$. In this case $\Gs=\mathsf{PO}(n,1)$ and $\Ks=\mathsf{O}(n)$. For a complete classification of semi-simple Lie groups of real rank one, see Knapp \cite{Knapp}. 

\begin{definition}[Convex-cocompact representations]\mbox{}
\begin{itemize}
    \item We say that a subgroup $\Gamma$ of $\Gs$ is \emph{convex-cocompact}\index{Convex-compact representation} if there exists a non-empty convex subset of $X$ invariant by the action of $\Gamma$ and on which $\Gamma$ acts properly discontinuously with compact quotient. Note that this implies that the subgroup $\Gamma$ must be discrete. 
    \item We say that a representation $\rho \co\Gamma \to \Gs$ is \emph{convex-cocompact} if the subgroup $\rho(\Gamma)$ is convex-cocompact and if the kernel of $\rho$ is finite. 
\end{itemize}
\end{definition} 
In particular, a convex-cocompact representation $\rho$ is faithful whenever $\Gamma$ is torsion-free.

In this setting, there is a metric point of view on this notion which we now give. 
When $\Gamma$ is a finitely generated group, it can be endowed with a word metric by considering any finitely generating set of $\Gamma$. Let $o \in X$ and consider the orbit map $\tau_{\rho,o}\co \Gamma \to X$, which sends an element $\gamma$ in $\Gamma$ to $\rho(\gamma).o$ in $X$. Recall that a map $f \co(\Xc,d_{\Xc}) \to (\Yc,d_{\Yc})$ between two metric spaces is a \emph{quasi-isometric embedding} if there exist two constants $\lambda>0$ and $C\geq0$ such that for all pairs of points $x_1,x_2$ in~$\Xc$, $\frac{1}{\lambda}d_{\Xc}(x_1,x_2)-C \leq d_{\Yc}(f(x_1),f(x_2)) \leq \lambda d_{\Xc}(x_1,x_2)+C$. We say that $\rho\co\Gamma \to \Gs$ is a \emph{quasi-isometric embedding} if there exists $o \in X$ such that the orbit map $\tau_{\rho,o}$ is a quasi-isometric embedding. Observe that whether $\tau_{\rho,o}$ is a quasi-isometric embedding does not depend on the choice of the finite generating set of $\Gamma$ nor on the point $o$ in $X$. For a finitely generated group, the second inequality in the quasi-isometric embedding condition is automatically satisfied.

\begin{proposition}\label{prop:CC=QIE}
Let $\Gamma$ be a finitely generated group. Then a representation $\rho \co \Gamma \to \Gs$ is convex-cocompact if and only if it is a quasi-isometric embedding.
\end{proposition}

\begin{proof}
The direct implication is a consequence of the Švarc–Milnor lemma (Proposition I.8.19 in Bridson-Haefliger \cite{bridson-haefliger}). In fact, note that for this implication, we do not need to assume that $\Gamma$ is finitely generated as it is also a consequence of the lemma. 

For the reverse implication, the assumption that $\rho$ is a quasi-isometric embedding, together with the fact that $\Gamma$ is finitely generated, immediately gives the finiteness of the kernel of $\rho$. Now consider $\Lambda = \Lambda_{\Gamma}$, the limit set of $\Gamma$, which is defined as the set of accumulation points of an orbit of $\rho(\Gamma)$ in $\partial X$, that is,  $\Lambda=\overline{\rho(\Gamma)o} \cap \partial X$, for some $o \in X$. The limit set does not depend on the choice of $o$ in $X$. To define the convex set on which $\rho(\Gamma)$ acts, take the convex hull $\mathrm{CH}(\Lambda)$ of $\Lambda$ in $X$. It is a non-empty convex subset in $X$ on which $\rho(\Gamma)$ acts properly discontinuously since $\rho$ is a quasi-isometric embedding. To see that $\rho(\Gamma)$ acts cocompactly on $\mathrm{CH}(\Lambda)$, it suffices to verify that the convex hull $\mathrm{CH}(\Lambda)$ stays in a bounded neighborhood of the orbit $\rho(\Gamma)o$. This follows from two facts. The first fact is that $\mathrm{CH}(\Lambda)$ stays in a bounded neighborhood of the Gromov hull of $\Lambda$, which is defined as the union of all geodesics in $X$ with endpoints in $\Lambda$ (see \cite{bourdon}, this is a consequence of the fact that any point in a simplex is at uniform bounded distance to the edges of the simplex). The second fact is that the Gromov hull of $\Lambda$ stays in a bounded neighborhood of the orbit $\rho(\Gamma)o$ by the Morse lemma in hyperbolic spaces (see for example Theorem 1.7 in Bridson-Haefliger \cite{bridson-haefliger} or Theorem 3.1 in Coornaert-Delzant-Papadopoulos\cite{CDP}).
\end{proof}

\begin{remark}
    There are many other caracterizations of convex-cocompact representions in the rank-one setting. The interested reader can consult Bowditch \cite{bow_geo} and Theorem 1.36  in Kapovich-Leeb's survey \cite{KL-survey} for further details.
\end{remark}

The key observation is now that the quasi-isometric embedding property is stable in hyperbolic spaces. In addition, the constants of quasi-isometry can be chosen uniformly in a small neighborhood.

\begin{proposition}\label{prop:local-uniform-cst-qie}
Let $\rho : \Gamma \to \mathsf{G}$ be a quasi-isometric embedding. Then there exists an open neighborhood $\Uc$ of $\rho$ in $\mathrm{Hom}(\Gamma,\mathsf{G})$ and two constants $\lambda>0$ and $C\geq0$ such that every representation $\rho \in \Uc$ is a quasi-isometric embedding with constants $\lambda$ and $C$.
\end{proposition}

We refer to Canary's survey \cite{can_dyn} (Proposition 4.1) for a proof of this result. We emphasize on the fact that the key ingredient is the stability of quasi-geodesics in Gromov-hyperbolic spaces.

Proposition \ref{prop:local-uniform-cst-qie} immediately proves that the set of convex-cocompact representations is open. This fact was first due to Marden \cite{marden_geo} in the case $\mathsf{G}=\PSLC$ and was then generalized by Thurston \cite{thu_geometry} for $\mathsf{G}=\mathrm{PO}(n,1)$. Note that the notion of convex-cocompactness is invariant by conjugation by an element of $\mathsf{G}$, and hence is a well-defined notion in the character variety $\X(\Gamma,\mathsf{G})$. Moreover, this notion is also preserved by $\mathrm{Aut}(\Gamma)$ as a notion in $\mathrm{Hom}(\Gamma,\mathsf{G})$ or by $\mathrm{Out}(\Gamma)$ as a notion in $\X(\Gamma,\mathsf{G})$. Proposition \ref{prop:local-uniform-cst-qie} is also the crucial step to study the action of the outer automorphism group $\mathrm{Out}(\Gamma)$ of $\Gamma$ on the set of convex-cocompact representations.

\begin{Theorem} \label{thm:prop-discCC} 
	The action of $\mathrm{Out}(\Gamma)$ on the set of convex-cocompact representations in the character variety $\X(\Gamma,\mathsf{G})$ is properly discontinuous. 
\end{Theorem}

\begin{proof}
Let us start with some observations. If $\rho\co \Gamma \to \mathsf{G}$ is a $(\lambda,C)$--quasi-isometric embedding, then we have:
		\begin{equation} \label{eq:QI}
			\frac{1}{\lambda} | \gamma | -C \leq d(o,\rho(\gamma)o) \leq \lambda |\gamma|+C, \qquad \forall \gamma \in \Gamma.
		\end{equation}
		Now denote, for all $g \in G$, $$l(g):=\underset{x \in X}{\inf} \; d(x,gx)$$ and for all $\gamma \in \Gamma$, let $\Vert \gamma \Vert$ denote the cyclically reduced word length of $\gamma$.  
		Noticing that we have  $$\displaystyle l(\rho(\gamma)) = \nobreak \underset{n \to \infty}{\lim} \frac{1}{n} d(o,\rho(\gamma)^no)$$ and $$\displaystyle \Vert \gamma \Vert = \underset{n \to \infty}{\lim} \frac{1}{n} | \gamma^n|,$$ we deduce, applying Equation \eqref{eq:QI} to the element $\gamma^n$ and taking the limit, that:
		\begin{equation}\label{eq:WD}
			\frac{1}{\lambda} \Vert \gamma \Vert \leq l(\rho(\gamma)) \leq \lambda \Vert \gamma \Vert, \qquad \forall \gamma \in \Gamma.
		\end{equation}
		Note that the inequality \eqref{eq:WD} is invariant by conjugation.
		
		Let $K$ be a compact set consisting only of (conjugacy classes of) convex-cocompact representations, then the elements in $K$ are (conjugacy classes of) quasi-isometric embeddings. Moreover, since $K$ is compact, using Theorem \ref{thm:unif-QI}, we can choose uniform constants for the quasi-isometry for the elements in $K$ (choose a covering of $K$ by open sets consisting of uniform quasi-isometric embeddings and then extract a finite one). Let $(\lambda,C)$ be these uniform constants.

	Take $[\rho] \in K$  and $\Phi \in \mathrm{Out}(\Gamma)$ such that $\Phi.[\rho] \in K$. On the one hand, the representation $\rho$ is a $(\lambda,C)$--quasi-isometric embedding, so by Equation \eqref{eq:WD} we have:
	\begin{equation} \label{eq:rho}
		\frac{1}{\lambda} \Vert \gamma \Vert \leq l(\rho(\gamma)) \leq \lambda \Vert \gamma \Vert , \qquad \forall \gamma \in \Gamma.
	\end{equation}
	On the other hand, since a representative in $\Phi.[\rho]=[\rho \circ \Phi^{-1}]$ is also a $(\lambda,C)$--quasi-isometric embedding, again by \eqref{eq:WD}, we have: 
	\begin{equation*}
		\frac{1}{\lambda} \Vert \gamma \Vert \leq l(\rho (\Phi^{-1}(\gamma))) \leq \lambda \Vert \gamma \Vert, \qquad \forall \gamma \in \Gamma
	\end{equation*}
	which can be rewritten as:
	\begin{equation}\label{eq:Phi.rho}
		\frac{1}{\lambda} \Vert \Phi(\gamma) \Vert \leq l(\rho(\gamma)) \leq \lambda \Vert \Phi(\gamma) \Vert, \qquad \forall \gamma \in \Gamma.
	\end{equation}
	Note that there is an ambiguity while writing $\rho(\Phi^{-1}(\gamma))$ and $\Phi(\gamma)$, but since the lengths $l(\cdot)$ and  $\Vert \cdot \Vert$ are invariant by conjugation, the quantities $l(\rho (\Phi^{-1}(\gamma)))$ and  $\Vert \Phi(\gamma) \Vert$ are well-defined.
    
	Combining Equations \eqref{eq:rho} and \eqref{eq:Phi.rho} we deduce that:
	\begin{equation}\label{eq:bound-Phi}
	 \Vert \Phi(\gamma) \Vert \leq \lambda^2\Vert \gamma \Vert, \qquad \forall \gamma \in \Gamma.
	\end{equation}
	
	Now we use the following proposition of Canary \cite{can_dyn}:
	\begin{proposition}[Canary Proposition 2.3 \cite{can_dyn}] \label{prop:canary-EA}
	If $\Gamma$ is a torsion free hyperbolic group, then for any $A>0$, the set 
	\begin{equation*}
		E_A :=\{\Phi \in \mathrm{Out}(\Gamma) :  \Vert \Phi(\gamma) \Vert \leq A\Vert \gamma \Vert \quad \forall \gamma \in \Gamma \}
	\end{equation*}
	is finite. 
	\end{proposition}
	Note that the proposition \ref{prop:canary-EA} stated here is weaker than the one proved in \cite{can_dyn}. In fact, Canary only requires that the inequality defining the set $E_A$ be true for a finite collection of elements of $\Gamma$. 
	
	We can now conclude since Equation \eqref{eq:bound-Phi} gives a constant $A >0$ ($A=\lambda^2$), such that the set
	\begin{equation*}
		\{ \Phi \in \mathrm{Out}(\Gamma) : \Phi(K) \cap K \neq \emptyset\}
	\end{equation*}
	is contained in $E_A$ which is finite  by Proposition \ref{prop:canary-EA}. 
	\end{proof} 

When $S$ is a closed orientable surface, $\Gamma=\pi_1(S)$ and $\mathsf{G}=\PSLR$, the set of convex-cocompact representations in the character variety $\X(\pi_1(S),\PSLR)$ is nothing but the two connected components corresponding to the Teichmüller space of $S$, that is, $\mathrm{Teich}(S)$ and $\mathrm{Teich}(\overline{S})$, see Section \ref{o-surface-PSL-R}. 

When $\mathsf{G}=\PSLC$, and hence $X=\Hb^3$, the convex-cocompact representations of $\Gamma = \pi_1(S)$ correspond to the quasi-Fuchsian representations,\index{Quasi-Fuchsian representation} which play a fundamental role in the theory of Kleinian groups. Recall that the limit set $\Lambda$ of a representation $\rho \co\Gamma \to \PSLC$ is the set of accumulation points of an orbit of $\rho(\Gamma)$ in the boundary of $\Hb^3$, that is, $\Lambda=\overline{\rho(\Gamma)o}\cap \partial \Hb^3.$ (Note that this set does not depend on the choice of the orbit.) A representation $\rho \co\pi_1(S) \to \PSLC$ is called \emph{quasi-Fuchsian} if it is faithful and discrete and if its limit set $\Lambda$ in $\partial \Hb^3$ is a quasi-circle. The group $\rho(\Gamma)$ preserves its limit set $\Lambda$ and acts minimally on it, that is, every orbit in $\Lambda$ is dense. The `opposite' behavior happens, instead, on the complement, that is, $\rho(\Gamma)$ acts properly discontinuously on $\partial \Hb^3 \setminus \Lambda = \Omega$. The set $\Omega$ is called the domain of discontinuity for $\rho(\Gamma)$.

Let $\X_{QF}$ be the set of (conjugacy classes of) quasi-Fuchsian representations in the character variety $\X(\Gamma,\PSLC)$, and let $\X_{DF}$ be the set of (conjugacy classes of) discrete and faithful representations in $\X(\pi_1(S), \PSLC)$. (Recall that the set $\X_{QF}$ coincides with the set of (conjugacy classes of) convex-cocompact representations.) Sullivan \cite{sullivan} proved that $\X_{QF}$ is the interior of $\X_{DF}$ and Namazi--Souto \cite{namazi-souto} and Ohshika \cite{ohshika}, showed that the closure of $\X_{QF}$ is $\X_{DF}$. In addition, there are many results showing that the boundary of $\X_{QF}$ has a very complicated topology. For example, McMullen \cite{mcmullen_complex} showed that $\X_{QF}$ self-bump and Bromberg \cite{bro_sel} and Magid \cite{magid} proved that there are points in the boundary where $\X_{DF}$ is not locally connected. 

The dynamics of the mapping class group $\MCG(S)$ on the complement of $\X_{QF}$ is much more mysterious. A first observation is that the action of $\MCG(S)$ on the boundary $\partial \X_{QF}$ of $\X_{QF}$ cannot be properly discontinuous. Indeed, if $\Phi$ is a pseudo-Anosov mapping class in $\MCG$ (for a classification of the elements in the mapping class group, see Farb-Margalit \cite{FM}, for example), one can consider a $3$--manifold which fibers over the circle with fiber $S$ and monodromy $\Phi$. The manifold can be endowed with a finite volume hyperbolic structure by Thurston \cite{thurston2}, see also Otal \cite{otal-hyperbolisation}. This gives a representation $\rho\co\pi_1(S) \to \PSLC$ whose image in the character variety $\X(\pi_1(S), \PSLC)$ lies in the boundary of $\X_{QF}$ and is fixed under the action of $\Phi$. The existence of a fixed point for an infinite order mapping class on the boundary of $\X_{QF}$ prevents the action of $\MCG$ from being properly discontinuous on $\partial \X_{QF}$. Lee \cite{lee_thesis} furthermore proved that, in fact, no point on $\partial \X_{QF}$ can lie in an open set on which the action of $\MCG(S)$ is properly discontinuous (Proposition IV. 6 in \cite{lee_thesis}). This relies on a result of Minsky saying that no point on $\partial \X_{QF}$ corresponding to a compact hyperbolic manifold with an accidental parabolic can lie on an open domain of discontinuity (Lemma IV.4 in \cite{lee_thesis}) and on the fact that geometrically finite points are dense on the boundary of $\X_{QF}$. As a consequence, there is no connected open domain of discontinuity in $\X(\pi_1(S),\PSLC)$ strictly containing the set $\X_{QF}$ of quasi-Fuchsian representations. In fact, Goldman conjectured that $\X_{QF}$ is the maximal domain of discontinuity for the action of the mapping class group $\MCG(S)$ on $\X(\pi_1(S),\PSLC)$, see \cite{can_dyn}. Souto-Storm \cite{Souto-Storm} study the dynamics of $\MCG(S)$ on $\partial \X_{QF}$ from a topological viewpoint. They proved that a $\MCG(S)$-invariant continuous function on $\partial \X_{QF}$ must be constant. More precisely, they study the action on a nowhere dense topologically perfect subset $\overline{\mathcal{C}}$ of $\partial \X_{QF}$ which is contained in the orbit closures of points that form a dense subset in $\partial \X_{QF}$. They show that the action of $\MCG(S)$ on $\overline{\mathcal{C}}$ is topologically transitive, meaning that there exists a dense orbit. 

\subsection{Primitive-stable representations}\label{PS}

We now restrict our attention to the case of non abelian free groups and primitive-stable representations introduced by Minsky \cite{Minsky-primitive}. 

Let $F_n$ be a free group of rank $n$, with $n\geq 2$. We say that an element in $F_n$ is \emph{primitive} if it is contained in a free basis of $F_n$ and we denote $\mathcal{Pr}$ the set of primitive elements in $F_n$. Notice that $\mathcal{Pr}$ is invariant under conjugation and inversion. Fix $\Sigma$ a free generating set of $F_n$ and consider $\mathrm{Cay}$, the Cayley graph of $F_n$ with respect to the generating set $\Sigma$. The Cayley graph $\mathrm{Cay}$ comes equipped with a distance coming from the word metric associated to the generating set $\Sigma$. Every element $\gamma \in F_n \setminus \{\mathrm{id}\}$ has an \emph{axis} in $\mathrm{Cay}$, which is the unique $\gamma$--invariant geodesic in $\mathrm{Cay}$. We denote by $\mathrm{Ax}_{\mathcal{Pr}}$ the set of \emph{primitive geodesics} in $\mathrm{Cay}$, that is, the set of geodesics in $\mathrm{Cay}$ which are axes of some primitive element in $F_n$. 

Recall that we defined in section \ref{CC}, for a basepoint $o \in X$ and a representation $\rho \co \Gamma \to \Gs$, the orbit map $\tau_{\rho,o}$ that sends $\gamma \in \Gamma$ to $\rho(\gamma)o$. We can extend the orbit map defined on $\Gamma$ to the whole Cayley graph $\mathrm{Cay}$ by requiring that the map sends the edges of the graph to geodesic segments.  

\begin{definition}[Minsky \cite{Minsky-primitive}] \label{def:PS} Let $\rho \co F_n \to \PSLC$ be a representation of $F_n$ in $\PSLC$. We say that $\rho$ is \emph{primitive-stable}\index{Primitive-stable representation} if there exist two constants $\lambda >0$, $C\geq 0$ such that for all primitive geodesic $L \in \mathrm{Ax}_{\mathcal{Pr}}$, the orbit map $\tau_{\rho,o}$ restricted to $L$ is a $(\lambda,C)$--quasi-isometric embedding.
\end{definition}

Note that this definition does not depend on the choice of the basepoint $o \in \Hb^3$. The primitive-stability condition is also invariant by conjugation, hence it makes sense in the character variety. In addition, one can observe that the set of (conjugacy classes of) primitive-stable representations is invariant under the action of $\mathrm{Out}(F_n)$, since $\mathrm{Out}(F_n)$ preserves the set of (conjugacy classes of) primitive elements in $F_n$. It is clear, in view of Proposition \ref{prop:CC=QIE}, that all convex-cocompact representations are in particular primitive-stable. We denote by $\X_{PS}(F_n,\PSLC)$ the set of (conjugacy classes of) primitive-stable representations in the character variety $\X(F_n,\PSLC)$.  Minsky proved the following result:
\begin{Theorem}[Minsky \cite{Minsky-primitive}] \label{thm:PS}
	The set $\X_{PS}(F_n,\PSLC)$ is open in $\X(F_n, \PSLC)$, and the group $\mathrm{Out}(F_n)$ acts on it properly discontinuously. Moreover, the set $\X_{PS}(F_n,\PSLC)$ of (conjugacy classes of) primitive-stable representations strictly contains the set $\X_{CC}(F_n,\PSLC)$ of (conjugacy classes of) convex-cocompact representations. 
\end{Theorem}

In fact, Minsky constructed a primitive-stable representation on the boundary of the set of convex-cocompact representations. Together with the openness of the set of primitive-stable representations, this implies the existence of a non-discrete primitive-stable representation (since convex-cocompact representations form the interior of the set of faithful and discrete representations, see Sullivan \cite{sullivan}), and even a primitive-stable representation with dense image. This behavior is in great contrast with the closed surface case, where no point on the boundary of the set of convex-cocompact representations can lie in an open domain of discontinuity, see Section \ref{CC}. 

The proof of the openness of $\X_{PS}(F_n,\PSLC)$ and of the fact that $\mathrm{Out}(F_n)$ acts properly discontinuously on it in Theorem \ref{thm:PS} follows the same structure as the proof of the same results for convex-cocompact representations (see Proposition \ref{prop:local-uniform-cst-qie} and Theorem \ref{thm:prop-discCC}) by establishing an analogue of Proposition \ref{prop:local-uniform-cst-qie} for primitive-stable representations. The core of the work of Minsky is to prove the existence of a primitive-stable representation on the boundary of the set of convex-cocompact representations. He constructs his example as a discrete and faithful, geometrically finite representation with one cusp $c$ which corresponds to a blocking curve, namely a curve with some power that does not appear as a subword of a cyclically reduced primitive word. Because of this property, the geodesics representing the primitive elements will be forced to avoid the cusp and to stay in a compact region of the quotient manifold $\Hb^3/\rho(F_n)$, see Minsky \cite{Minsky-primitive}. 

The question of which points on the boundary of the set of convex-cocompact representations correspond to primitive-stable representations has been addressed by Jeon--Kim--Lecuire--Ohshika in \cite{JKLO}. They proved a conjecture of Minsky (in \cite{Minsky-primitive}) which says that if $\rho \co F_n \to \PSLC$ is discrete, faithful and without parabolics, then $\rho$ is primitive-stable, see Theorem 1.1 in \cite{JKLO}. Moreover, they proved the following geometric criterion: a discrete and faithful representation of $F_n$ in $\PSLC$ is primitive-stable if and only if in the quotient $\Hb^3/\rho(F_n)$ each component of the parabolic loci and each ending lamination is disk-busting, see Theorem 1.2 in \cite{JKLO}). Recall that a measured lamination $\lambda$ (or a simple closed curve) is called \textit{disk-busting} if there exists $\eta>0$ such that for any essential disk $A$,  we have that $i(\partial A, \lambda) > \eta$.

Furthermore, there is a larger set on which the action of $\mathrm{Out}(F_n,\PSLC)$ cannot be properly discontinuous. Let $\X_{PS}'(F_n,\PSLC)$ be the set of (conjugacy classes of) representations which are convex-cocompact on every proper free factor. The set of primitive-stable representations $\X_{PS}(F_n,\PSLC)$ is contained in $\X_{PS}'(F_n,\PSLC)$ (Lemma 3.2 (3) in \cite{Minsky-primitive}) and $\X_{PS}'(F_n,\PSLC)$ is contained in the complement of  the set $\X_{Rdn}(F_n,\PSLC)$ of redundant representations. A representation $\rho \in \X(F_n,\PSLC)$ is called \emph{redundant} if there exists a proper free factor $A$ of $F_n$ with $\rho(A)$ dense in $\PSLC$. In fact, $\X_{Rdn}(F_n, \PSLC)$ is dense in $\X(F_n,\PSLC) \setminus \X_{PS}'(F_n,\PSLC)$, see Lemma 5.1 (1) in \cite{Minsky-primitive}. Minsky proved that every domain of discontinuity must be contained in $\X_{PS}'(F_n,\PSLC)$, see Lemma 5.1 (2) in \cite{Minsky-primitive}. He asked if $\X_{PS}(F_n,\PSLC)$ is the interior of $\X_{PS}'(F_n,\PSLC)$ and if so, the set of primitive-stable representations would be the largest domain of discontinuity for the action of $\mathrm{Out}(F_n)$ on $\X(F_n,\PSLC)$. In addition, Gelander and Minsky \cite{gelander-minsky} proved a conjecture of Lubotzky \cite{lub-survey} that the action of $\mathrm{Aut}(F_n,\PSLC)$ on the set $\X_{Rdn}(F_n,\PSLC)$ of redundant representations is ergodic with respect to the Haar measure on $\X(F_n,\PSLC)$. See also Lubotzky's survey \cite{lub-survey}.

Lee extended Minsky's results in different ways. In \cite{lee-twisted} she considers representations from $\pi_1(M)$, where $M$ is a twisted interval bundle (that is, an interval bundle which is not a product). Note that in this case, this group $\pi_1(M)$ is the fundamental group of a non-orientable surface. This, in particular, shows substantial difference on the dynamics on the $\PSLC$--character varieties associated to orientable and non-orientable surface groups. 
On the other hand, in \cite{lee_dyn} she focuses on representations of fundamental groups of compression bodies $M$ (without toroidal boundary components). In that case, the fundamental group $\pi_1(M)$ is the free product of a free group and a finite number of closed orientable surface groups. 

\begin{theorem}[Lee \cite{lee-twisted, lee_dyn}]
    Let $M$ be a hyperbolizable twisted $I$--bundle over a non-orientable hyperbolic surface or let $M$ be a nontrivial hyperbolizable compression body without toroidal boundary components. Then there exists an open, $\Out(\pi_1(M))$--invariant subset $\X(\pi_1(M), \PSLC)$ in $\X(\pi_1(M), \PSLC)$ strictly containing the interior of the set $\X_{DF}(\pi_1(M), \PSLC)$ and on which $\mathrm{Out}(\pi_1(M))$ acts properly discontinuously.
\end{theorem}

In the first case, the set used by Lee in her proof is the set of primitive-stable representations, while in the second case, she considers the set of separable-stable representations. Recall that if $M$ is a compression body that is not the connected sum of two trivial $I$--bundles over closed surfaces, an element in $\pi_1(M)$ is called \textit{separable} if it corresponds to a loop in $M$ that can be freely homotoped to miss an essential disk. When $M$ is the connected sum of two trivial $I$--bundles over closed surfaces, instead, an element in $\pi_1(M)$ is called separable if it corresponds to a loop in $M$ that can be freely homotoped to miss an essential annulus contained in one of the two trivial interval bundles. 

In addition, in \cite{lee-twisted} and in \cite{kim-lee} (jointly with Kim for compression bodies), she also characterized exactly which points in the boundary of $\X_{DF}(\pi_1(M), \PSLC)$ are contained in the above sets. Since the work on compression bodies is outside the scope of this chapter (since $\pi_1(M)$ in that case is not a surface or a free group), we focus on the case of twisted interval bundles. 

\begin{theorem}[Lee \cite{lee-twisted}]
    Let $M$ be a hyperbolizable twisted $I$--bundle over a non-orientable hyperbolic surface, and let $\rho \in \X_{DF}(\pi_1(M, \PSLC)$. Then $\rho$ lies in the complement of the set of primitive-stable representations $\X_{PS}(M)$ if and only if there exists a primitive element $g \in \pi_1(M)$ with parabolic image. In addition a representation in $\X_{DF}(M) \setminus \X_{PS}(M)$ does not lie in any domain of discontinuity for the action of $\mathrm{Out}(\pi_1(M))$ on $\X(M)$. 
\end{theorem} 

The dynamics of $\mathrm{Out}(\pi_1(M))$ on the associated $\PSLC$--character variety for more general hyperbolizable $3$--manifolds (both with incompressible and with compressible boundary) is discussed in Canary \cite{can_dyn}. In particular, there is important work from Canary--Storm \cite{canary-storm}.\\

We end this section by mentioning that there is a natural analogue of primitive-stability for surface groups $\pi_1(S)$ (possibly with punctures) by considering essential simple closed curves on $S$ instead of primitive elements. The reader can easily convince himself that Definition \ref{def:PS} makes sense when replacing the set of primitive geodesic axes in the Cayley graph by the set of geodesic axes in the Cayley graph corresponding to essential simple closed curves on the surface $S$. We call the representations satisfying this property \emph{simple-stable} representations\index{Simple-stable representation}. Since the set of free homotopy classes of unoriented essential simple closed curves on $S$ is invariant under the mapping class group, we obtain a $\MCG(S)$-invariant subset in the character variety, which we denote by $\X_{SS}(\pi_1(S),\PSLC)$. By a similar proof than the one for primitive-stable representations (Theorem \ref{thm:PS}), one can prove that $\X_{SS}(\pi_1(S),\PSLC)$ is open and that $\MCG(S)$ acts properly discontinuously on it. Moreover, $\X_{SS}(\pi_1(S),\PSLC)$ contains $\X_{CC}(\pi_1(S),\PSLC)$, the set of (conjugacy classes of) convex-cocompact representations. When $S$ is a closed orientable surface, in light of Lee's result (Proposition IV.6 in \cite{lee_thesis} mentioned at the end of Section \ref{CC}) which ensures that no point on the boundary of the quasi-fuchsian space can lie in an open domain of discontinuity, we deduce that $\X_{SS}(\pi_1(S),\PSLC) \cap \partial \X_{CC}(\pi_1(S),\PSLC) = \emptyset$, which contrasts with the result of Minsky for free groups in Theorem \ref{thm:PS} (In fact, one could also deduce this fact directly from the classification of Kleinian groups). 

Remfort-Aurat considered in \cite{remfort-aurat} simple-stable representations of closed surface groups in $\mathsf{PU}(2,1)$, which is the isometry group of the complex hyperbolic plane $\Hb_{\Cb}^2$. He proved \cite{remfort-aurat} in this setting that, unlike in the $\PSLC$ case, $\X_{SS}(\pi_1(S),\mathsf{PU}(2,1))$ contains points on the boundary of $\X_{CC}(\pi_1(S),\mathsf{PU}(2,1))$, as well as points corresponding to representations which are not discrete and faithful. To do so, he starts from a discrete, faithful, and geometrically finite representation of a triangle group in $\mathsf{PU}(2,1)$ which contains a parabolic element known explicitly. He then considers the restriction of this representation to a closed surface subgroup and proves that the parabolic element corresponds to a non-simple closed curve in this surface group. Finally, he proves that the representation of the surface subgroup must be simple-stable, but it cannot be convex-cocompact, in view of the existence of the parabolic element. 

\subsection{Bowditch representations}\label{Bow}

In this section we discuss a theory inspired by the foundational work of Bowditch \cite{bow_mar} on the study of the action of the mapping class group of punctured surfaces on the associated $\mathsf{PSL}_2(\mathbb{C})$--character varieties. We will then discuss more recent results related to these representations that have connections with representations discussed in the previous section. In particular, we will mostly focus on the following results of:
\begin{enumerate}
    \item Bowditch \cite{bow_mar} on representations for the once-punctured torus $S_{1,1}$;
    \item Tan--Wong--Zhang \cite{TWZ, TWZ-corr} on representations for the one-holed torus $S_{1,1}$. 
    \item Maloni--Palesi--Tan \cite{MPT} on representations of the four-holed sphere $S_{0,4}$;
    \item Maloni--Palesi \cite{MP} on representations for the three-holed projective plane $N_{1,3}$. 
\end{enumerate}

In the rest of this section we will consider $S = S_{1,1}$, $S= S_{0,4}$, or $S= N_{1,3}$. Note that the associated fundamental group $\pi_1(S)$ is the non-abelian free group $F_2$, in the first case, and $F_3$, in the other two cases. We consider the character variety $\mathfrak{X}(S)= \mathfrak{X}(\pi_1(S), \mathsf{PSL}_2(\mathbb{C}))$. In all these cases the mapping class groups $\MCG(S)$ is a subgroup of the outer automorphism group $\mathrm{Out}(\pi_1(S))$, and as such it acts on the character variety $\mathfrak{X}(S)$. The main result of this section will be the description of a domain of discontinuity for such actions. In particular we will discuss the following result. 
\begin{theorem}[Bowditch \cite{bow_mar}, Tan--Wong--Zhang \cite{TWZ}, Maloni--Palesi--Tan \cite{MPT}, Maloni--Palesi \cite{MP}]\label{bow}
Given a surface $S \in \{S_{1,1}, S_{0,4}, N_{1,3}\}$, there exist open and $\MCG(S)$--invariant subsets $\X_{BQ}(S)$ in $\mathfrak{X}(S)$ on which the mapping class group $\MCG(S)$ acts properly discontinuously and which contain the sets $\X_{PS}(\pi_1(S))$ of primitive-stable representations discussed in the previous section. In addition,  for $S \in \{S_{0,4}, N_{1,3}\}$ this last containment is strict.
\end{theorem}

The relationship between the set of Bowditch representations $\X_{BQ}(S_{1,1})$ and the set of primitive-stable representations in the case of $S=S_{1,1}$ will be discussed in Section \ref{bow-PS}.

In order to explain the proof of Theorem \ref{bow}, we have to first recall the structure of a few complexes associated with these surfaces.  

\subsubsection{Curves in \texorpdfstring{$S_{1,1}$}{S11}, \texorpdfstring{$S_{0,4}$}{S04} and \texorpdfstring{$N_{1,3}$}{N13}}\label{curve-complexes}

\begin{figure}
[hbt] \centering
\includegraphics[height=3 cm]{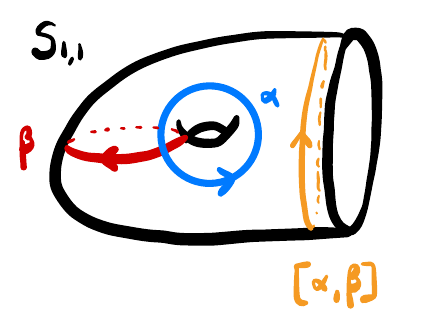}
\includegraphics[height=3.5 cm]{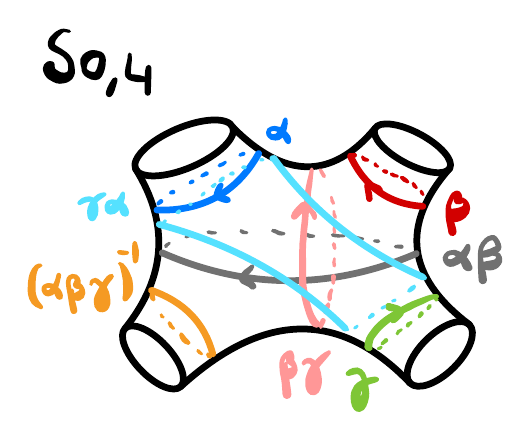}
\includegraphics[height=3.5cm]{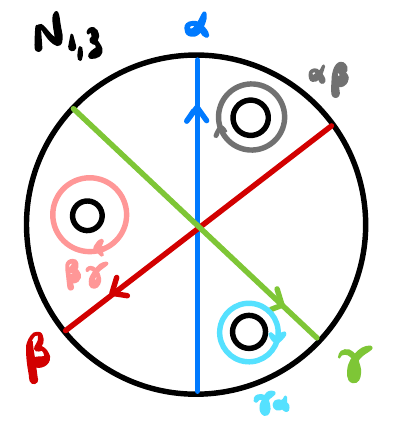}
\caption{On the left, $S = S_{1,1}$ and the curves $\alpha$, $\beta$, and $[\alpha, \beta]$. On the left, $S = S_{0,4}$ and the curves $\alpha$, $\beta$, $\gamma$, $(\alpha\beta\gamma)^{-1}$, $\alpha\beta$, $\beta\gamma$, and $\gamma\alpha$. On the right, $S = N_{1,3}$ and the curves $\alpha$, $\beta$, and $\gamma$.} 
\label{fig:fig-surf}
\end{figure}

Let $S \in \{S_{1,1}, S_{0,4}, N_{1,3}\}$. The fundamental group $\Gamma = \pi_1(S)$ is isomorphic to the non-abelian rank--$2$ free group $F_2:=\langle \alpha, \beta \rangle$ (if $S = S_{1,1}$), or to the non-abelian rank--$3$ free group $F_3:=\langle \alpha, \beta, \gamma \rangle$ (if $S = S_{0,4}$ or $N_{1,3}$). Abusing notation, we will use the following conventions:
\begin{itemize}
	\item[$S = S_{1,1}$] Let $\alpha$ and $\beta$ represent the (free homotopy class of) loops associated with the two generators so that the peripheral curve $\gamma$ going around the boundary component of $S$ is oriented so as to be homotopic to $[\alpha,\beta]:=\alpha\beta\alpha^{-1}\beta^{-1}$. 
	\item[$S = S_{0,4}$] Let $\alpha$, $\beta$, $\gamma$ represent the (free homotopy class of) loops associated with the three generators so that the peripheral curves around the boundary component of $S$ are oriented so as to be homotopic to $\{\alpha,\beta, \gamma, (\alpha\beta\gamma)^{-1}\}$. 
	\item[$S = N_{1,3}$] Let $\alpha$, $\beta$, $\gamma$ represent the (free homotopy class of) $1$--sided loops associated with the three generators so that the peripheral curves around the boundary component of $S$ are oriented so as to be homotopic to $\{\alpha\beta, \beta\gamma, \gamma\alpha\}$. 
\end{itemize} 

Let $\Sc = \Sc(S)$ be the set of free homotopy classes of simple closed essential unoriented curves on $S$. Recall that we call a curve \emph{essential} if it is non-trivial and non-peripheral, or equivalently if it does not bound a disk, or a one-holed disk. We will generally omit the word essential from now on. Note that one can identify $\Sc$ with a well-defined subset of $\Gamma/\!\sim$, where the equivalence relation $\sim$ on $\Gamma$ is defined as follows: $g \sim h$ if and only if $g$ is conjugate to $h$ or $h^{-1}$. Simple closed curves in a non-orientable surface are of two types: $1$--{\em sided} if its tubular neighborhood is homeomorphic to a M\"obius strip, and $2$--{\em sided} if its tubular neighborhood is homeomorphic to an annulus. We let $\Sc_i$, where $i = 1, 2$, be the subset of $\Sc$ corresponding to $i$--sided simple closed curves, so that we have, for example, $\Sc(N_{1,3}) = \Sc_1(N_{1,3})\cup \Sc_2(N_{1,3})$.

\subsubsection{Simplicial complexes associated with \texorpdfstring{$\mathcal{S}({S})$}{S(S)}}\label{ss:simple}

In this section, we will recall the definition of a few complexes that will be used in the description of the proof sketch. See Farb--Margalit \cite{FM} for a discussion on the complex of curves of orientable surfaces, and Scharlemann \cite{sch_the} for a more detailed discussion on the complex of curves of non-orientable surfaces.

\begin{definition} \mbox{}
\begin{itemize}
    \item The curve complex $\mathcal{C} = \Cc(S)$ of $S \in \{S_{1,1}, S_{0,4}\}$ is the $2$--dimensional abstract simplicial complex, defined by setting $k$--simplices to be subsets of $k+1$ distinct (free homotopy classes of) essential simple closed curves in $S$ that pairwise intersect once for $S_{1,1}$ or twice for $S_{0,4}$.
    \item The complex of curves $\mathcal{C} = \mathcal{C}(N_{1,3})$ of $N_{1,3}$ is the $3$--dimensional abstract simplicial complex, where the $k$--simplices are given by subsets of $k+1$ distinct (homotopy classes of) $1$--sided simple closed curves in $N$ that pairwise intersect once.
\end{itemize}
\end{definition}
 
Recall (see \cite[Remark 2.1]{MP}) that in $N_{1,3}$, there is a $1$--to--$1$ correspondence between:
  \begin{itemize}
  	\item (unordered) pairs $(\alpha, \beta)$ of (free homotopy classes of) $1$--sided simple closed curves intersecting exactly once, that is, edges in $\mathcal{C}(N_{1,3})$; and 
	\item (free homotopy classes of) $2$--sided simple closed curves $\xi_{\alpha, \beta}$.
\end{itemize} 
This correspondence comes from the fact that the $\epsilon$--neighborhood of any pair of $1$--sided simple closed curves intersecting once corresponds to an embedded two holed projective plane $M$, where one of the boundary components of $M$ is homotopic to a boundary component of $N$, and such that the other boundary curve is $2$--sided curve, and corresponds to the element $\alpha\beta^{-1}$. 

Note that in the case of $S \in \{S_{1,1}, S_{0,4}\}$, the curve complex $\mathcal{C}$ is isomorphic to the Farey complex. The identification depends on a choice of generators for $\pi_1(S)$. Since we will not need this identification in the rest of the chapter, we will not discuss the details. 

\begin{definition}
Let $\mathcal{T}$ be the simplicial complex defined by letting the set of $k$--simplices in $\mathcal{T}$, denoted by $\mathcal{T}^{(k)}$, be the set $\mathcal{T}^{(k)} = \mathcal{C}^{(dim(\mathcal{C})-k)}$.  The complex $\mathcal{T}$ is called the simplicial dual of $\mathcal{C}$ and is a countably infinite simplicial tree properly embedded in  hyperbolic $2$--space $\Hb^{dim(\mathcal{C})}$ whose vertices all have degree $dim(\mathcal{C}) +1$. This means that $\mathcal{T}$ is a $3$--valent tree for $S = S_{1,1}, S_{0,4}$, and a $4$--valent tree for $S = N_{1,3}$. 
\end{definition}

\begin{figure}
[hbt] \centering
\includegraphics[height=6 cm]{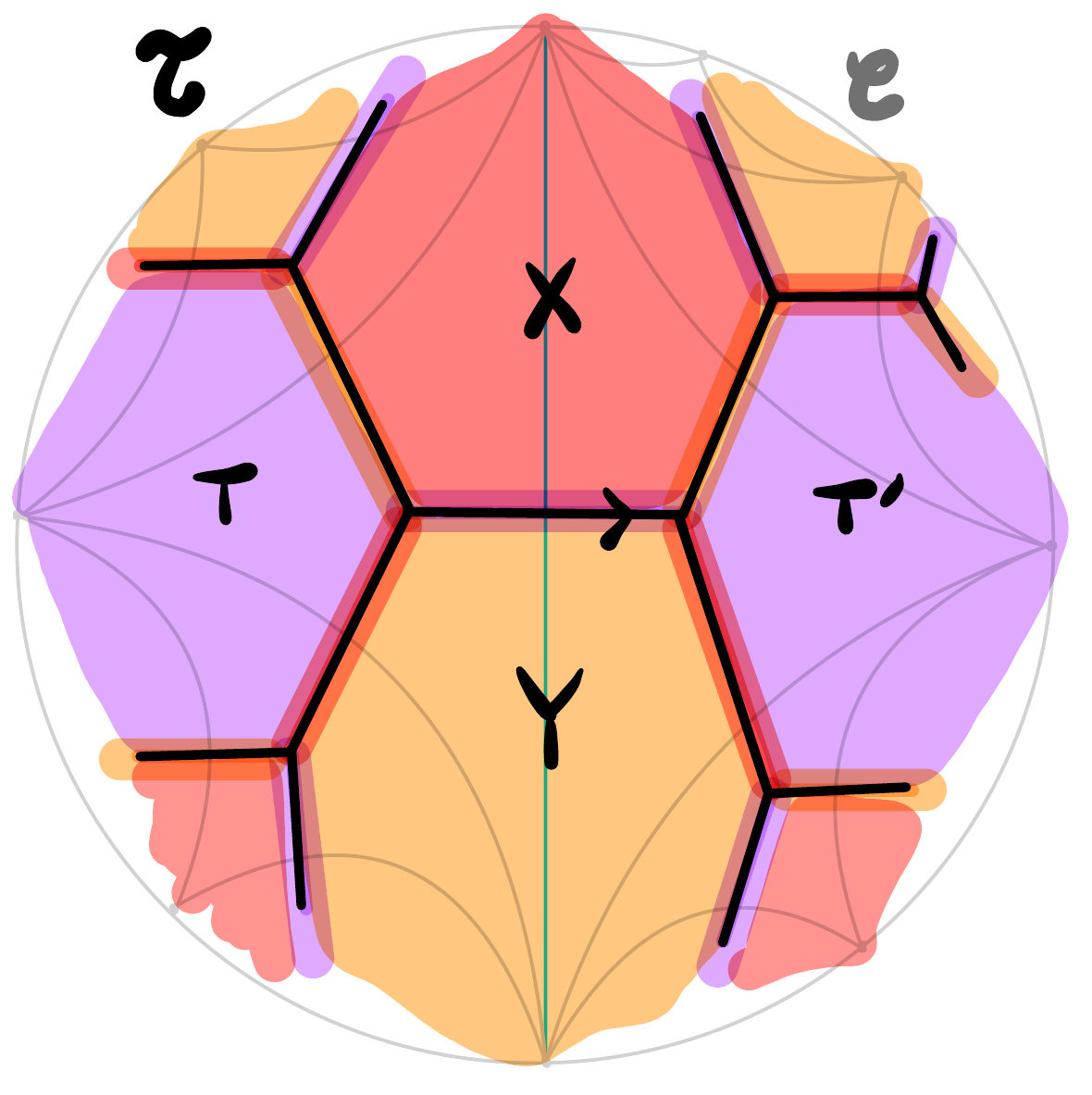}
\caption{The complexes $\mathcal{C}(S)$ (in grey) and $\mathcal{T}(S)$ (in black) for $S = S_{1,1}, S_{0,4}$. The coloring of $\mathcal{T}^{(1)}$ and for the regions in $\Omega$ and an orientation for an edge $e = X \cap Y$.}
\label{fig:f-t}
\end{figure}

Figure \ref{fig:f-t} illustrates $\mathcal{C}$ and $\mathcal{T}$ for $S = S_{1,1}$. Note that the graph $\mathcal{T}$ can be identified with the graph of (certain) ideal triangulations of $S$, with edges corresponding to edge-flips. Again, since this is not needed in the rest of the paper, we will not discuss the details of this correspondence. 

\subsubsection{Relative character varieties}\label{rel-char}

A classical result on the character varieties (see, for example, Fricke and Klein \cite{fri_vor}) states that the maps
\begin{align*}
	f_2 : \mathfrak{X}(F_2, \mathsf{SL}_2(\mathbb{C})) & \longrightarrow \, \, \, \mathbb{C}^3 \\
		[ \rho ] & \longmapsto 
		\begin{pmatrix} a \\ b \\ c \end{pmatrix} = 
		\begin{pmatrix} \mathrm{Tr}( \rho (\alpha)) \\ \mathrm{Tr}( \rho (\beta))\\ \mathrm{Tr}( \rho (\alpha\beta))) 				\end{pmatrix}
\end{align*}
and 
\begin{align*}
	f_3 : \mathfrak{X}(F_3, \mathsf{SL}_2(\mathbb{C})) & \longrightarrow \, \, \, \mathbb{C}^7 \\
		[ \rho ] & \longmapsto 
		\begin{pmatrix} a \\ b \\ c \\ d \\ x \\ y \\ z \end{pmatrix} = 
		\begin{pmatrix} \mathrm{Tr}( \rho (\alpha)) \\ \mathrm{Tr}( \rho (\beta))\\ \mathrm{Tr}( \rho (\gamma))\\ \mathrm{Tr}( \rho (\alpha\beta\gamma))\\ \mathrm{Tr}( \rho (\alpha\beta))\\ \mathrm{Tr}( \rho (\beta\gamma))\\ \mathrm{Tr}( \rho (\alpha\gamma)) 				\end{pmatrix}
\end{align*}
 provides an identification of the space $\mathfrak{X}(F_2, \mathsf{SL}_2(\mathbb{C})$ with $\mathbb{C}^3$ and an identification of the space $\mathfrak{X}(F_3, \mathsf{SL}_2(\mathbb{C})$ with the set
    \begin{equation} 
        \left\{(a,b,c,d, x,y,z) \in \mathbb{C}^7 \mid \mbox{equation } \eqref{vertex} \mbox{ holds} \right\},
    \end{equation}
where
\begin{equation} \label{vertex}
  a^2+b^2+c^2+d^2+abcd = x(ab+cd)+y(bc+ad)+z(ac+bd)+4-x^2-y^2-z^2-xyz.
\end{equation}
Recall that if $\rho \in \mathfrak{X}(F_2, \mathsf{SL}_2(\mathbb{C})$, we have 
\begin{equation*} \label{vertex_11}
  \mathrm{Tr}( \rho([\alpha,\beta])) = a^2+b^2+c^2-abc-2.
\end{equation*}

The relative character varieties correspond to the set of (classes of) representations for which the traces of the boundary curves are fixed. In particular, these are defined as:
\begin{itemize}
	\item[$S = S_{1,1}$] Let $\tau \in \Cb$, then  $$\mathfrak{X}_{\tau}(S) = \{ (a,b,c)\in \Cb^3 \mid a^2+b^2+c^2-abc-2 - \tau\}.$$ 
	\item[$S = S_{0,4}$] Let $(a, b, c, d) \in \Cb^4$, then we can define $$\mathfrak{X}_{(a, b, c, d)}(S) = \{ (x, y, z)\in \Cb^3 \mid \mbox{equation }\eqref{vertex} \mbox{ holds} \}.$$
	\item[$S = N_{1,3}$] Let $(x,y,z) \in \Cb^3$, then we can define $$\mathfrak{X}_{(x,y,z)}(S) = \{ (a,b, c, d)\in \Cb^4 \mid \mbox{equation }\eqref{vertex} \mbox{ holds} \}.$$
\end{itemize} 

To uniformize the notation we will write $\X_{\mathcal{P}}(S)$ to denote the above relative character varieties: 
\begin{itemize}
	\item $\mathfrak{X}_{c}(S)$ for $c \in \Cb$.
	\item $\mathfrak{X}_{(a, b, c, d)}(S)$ for $(a, b, c, d) \in \Cb^4$.
	\item $\mathfrak{X}_{(x,y,z)}(S)$ for  $(x,y,z) \in \Cb^3$.
\end{itemize} 

\subsubsection{Definition of \texorpdfstring{$\X_{BQ}(S)$}{XBQ(S)} and different characterizations}\label{characterizations}

We can now define the set $\X_{BQ}(S)$ for $S = S_{1,1}, S_{0,4}, N_{1,3}$. In fact, we will define this set for each relative character variety $\X_{BQ, \mathcal{P}}(S) \subset \X_{\mathcal{P}}(S)$ and then define $\X_{BQ}(S)$ to be the set obtained as the union over all the boundary values $\mathcal{P} \in \Cb^n$ for $n$ equal to the number of boundary components. 

Recall that if $k = 1,2$, then the set $\Sc_k = \Sc_k(S)$ is the set of free homotopy classes of essential unoriented $k$--sided simple closed curves in $S$. 

In the definition, we will need to define a value $M = M(\mathcal{P})$ which is the following
\begin{itemize}
     \item[ $S = S_{1,1}$]  $\mathcal{P} = \tau \in \Cb$ and $M = M(\tau) = 2$.
    \item[$S= S_{0,4}$] $\mathcal{P} = (a,b, c, d) \in \Cb^4$ and $M = M\left((a,b,c,d)\right) = 2 +\frac{1}{2} \mathrm{max}(|ab+cd|, |bc+ad|, |ac+bd|)$, see Section 3 in \cite{MPT}.
    \item[\;$S = N_{1,3}$] $\mathcal{P} = (x,y,z) \in \Cb^3$ and $M = M\left((x,y,z)\right) = 2 + \mathrm{max}(|x|, |y|, |z|)$, see Section~3 in \cite{MP}.
\end{itemize}

We also define for $K >0$ the set $\mathcal{S}_2(K)$ as follows:
\begin{itemize}
    \item For $S= S_{1,1}, S_{0,4}$, then $$\mathcal{S}_2(K) := \left\{\gamma \in \mathcal{S} \mid |\Tr\rho(\gamma)| \le K\right\},$$
    \item For $S= N_{1,3}$, then
$$\mathcal{S}_2(K)  := \left\{\xi_{\alpha, \beta} \in \mathcal{S}_2 \mid (|\Tr\rho(\alpha)| \le K \text{ or } |\Tr\rho(\beta)| \le K )\text{ and } |\Tr\rho(\xi_{\alpha, \beta})| \le K^2+ M(\mathcal{P}) -2\right\},$$
\end{itemize}

We will also need to define an \textit{auxiliary function}:
$$\sigma\co \mathcal{S}_2(S) \to \Cb.$$
The precise definition is a bit technical and, since equation \eqref{vertex} is not symmetric in its entries, depends on the `coloring' of the curves as defined in the following section, but as an example, we can see the following definition for certain choices of elements in $\mathcal{S}_2(S)$ (that is, the ones corresponding to the role played by $\alpha$ for $S = S_{1,1}, S_{0,4}$ or the one played by the curve $\alpha\beta^{-1}$ for $S = N_{1,3}$):
$$\left\{
     \begin{array}{lr}
     \sigma(x) = x^2 - (\tau+2) & \text{ if } x \in \mathcal{S}_2(S_{1,1})\\
      \sigma(x) = \left(x^2+a^2+b^2-abx-4\right) \left(x^2+c^2+d^2-cdx-4\right) & \text{ if } x \in \mathcal{S}_2(S_{0,4})\\
      \sigma(\xi_{a,b}) = \left(x^2+a^2+b^2-abx-4\right) \left(y^2+z^2+(ab-x)^2-yz(ab-x)-4\right) & \text{ if } x \in \mathcal{S}_2(N_{1,3}) \\
     \end{array}
   \right.$$
The zeros of this function are related to representations that,
restricted to certain subsurfaces, are reducible.

\begin{definition}[Bowditch \cite{bow_mar}, Tan--Wong--Zhang \cite{TWZ}, Maloni--Palesi--Tan \cite{MPT}, Maloni--Palesi \cite{MP}] \label{def:BQ-PSLC}
    A representation $\rho$ is in the set $\X_{BQ, \mathcal{P}}(S)$ if it satisfies the following conditions:
    \begin{enumerate}
  \item[{(BQ1)}] $\forall \gamma \in \Sc$, $\rho(\gamma)$ is loxodromic.
  \item[{(BQ2)}] $\# \mathcal{S}_2(M)< \infty$, where $M = M(\mathcal{P})> 0$ is defined as above.
  \item[{(BQ3)}] (Only for $S = S_{0,4}, N_{1,3}$) For all $\gamma \in \mathcal{S}_2$, we have that $\sigma(\gamma) \neq 0$. 
  \end{enumerate}
  A representation $\rho \in \X_{BQ, \mathcal{P}}(S)$ is called a \textit{Bowditch}\index{Bowditch representation} representation.
\end{definition}

We will now state various definitions that turn out to be equivalent to the definition given above, but we need to fix some notation first. Given $\rho$ in the character variety $\mathfrak{X}$, we can define the function $\mathrm{L}_\rho = \mathrm{L}(\rho(\cdot)) \co \Sc(S) \to \Cb$, well-defined modulo $2i\pi$ by 
\begin{equation}\label{eq:trace-length}
\Tr(\rho(\gamma)) = 2 \mathrm{cosh}(\mathrm{L}(\rho(\gamma))/2).
\end{equation}
The function $\mathrm{L}_\rho$ is the complex translation length function, and its real part corresponds to the translation length. Also, let $\mathrm{W}$ denote the minimal cyclically reduced word length with respect to some generating set. We can define the following properties:
\begin{enumerate}
  \item[{(BQ2')}] $ \forall K>0$, $\; \#\mathcal{S}_2(K)  < \infty.$
  \item[{(BQ4)}] There exists $k, m >0$ such that $|l (\rho (\gamma))| \geq k W(\gamma) - m$ for all $\gamma \in \Sc$.
  \item[{(BQ5)}] $\forall K\geq M(\mathcal{P})$, the attracting subgraph $T_\rho (K)$ (which we will define in the next section) is finite.
  \end{enumerate}
  
As a by-product of the proof, one can also prove various equivalent characterizations of the definition above:
\begin{theorem}[Bowditch \cite{bow_mar}, Tan--Wong--Zhang \cite{TWZ}, Maloni--Palesi--Tan \cite{MPT}, Maloni--Palesi \cite{MP}]\label{Main-characterization} 
Let $\rho\in\X_{\mathcal{P}}$. The following are equivalent:
\begin{enumerate}
\item[$(1)$] The representation $\rho$ satisfies $(BQ1)$,  $(BQ2)$ and $(BQ3)$;
\item[$(2)$] The representation $\rho$ satisfies $(BQ1)$ and $(BQ2')$;
\item[$(3)$] The representation $\rho$ satisfies $(BQ4)$;
\item[$(4)$] The representation $\rho$ satisfies $(BQ5)$.
\end{enumerate}
\end{theorem}

Note that in the case of $N= N_{1,3}$, condition $(BQ3)$ is also equivalent to $(BQ4')$ or $(BQ4'')$, where
\begin{enumerate}
   \item[{(BQ4')}] There exists $k, m >0$ such that $|l (\rho (\gamma))| \geq k W(\gamma) - m $ for all $\gamma \in \Sc_1$.
   \item[{(BQ4'')}] There exists $k, m >0$ such that $|l (\rho (\gamma))| \geq k W(\gamma) - m $ for all $\gamma \in \Sc_2$.
\end{enumerate} 

\subsubsection{Sketch of the main proof}\label{proof}

We now give an overview of the proof of Theorem \ref{bow}. For simplicity, we will focus on the case $S = S_{0,4}$ but the main ideas can be adapted to the other cases. We will point out how to adapt the arguments in the other cases, especially for $S=N_{1,3}$ as that has some extra complications. In fact, the case $S= S_{1,1}$ is simpler since equation \eqref{vertex_11} is symmetric in the entries $(a,b,c)$, differently from equation \eqref{vertex}. The case $S= N_{1,3}$ is slightly more technical since we have a different role played by $1$--sided and $2$--sided curves, but we decided not to go into these technical details here. 

For the proof we will need to consider the set $\Omega:=\pi_0(\mathbf{H}^2- \mathcal{T})$ of connected components of the space $\mathbf{H}^2 - \mathcal{T}$, which is in one-to-one correspondence with the set of vertices of $\mathcal{C}$ and hence with the set $\Sc$; that is, we have $$\Omega \cong \mathcal{C}^{(0)} \cong \Sc.$$ We will denote regions in $\Omega$ with capital letters, for example $X \in \Omega$. 

The proof can be organized into six main steps:
\begin{enumerate}
	\item[(1)] orientation (and coloring) on edges of $\mathcal{T}$, 
    \item[(2)] fork lemma and connectivity of the set of regions with small trace;
	\item[(3)] growth of traces for neighbors around a region and behavior of neighboring regions of escaping rays;
   \item[(4)] definition of an attracting subgraph of $\mathcal{T}$;
	\item[(5)] Fibonacci growth and equivalent characterizations of Bowditch representations;
	\item[(6)] proof that the set $\X_{BQ}$ strictly contains primitive-stable representations when $S = S_{0,4}$ or $S = N_{1,3}$. 
\end{enumerate}

\begin{enumerate}
\item[Step (1):] We have a $3$--coloring of the edges of $\mathcal{T}(S_{0,4})$ and a $4$-coloring of the edges of $\mathcal{T}(N_{1,3})$ such that at each vertex every edge has a different color. This coloring extends to regions as explained in Figure \ref{fig:f-t}. This coloring reflect the fact that Equation \eqref{vertex} is not symmetric with respect to its variables. For the first step, we notice that any representation $\rho\in \X_{\mathcal{P}}$ assigns to each region $X \in \Omega$ a complex value $x := \mathrm{Tr} (\rho (X)) \in \Cb$. We will always use the convention that the image of this map will be denoted with the corresponding lower case letter. 

With this, we can then define an orientation on the edges of $\mathcal{T}$ as follows. Let $e = X \cap Y$ be an edge  in $\mathcal{T}$ between regions $T$ and $T'$, see Figure \ref{fig:f-t}. If $t = |\mathrm{Tr} (\rho (T)) | > |\mathrm{Tr} (\rho (T'))| = t' $, then the oriented edge points towards $T'$. In the case where there is an equality, one can choose the orientation arbitrarily (and the choice will not affect the result). For $S = N_{1,3}$ the definition is similar, but an edge correspond to $e = X \cap Y \cap Z$. 

\item[Step (2):] We say that a vertex $v \in \mathcal{T}^{(0)}$ is a \textit{fork} if there are (at least) two arrows pointing away from $v$. We then prove a `Fork Lemma' showing that if a vertex $v = X \cap Y \cap Z$ is a fork, then $\min \{|x| , |y|, |z| \} \leq M(\mathcal{P})$ for $x=\mathrm{Tr}(\rho(X))$, $y=\mathrm{Tr}(\rho(Y))$ and $z=\mathrm{Tr}(\rho(Z))$. We use this result to show, for $\rho \in \X_{\mathcal{P}}$ and $K > M(\mathcal{P})$, that the set $\Omega_\rho (K)$  of regions with traces smaller than $K$ is connected. The proof is by contradiction and considers the distance between two connected components. The Fork Lemma is then used to reach the conclusion when such a distance is greater than $2$ since in that case a fork will appear. When $S = N_{1,3}$ the Fork Lemma needs to consider both regions (associated to $1$--sided curves) touching the vertex and faces (associated to $2$--sided curves) touching the vertex, and similarly one can then conclude the connectivity for the union of regions with small trace and the edge connectivity  for the union of faces with small trace. 

\item[Step (3)] For the third step, we first study the behavior of the values of the (norm of the) traces for regions that are all neighbors of a central region $X$. An important conclusion is that if $x = \mathrm{Tr}(\rho(X)) \notin [-2,2]$ and such that $\sigma(X)\neq 0$, then the sequence $(u_n)_{n\in \Nb}$ defined by $u_n = \mathrm{Tr} (\rho (X_n))$ grows exponentially in both directions. We then consider \emph{escaping rays}, that is, infinite geodesic rays where each edge $e_n$ is directed from $v_n$ to $v_{n+1}$. We show that if $\{ e_n \}_{n\in \Nb}$ is an escaping ray, then there are two cases: 
\begin{enumerate}
\item[$(1)$] there exists a region $\alpha$ such that the ray is eventually contained in the boundary of a region $X$, such that $x \in [-2,2]$, or $\sigma(X) = 0$, or 
\item[$(2)$] the ray meets infinitely many distinct regions with trace smaller than $M(\mathcal{P})$.
\end{enumerate}
When $S = N_{1,3}$, these conditions need to be adapted to ask that either there exists a face $X$ such that the ray is eventually contained in the boundary of a face such that $x \in [-2,2]$, or $\sigma(X) = 0$, or the ray meets infinitely many distinct faces in $\mathcal{S}_2(M)$ where $M = M(\mathcal{P})$.

\item[Step (4)]
Using the description of the growth of the traces for regions around a fixed central region, we define, for every representation $\rho \in \X_{\mathcal{P}}$ and for every $K > M(\mathcal{P})$, a connected attracting subgraph $T_\rho(K)$ of $\mathcal{T}$. The connectivity of $T_\rho(K)$ comes from the connectivity shown at step $(2)$, while the fact that $T_\rho(K)$ is attractive comes from the definition of $T_\rho(K)$. We then show that for representations in $\X_{BQ}$ such a graph $T_\rho(K)$ is finite for all $K > M(\mathcal{P})$.

\item[Step (5):] In the penultimate step, given a vertex $v \in \mathcal{T}^{(0)}$ we define a Fibonacci function $F_v : \mathcal{S} \rightarrow \Nb$. We then show that this function has the following property: if $\{\alpha, \beta, \gamma\}$ is a set of free generators for $F_3$ and $v = X \cap Y \cap Z \in \mathcal{T}^{(0)}$ with $X = [\alpha\beta]$, $Y = [\beta\gamma]$ and $Z = [\gamma\alpha]$, then for any element $\omega \in \Omega = \mathcal{S}$ the Fibonacci function corresponds to (a multiple of) the word length of $\omega$ with respect to $\{\alpha, \beta, \gamma\}$, that is $F_v(\gamma) = W(\gamma)$. Note that we only care about the asymptotic growth of this function, as you will see, so the choice of $v$ will not affect our results. We then define what it means for a function $g \co\mathcal{S} \rightarrow [0, \infty)$ to have \textit{Fibonacci growth}: that there exist constants $\kappa_1, \kappa_2 >0$ such that $$\kappa_1 F_v (X) \leq g(X) \leq \kappa_2 F_v (X)$$ for all $X \in \Omega$. If only the lower (respectively upper) bound is satisfied, the function $g$ is said to have a lower (respectively upper) Fibonacci growth. Given any $\rho \in \mathfrak{X}_{\mathcal{P}}$, we denote by $\phi_\rho \co\Omega \rightarrow \Cb$ the function $\phi_\rho (X) := \log | \mathrm{Tr} (\rho (X)) |$. We consider the function $\phi_\rho^+ := \max \{ \phi_\rho , 0 \}$. We then show that this function, for any $\rho \in \X_{\mathcal{P}}$, has an upper Fibonacci growth, but for representations in $\rho \in \mathfrak{X}_{BQ}$ we show that the function has upper and lower Fibonacci growth. 

We use this result to show that we can characterize representations in $\X_{BQ}$ in terms of the attracting graph $T_\rho(K)$ as follows: $\rho$ is in $\X_{\mathcal{P}, BQ}$ if and only if the subgraph $T_\rho (K)$ is finite for all $K> M(\mathcal{P})$, which proves Theorem \ref{Main-characterization}. This characterization is crucial to show that the set $\X_{BQ}$ is open and the mapping class group acts on it properly discontinuously, and hence to prove \ref{bow}.

\item[Step (6):] In order to show that the set $\X_{BQ}$ strictly contains primitive-stable representations when $S = S_{0,4}$ or $S = N_{1,3}$, we consider hyperbolic structures with parabolic boundary components, which are in $\X_{BQ}$ but not in $\X_{PS}$. Note that we do not know the strict inclusion for all boundary values. 
\end{enumerate}

\begin{remark}
Lawton--Maloni--Palesi \cite{LMP} generalize this theory to representations $\rho\co \pi_1(S_{1,1}) \to \mathsf{SU}(2,1) = \mathrm{Isom}(\mathbb{H}_{\Cb})$. To do so, they have to introduce a new graph $\mathcal{E}$, the edge graph, that better encodes the different structure of the character variety in this setting, but they defined the Bowditch set in that context, showed the equivalences and showed that this is a domain of discontinuity for $\MCG(S_{1,1})$ strictly containing convex-cocompact representations. 
\end{remark}

\subsubsection{Bowditch representations in Gromov hyperbolic spaces}\label{gromov}

We will now explain how the above discussion on Bowditch representations and the approach described in Section \ref{characterizations} can be  generalized for representation of the once-punctured torus in isometry groups of Gromov-hyperbolic spaces. 

Let us first recall that in the setting of geodesic metric space, a \emph{$\delta$-hyperbolic space}, for $\delta \geq 0$, is a geodesic metric space where every triangle is $\delta$--thin. A triangle is \emph{$\delta$--thin} if every side of the triangle is contained in the $\delta$--neighborhood of the union of the two other sides. $\delta$--hyperbolic spaces were introduced by Gromov \cite{gromov} and give a large-scale metric framework for hyperbolicity, see Section \ref{gromov}. To every $\delta$--hyperbolic space, one can associate its Gromov-boundary. A geodesic $\delta$--space is said to be \emph{visible} when every distinct pair of points on the Gromov-boundary of the space can be joined by a bi-infinite geodesic. Examples of geodesic, $\delta$--hyperbolic spaces include all symmetric spaces of rank one, as the usual $n$--dimensional real hyperbolic space $\Hb^n$, but also non-proper examples, as the infinite-dimensional real hyperbolic space $\Hb^\infty$ or the curve complex associated with a hyperbolic surface, see Masur--Minsky \cite{MM1} about the last example. For a detailed account on Gromov-hyperbolic spaces, we refer the interested reader to Bridson--Haefliger \cite{bridson-haefliger} and Coornaert--Delzant--Papadopoulos \cite{CDP}. 

Analogously to the $n$--dimensional real hyperbolic space case $\Hb^n$, isometries of a $\delta$--hyperbolic space can be classified into three types: hyperbolic, parabolic and elliptic isometries. When $A$ is a hyperbolic isometry of a $\delta$--hyperbolic space $X$ and $o \in X$ is any point, the orbit map $\Zb \to X$ which sends $n\in \Zb$ to $A^no \in X$ is a quasi-isometric embedding and the isometry $A$ preserves exactly two distinct points on the boundary of the space.

As the trace of an isometry does not make sense anymore in $\delta$--hyperbolic spaces, one needs to replace it by another notion in order to define analogous conditions to $(BQ2)$ and $(BQ2')$ defined in Definition \ref{def:BQ-PSLC} and after. A natural notion to consider is the length of an isometry $A$, which can be defined either as the \textit{translation length} $$l(A):=\underset{x \in X}{\min} \, d(Ax,x),$$ or as the \textit{stable length} $$l_S(A):=\underset{n \to \infty}{\lim} \frac{1}{n}d(A^nx,x),$$ for $x\in X$. Observe that the stable length does not depend on the choice of $x \in X$. Note that one can easily see that in $\PSLC$ both notions are related to the trace, see Equation \eqref{eq:trace-length}. Note that these two notions are conjugacy invariant and that they are very close to each other: for rank one symmetric spaces $l(A)=l_S(A)$ and, in general, they only differ by an additive constant: $l_S(A) \leq l(A) \leq l_S(A)+16\delta$ (see \cite{CDP}). Also recall that the stable length characterizes hyperbolicity: an isometry $A$ is hyperbolic if and only if $l_S(A) >0$. 

We can now define, for $K>0$, the set $\Sc(K)$ in this setting as follows:
$$\Sc(K)=\{\gamma \in \Sc \mid l_S(\rho(\gamma)) \leq K\}.$$
Note that since we are working with the once-punctured torus, which is an orientable surface, all simple closed curves are $2$-sided, so $\Sc=\Sc_2$.

We now define a Bowditch set in this setting:
\begin{definition} \label{def:BQ in gromov-hyp} Let $X$ be a $\delta$--hyperbolic, geodesic, visible space. A representation $\rho\co F_2 \to \mathrm{Isom}(X)$ is in the set $\X_{BQ}(F_2,\mathrm{Isom}(X))$ if it satisfies the following conditions: 
	\begin{enumerate}
		\item[($\mathcal{BQ}1$)]  $\forall \gamma \in \Sc$, $\rho(\gamma)$ is hyperbolic.
		\item[($\mathcal{BQ}2$)] $\# \Sc(K_\delta)<\infty$, where $K_\delta=329\delta$.
	\end{enumerate}
\end{definition}

Similarly to the $\PSLC$ case, one can define, for every representation $\rho\co F_2 \to \mathrm{Isom}(X)$ and every $K>K_\delta$, a subtree of the Farey tree $\Tc$ denoted $T_\rho(K)$. This tree will be crucial to prove Theorem \ref{thm:BQopen prop disc} and its finiteness will characterize belonging to the Bowditch set. Note that, in contrast with the $\PSLC$ case, the large-scale nature of the setting prevents the tree $T_\rho(K)$ from being fully attracting, but in some sense remains ``quasi''-attracting, see \cite{schlich-equivalence}. 

We now also define the following conditions:
\begin{enumerate}
		\item[($\mathcal{BQ}2'$)] $\forall K >0$, $\# \Sc(K)<\infty$
		\item[($\mathcal{BQ}4$)] There exists $k,m>0$ such that $l_S(\rho(\gamma)) \geq kW(\gamma)-m$ for all $\gamma \in \Sc$. 
		\item[{($\mathcal{BQ}5$)}] For all $K\geq K_\delta$ ($K_\delta$ is defined in Definition \ref{def:BQ in gromov-hyp}), the tree $T_\rho(K)$ is finite.
	\end{enumerate}

Various analogous equivalent characterizations of the Bowditch set are in $\delta$--hyperbolic spaces:
\begin{Theorem}[Schlich \cite{schlich-equivalence}] \label{thm:BQ in Gromov-hyp}Let $X$ be a $\delta$--hyperbolic, geodesic, visible space. Let $\rho \co \Fb_2 \to \mathrm{Isom}(X)$. The following are equivalent:
	\begin{enumerate}
		\item The representation $\rho$ satisfies ($\mathcal{BQ}1$) and ($\mathcal{BQ}2$);
		\item The representation $\rho$ satisfies ($\mathcal{BQ}1$) and ($\mathcal{BQ}2'$);
		\item The representation $\rho$ satisfies ($\mathcal{BQ}4$);
		\item The representation $\rho$ satisfies ($\mathcal{BQ}5$).
	\end{enumerate}
\end{Theorem}

Note that the constant $K_\delta$ used to define ($\mathcal{BQ}2$) in Definition \ref{def:BQ in gromov-hyp} and which makes Theorem \ref{thm:BQ in Gromov-hyp} true only depends on the $\delta$--hyperbolic space $X$ and not on the value of the length of the boundary curve $[\alpha,\beta]$, that is, the condition used to define Bowditch representation does not depend on the relative character varieties of the one-holed torus. This was already the case in $\PSLC$ where Bowditch \cite{bow_mar} and Tan--Wong--Zhang \cite{TWZ} proved that the constant $M(\tau)=2$ works (recall that in this case the constant is a bound for the modulus of the trace), but this is in contrast with the case of the surfaces $S_{0,4}$ and $N_{1,3}$ where the definition of $M$ depends on the traces of the boundary components of the corresponding surfaces, see Section \ref{characterizations}. 

Using Theorem \ref{thm:BQ in Gromov-hyp} and in particular characterization ($\mathcal{BQ}5$), we can prove that the set $\X_{BQ}(F_2,\mathrm{Isom}(X))$ is an open domain of discontinuity in the character variety:

\begin{Theorem}[Schlich \cite{schlich-equivalence}] \label{thm:BQopen prop disc}
	The set $\X_{BQ}(F_2,\mathrm{Isom}(X))$ is open in $\X(F_2,\mathrm{Isom}(X))$, $\MCG$-invariant and $\MCG(S_{1,1})$ acts properly discontinuously on it. 
\end{Theorem}
Compare this result with Theorem \ref{bow}. The link with primitive-stable representation will be made in the next section. 

\textit{Proof strategy:}
The general strategy will follow the same lines as the one described above in Section \ref{proof}, except Step (6) which is not addressed here. However, the work of Bowditch \cite{bow_mar}, Tan--Wong--Zhang \cite{TWZ}, Maloni--Palesi--Tan \cite{MPT}, Maloni--Palesi \cite{MP} in $\PSLC$ and Lawton--Maloni--Palesi \cite{LMP} in $\mathsf{SU}(2,1)$ relies on the explicit description of the relative character varieties in terms of algebraic equations in order to prove the results in Step (1) to (5), which we do not have in the $\delta$--hyperbolic setting. One then needs to work out new large-scale geometric arguments, in order to replace the trace relations in the Lie groups $\PSLC$ and $\mathsf{SU}(2,1)$. For example, it is proved in \cite{schlich-equivalence} that there exists a constant $C_\delta$, depending only on the hyperbolicity constant $\delta$ of the space $X$, such that if $A$ and $B$ are two hyperbolic isometries with $l_S(A)>C_\delta$, $l_S(B)>C_\delta$, then we have:
\[ \max \{ l_S(AB),l_S(AB^{-1}) \} \geq l_S(A) + l_S(B)-C_\delta.\]
This last inequality can be thought of as a large-scale analogue of the trace relation $\Tr(AB)+\Tr(AB^{-1})=\Tr(A)\Tr(B)$ which holds in $\PSLC$, and will be a key tool to prove the Fork Lemma of Step (2). 

\subsection{Relationship between Bowditch and primitive-stable representations}\label{bow-PS}

In Section \ref{PS}, we introduced \emph{primitive-stable representations} for free groups of rank $n$, with $n\geq 2$, and in Section \ref{Bow} \emph{Bowditch representations} for punctured surfaces. When $n=2$, note that the free group of rank two $F_2$ is the fundamental group of the once-punctured torus $S_{1,1}$. We can then observe the remarkable fact that the mapping class group $\MCG(S_{1,1})$ is equal (up to index two) to the outer automorphism group $\mathrm{Out}(F_2)$. This last fact is very specific to the once-punctured torus and does not hold anymore for other punctured surfaces. When fixing an identification between $\pi_1(S_{1,1})$ and $F_2$, it is a classical fact that the set of free homotopy classes of simple closed essential unoriented curves on $S_{1,1}$ corresponds to the set of conjugacy and inversion classes of primitive elements in $F_2$, namely $\Sc=\mathcal{Pr}/\sim$ using the notation introduced in Sections \ref{PS} and \ref{Bow}.

It is not difficult to see that primitive-stable representations are, in particular, Bowditch representations. Minsky asked in \cite{Minsky-primitive} whether the set of primitive-stable representations of the free group of rank two coincides with the set of Bowditch representations of the once-punctured torus in $\PSLC$. In his thesis, Lupi \cite{lupi} answered this question positively when the representations take values in $\PSLR$. Later, Lee--Xu \cite{lee-xu} and independently Series \cite{ser_pri} proved this result for representations in $\PSLC$. 

\begin{Theorem}[Lupi \cite{lupi} in $\PSLR$, Lee--Xu \cite{lee-xu}, Series \cite{ser_pri} in $\PSLC$]
	$\X_{BQ}(S_{1,1},\PSLC)=\X_{PS}(F_2,\PSLC)$. 
\end{Theorem}

Note that Minsky introduced primitive-stability in $\PSLC=\mathrm{Isom}^+(\Hb^3)$, and that his definition immediately generalizes to more general isometry groups of metric spaces. 

Fl\'echelles \cite{flechelles} studied the link between primitive-stability and $(BQ)$--conditions for representations of the free group $F_2$ with values in $\mathrm{Isom}(\Hb^d)$, for $d\geq 3$. By extending the argument of Lee-Xu from $\Hb^3$ to $\Hb^n$, he proved:
\begin{Theorem}[Fl\'echelles \cite{flechelles}]\label{thm:BQ-PS-Hn}
	Let $\rho \co F_2 \to \mathrm{Isom}(\Hb^d)$ be a Coxeter extensible representation. There exists $\lambda>0$ such that $\rho$ is primitive-stable if and only if $\rho$ satisfies $(BQ1)$ and $\# \Sc(\lambda) <\infty$.
\end{Theorem}

In fact, Fl\'echelles proved this result for an a priori larger class of representations from $F_2$ to $\mathrm{Isom}(\Hb^d)$, namely those satisfying a property he called the half-length property (see Section 3 in \cite{flechelles}). In Theorem \ref{thm:BQ-PS-Hn}, the constant $\lambda$ depends on the representation $\rho$. More precisely, Fl\'echelles \cite{flechelles} showed that this constant can be chosen to be $\mathrm{Out}(F_2)$--invariant, to depend only on the dimension $d$ and on the primitive systole of $\rho$, which is the infimum of the translation length $l(\rho(\gamma))$ among all primitive elements $\gamma$ in $F_2$, and that the constant will go to infinity when the primitive systole goes to zero. 

Schlich generalized this equivalence in \cite{schlich} for representations with values in isometry groups of Gromov-hyperbolic spaces, using characterization ($\mathcal{BQ}4$) of Bowditch representations. Note that together with Theorem \ref{thm:BQ in Gromov-hyp}, this proves the following result, which can be seen from Theorem 1.1 in \cite{schlich} and Corollary 1.3 in \cite{schlich-equivalence}. Recall that $\X_{BQ}(F_2,\mathrm{Isom}(X))$ is defined in Definition \ref{def:BQ in gromov-hyp}:
\begin{Theorem}[Schlich \cite{schlich, schlich-equivalence}] \label{thm:eq-BQ-PS-Gromov-hyp}
Let $X$ be a $\delta$--hyperbolic, geodesic, visible space. Then $\X_{BQ}(F_2,\mathrm{Isom}(X))=\X_{PS}(F_2,\mathrm{Isom}(X))$.
\end{Theorem}

\begin{proof}[Proof sketch]
    We will now give a few ideas on the proof of Theorem \ref{thm:eq-BQ-PS-Gromov-hyp}. We start from a representation $\rho \co F_2 \to \mathrm{Isom}(X)$ satisfying $(\mathcal{BQ}4)$, which means that there exist constants $k$ and $m$ such that $l_S(\rho(\gamma))\geq k W(\gamma)-m$, for all $\gamma \in \Sc=\mathcal{Pr}/\sim$. In order to prove that $\rho$ is primitive-stable, the main step is to prove that, for every primitive element $\gamma \in \mathcal{Pr}$, the image under the orbit map $\tau_{\rho,o}$, for some basepoint $o\in X$ of the primitive geodesic associated with $\gamma$ in the Cayley graph (see Section \ref{PS}), stays in a uniform bounded neighborhood of the axis of $\rho(\gamma)$ in $X$. To this aim, we proceed by contradiction and assume there is a primitive element $\gamma$ such that the image of the orbit map restricted to the primitive geodesic associated with $\gamma$ in the Cayley graph goes far from the axis of $\rho(\gamma)$. This means that some subword $u$ of $\gamma$ makes a large `excursion', namely that it stays far from the axis. Using coarse hyperbolic geometry arguments, one shows that this implies that the isometry $\rho(u)$ moves the basepoint by a small amount, in this case we say that $u$ realizes a `quasi-loop'. The goal is then to find as many disjoint occurrences of  quasi-loops as possible in $\gamma$. To do so, one needs to use the structure of primitive elements in the free group of rank two and, in particular, properties on the redundancy of subwords of primitive elements. Finally, using a recursive argument, one can show that there is an arbitrary large proportion of the word $\gamma$ that does not displace the basepoint much, which will contradict the inequality above between the length of $\rho(\gamma)$ and the word length $W(\gamma)$, contradicting the assumption ($\mathcal{BQ}4$). 
    \end{proof}

Recall that we introduced \emph{simple-stable} representations for surface groups (possibly with punctures) in the end of Section \ref{PS}. One can also ask for the link between simple-stable representations and Bowditch representations of (punctured)-surfaces. In the case of the four-punctured sphere, Schlich \cite{schlich_sphere} obtained an analogue of Theorem \ref{thm:eq-BQ-PS-Gromov-hyp} in Gromov-hyperbolic spaces, using condition $(\mathcal{BQ}4)$ as a definition for Bowditch representation of the four-punctured sphere in this setting. This implies that in $\PSLC$ the set $\X_{BQ}(\pi_1(S_{0,4}),\PSLC)$ defined by Maloni--Palesi--Tan \cite{MPT} (see Definition \ref{def:BQ-PSLC} and Theorem \ref{Main-characterization}) is equal to the set $\X_{SS}(\pi_1(S_{0,4}),\PSLC)$ of simple-stable representations. Note that in general, an analogue of Theorem \ref{thm:BQ in Gromov-hyp} for the four-punctured sphere in Gromov-hyperbolic spaces is not known.

\section{Representations in higher rank Lie groups}\label{higher-rank}

In this section we discuss how to generalize some of the questions addressed in previous sections, especially in Sections \ref{PSL-C}, to representations into semisimple Lie groups of higher rank. 

	\subsection{Anosov representations} 
	
	Anosov representations have been considered a fruitful generalization of the notion of convex-cocompact representation into rank-one Lie groups to representations in higher rank Lie groups. They were first introduced by Labourie \cite{lab_ano} in his study of the surface groups representations defined by Hitchin \cite{hit_lie}, and were later generalized to representations of hyperbolic groups by Guichard-Wienhard \cite{gui_ano} and further studied by many authors \cite{KLP1, GGKW, BPS}. They provide a rich class of discrete representations which share many key features  with convex-cocompact representations in rank one. 

	We will discuss a definition of Anosov representations only in terms of the growth of their singular value gaps, following Kapovich-Leeb-Porti \cite{KLP1} and Bochi-Potrie-Sambarino \cite{BPS}. We invite the interested reader to use the survey paper by Canary \cite{can_dyn} to see other definitions for Anosov representations and the relationships between them.
    
    Let $\Gamma$ be a finitely generated group equipped with a word metric $| \cdot \nobreak |$ coming from a finite generating set. Let $\Kb = \Rb$ or $\Cb$ and $d\geq 2$ be an integer and denote $\Gs=\mathsf{GL}_d(\Kb)$. We choose here to introduce Anosov representations in $\mathsf{G} = \mathsf{GL}_d(\Kb)$ since the main results we will survey hereafter are for this group and since the definition simplifies in this setting and does not need to introduce a lot of Lie theoretical background needed for the general setting. Let $K$ be a maximal compact of $\mathsf{G}$, namely $\mathsf{K}=\mathsf{O}(d)$ if $\Kb=\Rb$ and $K=\mathsf{U}(d)$ if $\Kb=\Cb$. The \emph{Cartan decomposition} of $\mathsf{G}$ asserts that, for all $g \in G$, there exist $k_1, k_2 \in \mathsf{K}$ and $\sigma_1(g) \geq \dots \geq  \sigma_d(g) \in \Rb$ such that 
	\begin{equation*}
		g=k_1 \mathrm{diag}(\sigma_1(g),\dots,\sigma_d(g))k_2.
	\end{equation*}
	The real numbers $\sigma_1(g), \dots, \sigma_d(g)$ are the \emph{singular values} of $g$. 
	
	Unlike in the rank one case, there are several `ways' of being Anosov. In the case of $\mathsf{G} = \mathsf{GL}_d(\Kb)$, they can be encoded by an integer $1 \leq k \leq d-1$  and give rise to the notion of a $k$--\emph{Anosov} representation.
	\begin{definition}[Anosov representations in $\mathsf{GL}_d(\Kb)$] \label{def:anosov}
		Let $\Gamma$ be a finitely generated group and $\rho: \Gamma \to G$ a representation of $\Gamma$ in $G$. Let $1 \leq k \leq d-1$ be an integer. We say that $\rho$ is $k$--\emph{Anosov}\index{Anosov representation} if there exist two constants $\lambda >0, C>0$ such that
		\begin{equation}\label{eq:anosov}
			\frac{\sigma_k(\rho(\gamma))}{\sigma_{k+1}(\rho(\gamma))} \geq Ce^{\lambda |\gamma|}, \qquad \forall \gamma \in \Gamma.
		\end{equation} 
		We say that $\rho$ is \emph{Anosov} if there exists an integer $1 \leq k \leq d-1$ for which $\rho$ is $k$-Anosov.
	\end{definition}
	We denote by $\mathfrak{X}_{Anosov}(\Gamma,G)$ the subset of $\mathfrak{X}(\Gamma,G)$ consisting of all Anosov representations. This definition could be generalized to other semi-simple Lie groups, and we refer the reader to the above references to see a general discussion of this theory.
	
	\begin{remark} \label{rem:anosov}
	\begin{enumerate}
		\item This definition implies that the group $\Gamma$ must be hyperbolic (see \cite{BPS} and \cite{KLP1}). See Section \ref{gromov} and the references therein for a discussion on Gromov hyperbolic spaces and groups.
		\item This is not the original definition given by Labourie \cite{lab_ano}, but it has been proven to be equivalent by Kapovich-Leeb-Porti \cite{KLP1} and independently Bochi-Potrie-Sambarino \cite{BPS}.
		There are many other characterizations, for example in terms of boundary maps (from the boundary of the group into flag varieties), or in terms of the coarse geometry in the symmetric space of $G$, or in terms of eigenvalues. We refer the interested reader to \cite{gui_ano}, \cite{GGKW}, \cite{KP} and to the survey \cite{can_dyn}.
		\item \label{rem:anosov-PGL}Since the definition of an Anosov representation only involves a \emph{ratio} of singular values, we immediately see that this definition still makes sense when the target group is $\mathsf{PGL}_d(\Kb)$.  
		\end{enumerate}
	\end{remark}
	
	Let $X = X_{\mathsf{G}}=\mathsf{G}/\mathsf{K}$ be the symmetric space of $\mathsf{G}$ equipped with its Riemannian metric. When $\mathsf{G}=\mathsf{GL}_d(\Rb)$, $X$ identifies with the set of positive definite $d\times d$ symmetric matrices and its tangent space at $[\mathrm{Id}]$ with the space of $d\times d$ symmetric matrices.  The \emph{Killing form} defined by $\langle A , B \rangle = \mathrm{Tr}(AB)$ is a scalar product on $T_{[Id]}X$, and can be extended on $T_pX$, for all $p \in X$ by left multiplication.  This defines a Riemannian metric on $X$. Recall that the finitely generated group $\Gamma$ is equipped with a word metric and that we say that $\rho : \Gamma \to \mathsf{G}$ is a \emph{quasi-isometric embedding} if there exists $o \in X$ such that the orbit map $\tau_{\rho,o} : \Gamma \to X$, $o \mapsto \rho(\gamma).o$ is a quasi-isometric embedding between the metric spaces $\Gamma$ and $X$ and the kernel of $\rho$ is finite. 

	Guichard--Wienhard (\cite{gui_ano}) proved that, for $\mathsf{G}$ a semi-simple Lie group, Anosov representations are quasi-isometric embeddings. In particular, Anosov representations are discrete and are faithful whenever the group $\Gamma$ is torsion free. 
	
	\begin{remark}
	In the case of  $\mathsf{G}=\mathsf{SL}(n,\Rb)$, it is not hard to see that Anosov representations are quasi-isometric embeddings. This uses Definition \ref{def:anosov} and the fact that, if we denote by $d_X$ the distance in $X$, for all $g \in G$, we have that $d_X(g.o,o)=\Vert \log \sigma(g) \Vert_2$, where $\Vert \cdot \Vert_2$ is the euclidean norm. Indeed, this norm is equivalent to the sup norm, so $d_X(g.o,o) \geq A \log \sigma_1(g)$ for some positive constant $A$, and since $\log \sigma_1(g) + \dots + \log \sigma_d(g)=0$ (because $g \in \mathsf{SL}(n \Rb)$), we have the inequalities:
	\begin{align*}
		\log \sigma_1(g) & \geq \frac{1}{d} (\log \sigma_1(g) + (\log \sigma_1(g) + \dots + \log \sigma_{d-1}(g))) \\
		& \geq \frac{1}{d}(\log \sigma_1(g)-\log \sigma_d(g)) \geq \frac{1}{d}(\log \sigma_k(g)-\log \sigma_{k+1}(g)).
	\end{align*} 
	\end{remark}
	
	Not only are Anosov representations quasi-isometric embeddings, but Guichard--Wienhard \cite{gui_ano} proved that the constants of quasi-isometry can be chosen uniformly in a neighborhood of an Anosov representation:
	\begin{theorem}[Guichard--Wienhard \cite{gui_ano}, Theorem 5.14] \label{thm:unif-QI}
		Let $\rho_0\co \Gamma \to \mathsf{G}$ be an Anosov representation. Then there exists an open neighborhood of $\rho_0$ in $\mathrm{Hom}(\Gamma,\mathsf{G})$ and two constants $\lambda,C >0$  such that every representation $\rho \in U$ is a $(\lambda,C)$--quasi-isometric embedding. 
	\end{theorem}

    Compare Theorem \ref{thm:unif-QI} with Proposition \ref{prop:local-uniform-cst-qie} in the rank-one case.
	As in the rank-one case, this property allows one to describe the action of the outer automorphism group $\mathrm{Out}(\Gamma)$ of $\Gamma$ on the set $\mathfrak{X}_{Anosov}(\Gamma,\mathsf{G})$ of Anosov representations.
	
	\begin{Theorem}[Guichard--Wienhard \cite{gui_ano}, Canary \cite{canary_survey_ano}]\label{thm:Anosov-domain of discontinuity}
		The outer automorphism group of $\Gamma$ acts properly discontinuously on $\mathfrak{X}_{Anosov}(\Gamma,\mathsf{G})$.
	\end{Theorem}

	\begin{remark} Labourie first proved this result for Hitchin representations in $\mathsf{PSL}_d(\Rb)$ and maximal representations in $\mathsf{PSp}(2d,\Rb)$ in \cite{lab_ano}. Guichard and Wienhard stated the Theorem for $\Gamma$ a surface group or a free group (Corollary 5.4) into semi-simple Lie groups, and Canary explained in \cite{canary_survey_ano} how to generalize to torsion-free hyperbolic groups. 
	\end{remark}

    \begin{proof} The proof is the same as the proof of Theorem \ref{thm:prop-discCC} using Theorem \ref{thm:unif-QI}.
    \end{proof}
	
	Let us denote by $\mathfrak{X}_{QI}(\Gamma,\mathsf{G})$ the subset in the character variety consisting of (conjugacy class) of representations which are quasi-isometric embeddings. As already mentioned above, we have ${\mathfrak{X}_{Anosov}}(\Gamma,\mathsf{G}) \subset \mathfrak{X}_{QI}(\Gamma,\mathsf{G})$. While in rank one, these two sets are known to be equal, this is not the case anymore in higher rank. Indeed, Gu\'eritaud--Guichard--Kassel--Wienhard, elaborating on an example of Guichard (\cite{gui-phd}), give an example of a representation of the free group of rank two into $\mathsf{SL}(4,\Rb)$ which is a quasi-isometric embedding but not an Anosov representation. This example also shows that the openness of $\mathfrak{X}_{QI}(\Gamma,\mathsf{G})$ fails in the higher rank setting, as this representation is a limit of indiscrete representations. Furthermore, Tsouvalas \cite{tsou} proved that in general, $\mathfrak{X}_{Anosov}(\Gamma,\mathsf{G})$ does not even coincide with the interior of $\mathfrak{X}_{QI}(\Gamma,\mathsf{G})$ because he provides examples of representations of a hyperbolic group in $\mathsf{SL}_d(\Rb)$, with $d \geq 5$, which are quasi-isometric embeddings but not limits of Anosov representations. He also exhibits a hyperbolic group $\Gamma$ and an open subset of $\mathrm{Hom}(\Gamma,\mathsf{SL}_d(\Kb))$, for $d$ sufficiently large, consisting entirely of quasi-isometric embeddings which are not Anosov. Moreover, the representations in this open set are non-locally rigid, meaning that they are limits of representations that are not conjugate to them. We can raise the question of the dynamics of the outer automorphism group on the set of quasi-isometric embeddings:
	\begin{question}
		Is the interior of $\mathfrak{X}_{QI}(\Gamma,\mathsf{G})$ an open domain of discontinuity for the action of the outer automorphism group ?
	\end{question}

    \subsection{Other higher rank domains of discontinuity}
	
	In Section \ref{PS}, we explained how Minsky constructed an open domain of discontinuity in the $\mathsf{PSL}(2,\Cb)$--character varieties of free groups containing the domain of convex-cocompact representations. His idea was to consider a weakening of the convex-cocompact condition by requiring the quasi-geodesic embedding property only along primitive geodesics in the group.
	Guichard--Gu\'eritaud--Kassel--Wienhard, while developing the theory of Anosov representations, suggested a higher rank analogue of primitive-stability, see Remark 1.6 (b) \cite{GGKW}. This has been then studied by other authors like Kim--Kim \cite{KK}, Kim--Tan--Zhang \cite{KTZ}, Wang \cite{wang2021anosov}, Tholozan--Wang \cite{tholozan-wang}. 
	
	Let $F_n$ be a non-abelian free group of rank $n$. Recall that we defined \emph{primitive} elements in $F_n$ in section \ref{PS} and denote $\mathcal{Pr}$ the set of primitive elements in $F_n$. Also recall that we denote by $\mathrm{Ax}_{\mathcal{Pr}}$ the set of \emph{primitive geodesics} in the Cayley graph $\mathrm{Cay}$ of $F_n$ (for a choice of a free generating set) which are the axes in $\mathrm{Cay}$ of the primitive elements in $F_n$. 
	
	\begin{definition}[Primitive-Anosov representations] \label{def:primitive-anosov} \index{Primitive-Anosov representation}
		Let $\rho\co F_n \to \mathsf{GL}_d(\Kb)$ be a representation of $F_n$ in  $\mathsf{GL}_d(\Kb)$ and let $1 \leq \nobreak k \leq \nobreak d-1$ be an integer. We say that $\rho$ is $k$--\emph{primitive-Anosov} if there exist two constants $\lambda>0, C>0$ such that for any primitive geodesic $\{ \gamma_n\}_{n \in \Zb} \in \mathrm{Ax}_{\mathcal{Pr}}$, with $\gamma_0=\mathrm{id}$, we have:
		\begin{equation}
		\frac{\sigma_k(\rho(\gamma_n))}{\sigma_{k+1}(\rho(\gamma_n))} \geq Ce^{\lambda n},  \qquad \forall n \in \Nb. 
		\end{equation}
		We say that $\rho$ is \emph{primitive-Anosov} if there exists an integer $1 \leq k \leq d-1$ for which $\rho$ is $k$--primitive-Anosov.
	\end{definition}
	
	\begin{remark}
		\begin{enumerate}
		\item Notice the similarity with Definition \ref{def:anosov} of Anosov representations. Definition \ref{def:primitive-anosov} is just a restriction of Definition \ref{def:anosov} to primitive geodesics. 
		\item As for Anosov representations, the notion still makes sense in $\mathsf{PGL}_d(\Kb)$, see Remark \ref{rem:anosov} \eqref{rem:anosov-PGL}.
		\item Primitive-Anosov representations can be characterized in many other ways. For example, Kim--Kim use in \cite{KK} a definition in terms of Morse-quasi-geodesics in the symmetric space, following the framework of Kapovich--Leeb--Porti \cite{KLP1}. In \cite{wang2021anosov}, Wang studied the equivalence between some properties characterizing primitive-Anosov representations, such as the dynamic of a flow, the existence of a dominated splitting for some bundle associated with the representation, or the existence of limit maps.
		\end{enumerate}
	\end{remark}
	
	Let $S$ be a compact, connected, orientable surface with boundary of negative Euler characteristic. Note that the fundamental group $\pi_1(S)$ of $S$ is a free group. Kim--Kim studied the link between convex projective structures on $S$, positive representations of $\Gamma$ in $\mathsf{PGL}_d(\Rb)$ and primitive-Anosov representations. We now introduce convex projective structures and positive representations. \\
    
    \textit{Convex projective structures in \texorpdfstring{$\mathbb{RP}^2$}{RP2}.}
	A convex projective structure on $S$ is a diffeomorphism between $S$ and a quotient $\Omega / \Gamma$ where $\Omega$ is a convex domain in $\mathbb{RP}^2$ and $\Gamma$ is a discrete subgroup of $\mathsf{PGL}(3, \Rb)$ acting properly and freely on $\Omega$. We may therefore identify $\Gamma$ with the fundamental group $\pi_1(S)$ of the surface $S$.\\ 
	
	\textit{Positive representations in \texorpdfstring{$\mathsf{PGL}_d(\Rb)$}{PGL(d,R)}.}
	A \emph{full flag} in $\Rb^d$ is an increasing sequence of subspaces $\{0\}=F^{(0)} \subset F^{(1)} \subset \dots \subset F^{(d)}=\Rb^d$ such that for all $0 \leq i \leq d$, we have $\dim F^{(i)}=i$. Denote by $\Fc(\Rb^d)$ the set of full flags in $\Rb^d$. Let $\mathsf{B}^+$ (respectively $\mathsf{B}^-$) be the set of upper triangular (respectively lower triangular) matrices in $\mathsf{PGL}_d(\Rb)$. The set $\Fc(\Rb^d)$ can be identified with $\mathsf{PGL}_d(\Rb)/\mathsf{B}^+$, thus it is a  homogeneous manifold. We say that two $n$-tuples of full flags are \emph{equivalent} if one is the image of the other by an element of $\mathsf{PGL}_d(\Rb)$. We denote by $\mathsf{U}^+ \subset \mathsf{B}^+$ the set of unipotent upper triangular matrices. We say that a matrix in $\mathsf{U}^+$ is \emph{totally positive} if all its minors are positive, except for those that are necessarily zero due to the structure of $\mathsf{U}^+$. We denote $U^+_{>0} $ the subset of totally positive matrices in $\mathsf{U}^+$. We say that an $n$-tuple of full flags $(F_1,\dots,F_n) \in \Fc(\Rb^d)^n$ is \emph{positive} if there exists $u_i \in U^+_{>0}$, for $1\leq i \leq n-2$, such that $(F_1,\dots,F_n)$ is equivalent to 
	\begin{equation*}
		(\mathsf{B}^+,\mathsf{B}^-,u_1.\mathsf{B}^-,(u_1u_2).\mathsf{B}^-,\dots,(u_1\dots u_{n-2}).\mathsf{B}^-).  
	\end{equation*}  
	
	We would now like to have a natural cyclic ordering on the boundary of the group $\pi_1(S)$. Since this is a free group (because $S$ has non-empty boundary), it is not clear what such a cyclic ordering should be only by looking at the group. However, we can use the identification with the fundamental group of the surface to give an ordering that will encode the geometry of the surface. We choose a finite volume hyperbolic metric on the interior of $S$ such that its completion has geodesic boundary components and $p$ cusps. Then the universal cover $\tilde{S}$ of the surface $S$ identifies with a convex subset of the hyperbolic plane and we define the boundary at infinity $\partial_\infty \pi_1(S)_p$ to be the boundary at infinity of this convex subset. Thus $\partial_\infty \pi_1(S)_p$ inherits a natural ordering from the boundary of the hyperbolic plane depending on the orientation on the surface $S$. The boundary $\partial_\infty \pi_1(S)_p$ does not depend on the choice of the hyperbolic metric, see \cite{Labourie-McShane}. 

    \indent	A map $\xi\co \partial_\infty \pi_1(S)_p \to \Fc(\Rb^d)$ is said to be \emph{positive} if it sends positively oriented tuples in $\partial_\infty \pi_1(S)_p$ to positive tuples of full flags in $\Fc(\Rb^d)$. Finally, we say that a representation $\rho\co \pi_1(S) \to \mathsf{PGL}_d(\Rb)$ is \emph{positive} if there exists a continuous positive $\rho$-equivariant map: $\xi\co\partial_\infty \pi_1(S)_p \to \Fc(\Rb^d)$ for some~	$p$.
	
	\begin{remark}
		There exists also a notion of positivity in real semi-simple Lie groups developed by Guichard--Wienhard in \cite{gui_ano} and a notion of positive representations of surface groups developed by Guichard--Labourie--Wienhard \cite{GLW}.
	\end{remark}

    We have the following result:
	
	\begin{theorem}[Kim--Kim \cite{KK} Theorem 1.2]
	Let $S$ be a compact, connected, orientable surface with boundary of negative Euler characteristic. The holonomy representations of convex-projective structures in $\mathbb{RP}^2$ and positive representations from $\pi_1(S)$ to $\mathsf{PGL}_d(\Rb)$ are primitive-Anosov. 
	\end{theorem}
    
	\begin{remark}
		Positive representations with unipotent boundary element are not Anosov representations.
        
		We say that a matrix in $\mathsf{PGL}(3,\Rb)$ is \emph{hyperbolic} if it is conjugate to $\begin{bmatrix}
			a & 0 & 0 \\
			0 & b & 0 \\
			0 & 0 & c
		\end{bmatrix}$ with $a>b>c>0$ and $abc=1$. \\
		We say that a matrix in $\mathsf{PGL}(3,\Rb)$ is \emph{quasi-hyperbolic} if it is conjugate to $\begin{bmatrix}
			a & 1 &0 \\
			0 & a & 0 \\
			0 & 0 & b
		\end{bmatrix}$ with $a>0,b>0$, $a \neq b$ and $a^2b=1$.
        
		Holonomies of convex projective structures that map every boundary component to a hyperbolic matrix are both Anosov and positive representations. On the other hand, holonomies of convex projective structures mapping any boundary component to a quasi-hyperbolic element are neither Anosov nor positive representations. 
		\end{remark}
        
		From this, Kim--Kim obtain the following corollary: 
	\begin{corollary}[Kim--Kim]
		There is an open domain of discontinuity in the character variety $\mathfrak{X}(F_n,\mathsf{PGL}_d(\Rb))$ for the action of the outer-automorphism group of $F_n$ which is strictly larger than the set of Anosov representations. 
	\end{corollary}
	
	\begin{remark}
		The fact that primitive-Anosov representations form an open domain of discontinuity is true in any semi-simple Lie group without compact factor, see \cite{KK} or \cite{wang2021anosov}. 
	\end{remark}
	
	The works mentioned above then suggest that a strategy for constructing domains of discontinuity larger than those arising from Anosov representations is to consider primitive-Anosov representations of a surface with boundary that degenerate along a boundary component. This raises the question of how to construct such domains for a closed surface group. In this setting, not only there is no boundary component but also primitive elements do not make sense anymore since the group $\pi_1(S)$ is not free. The natural analogous notion to consider would then be the notion of \emph{simple-Anosov} representations, as studied by Tholozan--Wang \cite{tholozan-wang}. The idea of the definition of simple-Anosov representations is the same as for primitive-Anosov representations but restricting to simple elements rather than primitive elements: consider a restriction of the Anosov condition \eqref{eq:anosov} to the simple geodesics in the group, corresponding to the simple closed curves on the surface. We give a precise definition below for completeness. 
	
	Let $S$ be a closed connected oriented surface with negative Euler characteristic. Fix $\Sc$ a generating set for its fundamental group $\pi_1(S)$ and consider $\mathrm{Cay}:=\mathrm{Cay}(\pi_1(S),\Sc \cup \Sc^{-1})$ the Cayley graph of $\pi_1(S)$ with respect to the generating set $\Sc \cup \Sc^{-1}$.  Let $\mathscr{S}$ be the set of simple closed curves on $S$. We denote by $\mathrm{Ax}_\mathscr{S}$ the set of \emph{simple geodesics} in $\mathrm{Cay}$, that is, the set of geodesics in $\mathrm{Cay}$ which are axes of some simple element in $\pi_1(S)$.
	
		\begin{definition}[Simple-Anosov representations] \label{def:simple-anosov}~\\ 
		\index{Simple-Anosov representation}Let $\rho\co\pi_1(S) \to \mathsf{GL}_d(\Kb)$ be a representation of $\pi_1(S)$ in  $\mathsf{GL}_d(\Kb)$ and $1 \leq \nobreak k \leq \nobreak d-1$ an integer. We say that $\rho$ is \emph{$k$-simple-Anosov} if there exist two constants $\lambda>0, C>0$ such that for all simple geodesic $\{\gamma_n\}_{n \in \Zb} \in \mathrm{Ax}_{\mathscr{S}}$, with $\gamma_0=\mathrm{id}$, we have:
		\begin{equation*}
			\frac{\sigma_k(\rho(\gamma_n))}{\sigma_{k+1}(\rho(\gamma_n))} \geq Ce^{\lambda n},  \qquad \forall n \in \Nb. 
		\end{equation*}
		We say that $\rho$ is \emph{simple-Anosov} if there exists an integer $1 \leq k \leq d-1$ for which $\rho$ is $k$-simple-Anosov.
	\end{definition}

		While in $\mathsf{PSL}(2,\Cb)$ it is not known whether or not there exist some simple-Anosov representations which are not convex-cocompact, when the dimension gets bigger, Tholozan--Wang proved the existence of simple-Anosov representations which are not Anosov in $\mathsf{SL}(2d,\Cb)$:

	 \begin{theorem}[Tholozan--Wang, \cite{tholozan-wang}] \label{thm:TW}
	 Let $S$ be a closed connected oriented surface with negative Euler characteristic and $d\geq 2$. The set of $d$-simple Anosov representations from $\pi_1(S)$ to $\mathsf{SL}(2d,\Cb)$ is an open domain of discontinuity for the action of the mapping class group of $S$, which contains points of the boundary of the domain of Anosov representations (hence is strictly larger that the domain of Anosov representations).
	 \end{theorem}
	 
 	Given a simple-Anosov representation on the boundary of the set of Anosov representations, Tholozan--Wang use openness to deform it into non discrete and faithful Zariski dense simple-Anosov representations, see Section 4.4 of \cite{tholozan-wang}.

   We try to give a brief overview of the construction of Tholozan--Wang of a $d$-simple Anosov representation in $\mathrm{SL}(2d,\Cb)$ which is not Anosov.
	\begin{proof}[Overview of the proof of Theorem \ref{thm:TW}]
	First, consider a Galois covering $\tilde{S}$ of degree $d$ of the surface $S$ and a simple closed curve $\tilde{c}$ on $\tilde{S}$ such that $\tilde{c}$ projects on a curve $c$ with a self-intersection on $S$. Thus the fundamental group $\Gamma_0:=\pi_1(\tilde{S})$ of $\tilde{S}$ is a finite index normal subgroup of the fundamental group $\Gamma := \pi_1(S)$ of $S$. The idea is now to consider a representation $\rho_0 : \Gamma_0 \to \mathrm{SL}(2,\Cb)$ which is geometrically finite, and such that the stabilizers of parabolic points are conjugate to the cyclic subgroup of $\Gamma_0$ generated by $\tilde{c}$.
    
	They then use a standard procedure for constructing a representation $\rho\co\Gamma \to \mathsf{GL}(2d,\Cb)$ of a group $\Gamma$ from a representation $\rho_0 : \Gamma_0 \to \mathsf{GL}(2,\Cb)$ of a finite index (normal) subgroup $\Gamma_0$ of $\Gamma$. We briefly explain the construction. The subgroup $\Gamma_0$ is finite index in $\Gamma$, which means that we can decompose $\Gamma= \nobreak \underset{i=1}{\overset{d}{\sqcup}} \gamma_i \Gamma_0$. Let $V_0:=\Cb^2$ and consider $d$ formal copies of $V_0$ denoted by~$\gamma_i V_0$, for all $1 \leq i \leq d$. Then the vector space $V:=\Cb^{2d}$ can be identified with $\displaystyle \bigoplus_{i=1}^d \gamma_i V_0$. Let us define a linear action $\rho$ of the group $\Gamma$ on $V$. Take $\gamma\in \Gamma$, an integer $1 \leq i \leq d$ and $x_i \in \gamma_i V_0$, which we write $x_i=\gamma_ix$, where $x$ is the corresponding point in $V_0$. Then there exists an index $1 \leq j \leq d$ and an element $\gamma_0 \in \Gamma_0$ such that $\gamma \gamma_i=\gamma_j \gamma_0$ and we set $\rho(\gamma) x_i :=  \gamma_j (\rho_0(\gamma_0)x)$, which means that $\rho(\gamma) x_i$ is the image of $\rho_0(\gamma_0)x \in V_0$ in the copy $\gamma_j V_0$ of $V_0$. The induced representation $\rho$ is commonly denoted by $\mathrm{Ind}_{\Gamma_0}^{\Gamma}(\rho_0)$.
    
	The induced representation $\rho$ of the representation $\rho_0\co\Gamma_0 \to \mathsf{SL}(2,\Cb)$ defined above will satisfy the required properties:
	\begin{itemize}
	\item The curve $c$ on $S$ is self-intersecting, then, seen as a closed geodesic (for any hyperbolic metric on $S$) in the unit tangent bundle of $S$, it is disjoint from the closure of the union of all simple closed geodesic on $S$. This will imply that $\rho_0$ will have good geometric behavior (namely  being $1$-Anosov) in the direction of the group $\Gamma_0$ corresponding to the lifts on $\tilde{S}$ of the simple closed curves on $S$. Tholozan-Wang proved that passing to the induced representation preserves this `good geometric behavior', which means here that $\rho$ is $d$-simple-Anosov.
	\item 
	By construction, the image of the curve $\tilde{c} \in \Gamma_0$ by $\rho_0$ is parabolic, so in particular the modulus of the two eigenvalues of $\rho_0(\tilde{c})$ are the same and equal to $1$. This will imply that, in the induced representation $\rho$, the $d$-th and $(d+1)$-th eigenvalues (ordered in the non-increasing order of their modulus) also have both modulus $1$, which will forbid $\rho$ to be a $d$--Anosov representation.
	\end{itemize}
	\end{proof}	
    	

\bibliographystyle{amsalpha}

\bibliography{sample}

@article {BFM-minimalDT,
    AUTHOR = {Bouilly, Yohann and Faraco, Gianluca and Maret, Arnaud},
     TITLE = {Mapping class group orbit closures for {D}eroin-{T}holozan
              representations},
   JOURNAL = {J. Mod. Dyn.},
  FJOURNAL = {Journal of Modern Dynamics},
    VOLUME = {21},
      YEAR = {2025},
     PAGES = {805--850},
      ISSN = {1930-5311,1930-532X},
   MRCLASS = {37D40 (20C15 37J06 53D30 57K20)},
  MRNUMBER = {5009413},
       DOI = {10.3934/jmd.2025019},
       URL = {https://doi.org/10.3934/jmd.2025019},
}

@misc{BM,
      title={Tykhyy's Conjecture on finite mapping class group orbits}, 
      author={Samuel Bronstein and Arnaud Maret},
      year={2025},
      eprint={2409.04379},
      archivePrefix={arXiv},
      primaryClass={math.DS},
      url={https://arxiv.org/abs/2409.04379}, 
}

@incollection{KL-survey,
    AUTHOR = {Kapovich, Michael and Leeb, Bernhard},
     TITLE = {Discrete isometry groups of symmetric spaces},
 BOOKTITLE = {Handbook of group actions. {V}ol. {IV}},
    SERIES = {Adv. Lect. Math. (ALM)},
    VOLUME = {41},
     PAGES = {191--290},
 PUBLISHER = {Int. Press, Somerville, MA},
      YEAR = {2018},
      ISBN = {978-1-57146-365-4},
   MRCLASS = {22E40 (20E42 20F65 53C35)},
  MRNUMBER = {3888689},
MRREVIEWER = {Jean\ Raimbault},
}

@article{magnus-coords,
    AUTHOR = {Magnus, Wilhelm},
     TITLE = {Rings of {F}ricke characters and automorphism groups of free
              groups},
   JOURNAL = {Math. Z.},
  FJOURNAL = {Mathematische Zeitschrift},
    VOLUME = {170},
      YEAR = {1980},
    NUMBER = {1},
     PAGES = {91--103},
      ISSN = {0025-5874,1432-1823},
   MRCLASS = {20F28 (13B25 20F36 20G20)},
  MRNUMBER = {558891},
MRREVIEWER = {D.\ L.\ Johnson},
       DOI = {10.1007/BF01214715},
       URL = {https://doi.org/10.1007/BF01214715},
}

@article{EGPS,
    AUTHOR = {Erlandsson, Viveka and Gendulphe, Matthieu and Pasquinelli,
              Irene and Souto, Juan},
     TITLE = {Mapping class group orbit closures for non-orientable
              surfaces},
   JOURNAL = {Geom. Funct. Anal.},
  FJOURNAL = {Geometric and Functional Analysis},
    VOLUME = {33},
      YEAR = {2023},
    NUMBER = {3},
     PAGES = {637--693},
      ISSN = {1016-443X,1420-8970},
   MRCLASS = {57K20 (37F32)},
  MRNUMBER = {4597639},
MRREVIEWER = {Dongryul\ M.\ Kim},
       DOI = {10.1007/s00039-023-00638-7},
       URL = {https://doi.org/10.1007/s00039-023-00638-7},
}

@article{mag_cou,
    AUTHOR = {Magee, Michael},
     TITLE = {Counting one-sided simple closed geodesics on {F}uchsian
              thrice punctured projective planes},
   JOURNAL = {Int. Math. Res. Not. IMRN},
  FJOURNAL = {International Mathematics Research Notices. IMRN},
      YEAR = {2020},
    NUMBER = {13},
     PAGES = {3886--3901},
      ISSN = {1073-7928,1687-0247},
   MRCLASS = {57K20 (32G15 37D40)},
  MRNUMBER = {4120312},
       DOI = {10.1093/imrn/rny112},
       URL = {https://doi.org/10.1093/imrn/rny112},
}

@article{hua_sim,
	Author = {Huang, Yi and Norbury, Paul},
	Journal = {Geometriae Dedicata},
	Number = {1},
	Pages = {113--148},
	Publisher = {Springer},
	Title = {Simple geodesics and Markoff quads},
	Volume = {186},
	Year = {2017}}

@unpublished{gen_wha,
	Author = {Gendulphe, Matthieu},
	Month = {June},
	Note = {Preprint arXiv:1706.08798},
	Title = {What's wrong with the growth of simple closed geodesics on nonorientable hyperbolic surfaces},
	Year = {2017}}

@article{goldman-topological,
    AUTHOR = {Goldman, William M.},
     TITLE = {Topological components of spaces of representations},
   JOURNAL = {Invent. Math.},
  FJOURNAL = {Inventiones Mathematicae},
    VOLUME = {93},
      YEAR = {1988},
    NUMBER = {3},
     PAGES = {557--607},
      ISSN = {0020-9910,1432-1297},
   MRCLASS = {57M05 (22E40 32G15)},
  MRNUMBER = {952283},
MRREVIEWER = {William\ Harvey},
       DOI = {10.1007/BF01410200},
       URL = {https://doi.org/10.1007/BF01410200},
}

@article{delzant,
    AUTHOR = {Delzant, Thomas},
     TITLE = {Hamiltoniens p\'{e}riodiques et images convexes de
              l'application moment},
   JOURNAL = {Bull. Soc. Math. France},
  FJOURNAL = {Bulletin de la Soci\'{e}t\'{e} Math\'{e}matique de France},
    VOLUME = {116},
      YEAR = {1988},
    NUMBER = {3},
     PAGES = {315--339},
      ISSN = {0037-9484},
   MRCLASS = {58F05},
  MRNUMBER = {984900},
MRREVIEWER = {J.\ J.\ Duistermaat},
       URL = {http://www.numdam.org/item?id=BSMF_1988__116_3_315_0},
}

@article{PX2,
    AUTHOR = {Pickrell, Doug and Xia, Eugene Z.},
     TITLE = {Ergodicity of mapping class group actions on representation
              varieties. {II}. {S}urfaces with boundary},
   JOURNAL = {Transform. Groups},
  FJOURNAL = {Transformation Groups},
    VOLUME = {8},
      YEAR = {2003},
    NUMBER = {4},
     PAGES = {397--402},
      ISSN = {1083-4362,1531-586X},
   MRCLASS = {57M50 (20F65 22C05 37A15)},
  MRNUMBER = {2015257},
MRREVIEWER = {A.\ H.\ Dooley},
       DOI = {10.1007/s00031-003-0819-6},
       URL = {https://doi.org/10.1007/s00031-003-0819-6},
}

@article{PX1,
    AUTHOR = {Pickrell, Doug and Xia, Eugene Z.},
     TITLE = {Ergodicity of mapping class group actions on representation
              varieties. {I}. {C}losed surfaces},
   JOURNAL = {Comment. Math. Helv.},
  FJOURNAL = {Commentarii Mathematici Helvetici},
    VOLUME = {77},
      YEAR = {2002},
    NUMBER = {2},
     PAGES = {339--362},
      ISSN = {0010-2571,1420-8946},
   MRCLASS = {22F50 (37A15)},
  MRNUMBER = {1915045},
MRREVIEWER = {A.\ H.\ Dooley},
       DOI = {10.1007/s00014-002-8343-1},
       URL = {https://doi.org/10.1007/s00014-002-8343-1},
}

@incollection{GX,
    AUTHOR = {Goldman, William M. and Xia, Eugene Z.},
     TITLE = {Ergodicity of mapping class group actions on {${\rm
              SU}(2)$}-character varieties},
 BOOKTITLE = {Geometry, rigidity, and group actions},
    SERIES = {Chicago Lectures in Math.},
     PAGES = {591--608},
 PUBLISHER = {Univ. Chicago Press, Chicago, IL},
      YEAR = {2011},
      ISBN = {978-0-226-23788-6; 0-226-23788-5},
   MRCLASS = {22F50 (22D40 37A25 57N05)},
  MRNUMBER = {2807844},
MRREVIEWER = {Yasushi\ Yamashita},
}

@article {funar-marche,
    AUTHOR = {Funar, Louis and March\'{e}, Julien},
     TITLE = {The first {J}ohnson subgroups act ergodically on {$\rm
              SU_2$}-character varieties},
   JOURNAL = {J. Differential Geom.},
  FJOURNAL = {Journal of Differential Geometry},
    VOLUME = {95},
      YEAR = {2013},
    NUMBER = {3},
     PAGES = {407--418},
      ISSN = {0022-040X,1945-743X},
   MRCLASS = {37A25 (22Exx 57M50 58D19)},
  MRNUMBER = {3128990},
       URL = {http://projecteuclid.org/euclid.jdg/1381931734},
}

@article{GX-Torelli,
    AUTHOR = {Goldman, William M. and Xia, Eugene Z.},
     TITLE = {Action of the {J}ohnson-{T}orelli group on representation
              varieties},
   JOURNAL = {Proc. Amer. Math. Soc.},
  FJOURNAL = {Proceedings of the American Mathematical Society},
    VOLUME = {140},
      YEAR = {2012},
    NUMBER = {4},
     PAGES = {1449--1457},
      ISSN = {0002-9939,1088-6826},
   MRCLASS = {57M05 (22D40)},
  MRNUMBER = {2869130},
MRREVIEWER = {\'{A}gota\ Figula},
       DOI = {10.1090/S0002-9939-2011-10972-9},
       URL = {https://doi.org/10.1090/S0002-9939-2011-10972-9},
}

@article{DT,
    AUTHOR = {Deroin, Bertrand and Tholozan, Nicolas},
     TITLE = {Supra-maximal representations from fundamental groups of punctured spheres to {$\mathrm{PSL}(2,\mathbb{R})$}},
   JOURNAL = {Ann. Sci. \'{E}c. Norm. Sup\'{e}r. (4)},
  FJOURNAL = {Annales Scientifiques de l'\'{E}cole Normale Sup\'{e}rieure.
              Quatri\`eme S\'{e}rie},
    VOLUME = {52},
      YEAR = {2019},
    NUMBER = {5},
     PAGES = {1305--1329},
      ISSN = {0012-9593,1873-2151},
   MRCLASS = {57K20 (20F34 22F10)},
  MRNUMBER = {4057784},
MRREVIEWER = {Patricio\ Gallardo},
       DOI = {10.24033/asens.2410},
       URL = {https://doi.org/10.24033/asens.2410},
}

@article{maret,
    AUTHOR = {Maret, Arnaud},
     TITLE = {Ergodicity of the mapping class group action on
              {D}eroin-{T}holozan representations},
   JOURNAL = {Groups Geom. Dyn.},
  FJOURNAL = {Groups, Geometry, and Dynamics},
    VOLUME = {16},
      YEAR = {2022},
    NUMBER = {4},
     PAGES = {1341--1368},
      ISSN = {1661-7207,1661-7215},
   MRCLASS = {58D29 (37C85 37D40 57K20 57M05)},
  MRNUMBER = {4536432},
MRREVIEWER = {Vagn\ Lundsgaard\ Hansen},
       DOI = {10.4171/ggd/695},
       URL = {https://doi.org/10.4171/ggd/695},
}

@incollection {goldman-survey,
    AUTHOR = {Goldman, William M.},
     TITLE = {Mapping class group dynamics on surface group representations},
 BOOKTITLE = {Problems on mapping class groups and related topics},
    SERIES = {Proc. Sympos. Pure Math.},
    VOLUME = {74},
     PAGES = {189--214},
 PUBLISHER = {Amer. Math. Soc., Providence, RI},
      YEAR = {2006},
      ISBN = {978-0-8218-3838-9; 0-8218-3838-5},
   MRCLASS = {57M50 (53C24)},
  MRNUMBER = {2264541},
       DOI = {10.1090/pspum/074/2264541},
       URL = {https://doi.org/10.1090/pspum/074/2264541},
}

@article {Penner,
    AUTHOR = {Penner, Robert C.},
     TITLE = {The decorated {T}eichm\"{u}ller space of punctured surfaces},
   JOURNAL = {Comm. Math. Phys.},
  FJOURNAL = {Communications in Mathematical Physics},
    VOLUME = {113},
      YEAR = {1987},
    NUMBER = {2},
     PAGES = {299--339},
      ISSN = {0010-3616,1432-0916},
   MRCLASS = {32G15 (14H15 53A35)},
  MRNUMBER = {919235},
MRREVIEWER = {C.\ Earle},
       URL = {http://projecteuclid.org/euclid.cmp/1104160216},
}

@article {roger-yang,
    AUTHOR = {Roger, Julien and Yang, Tian},
     TITLE = {The skein algebra of arcs and links and the decorated
              {T}eichm\"{u}ller space},
   JOURNAL = {J. Differential Geom.},
  FJOURNAL = {Journal of Differential Geometry},
    VOLUME = {96},
      YEAR = {2014},
    NUMBER = {1},
     PAGES = {95--140},
      ISSN = {0022-040X,1945-743X},
   MRCLASS = {32G15 (53D55 57M27)},
  MRNUMBER = {3161387},
MRREVIEWER = {J\'{e}r\'{e}my\ Toulisse},
       URL = {http://projecteuclid.org/euclid.jdg/1391192694},
}

@article {mir_gro,
    AUTHOR = {Mirzakhani, Maryam},
     TITLE = {Growth of the number of simple closed geodesics on hyperbolic
              surfaces},
   JOURNAL = {Ann. of Math. (2)},
  FJOURNAL = {Annals of Mathematics. Second Series},
    VOLUME = {168},
      YEAR = {2008},
    NUMBER = {1},
     PAGES = {97--125},
      ISSN = {0003-486X,1939-8980},
   MRCLASS = {32G15},
  MRNUMBER = {2415399},
MRREVIEWER = {Hsian-Hua\ Tseng},
       DOI = {10.4007/annals.2008.168.97},
       URL = {https://doi.org/10.4007/annals.2008.168.97},
}

@unpublished{RY,
	author = {Ryu, Inyoung and Yang, Tian},
	note = {preprint, arXiv:2507.07391},
	title = {Connected components of the space of type-preserving representations},
	year = {2025}}

@unpublished{ryu,
	author = {Ryu, Inyoung},
	note = {preprint, arXiv:2511.13989},
	title = {On totally hyperbolic non-Fuchsian type-preserving representations},
	year = {2025}}

@unpublished{canary_survey_ano,
	author = {Canary, Richard},
	note = {preprint, available at \texttt{https://websites.umich.edu/~canary/lecnotespublic61125.pdf}},
	title = {Anosov Representations: Informal Lecture Notes},
	year = {2025}}

@unpublished{maret-character-variety,
    author = {Maret, Arnaud},
    title = {A note on Character Varieties},
    note = {preprint, available at \texttt{https://arnaudmaret.com/files/character-varieties.pdf}},
    year = {2025}}

@unpublished{schlich_sphere,
	author = {Schlich, Suzanne},
	note = {preprint, arXiv:2503.21859},
	title = {Simple-stable and {B}owditch representations of the four-punctured sphere group in Gromov-hyperbolic spaces},
	year = {2025}}

@incollection {gromov,
    AUTHOR = {Gromov, Michael},
     TITLE = {Hyperbolic groups},
 BOOKTITLE = {Essays in group theory},
    SERIES = {Math. Sci. Res. Inst. Publ.},
    VOLUME = {8},
     PAGES = {75--263},
 PUBLISHER = {Springer, New York},
      YEAR = {1987},
      ISBN = {0-387-96618-8},
   MRCLASS = {20F32 (20F06 20F10 22E40 53C20 57R75 58F17)},
  MRNUMBER = {919829},
MRREVIEWER = {Christopher\ W.\ Stark},
       DOI = {10.1007/978-1-4613-9586-7\{_}3}

@unpublished{flechelles,
	author = {Fl{\'e}chelles, Balthazar},
	note = {preprint, arXiv:2507.09771},
	title = {Primitive stability and the Q-conditions for the rank two free group in hyperbolic d-space},
	year = {2025}}

@unpublished{schlich,
	author = {Schlich, Suzanne},
	note = {preprint, arXiv:2211.14546},
	title = {Equivalence of primitive-stable and {B}owditch actions of the free group of rank two on {G}romov-hyperbolic spaces},
	year = {2022}}

@article{witten-quant,
    AUTHOR = {Witten, Edward},
     TITLE = {On quantum gauge theories in two dimensions},
   JOURNAL = {Comm. Math. Phys.},
  FJOURNAL = {Communications in Mathematical Physics},
    VOLUME = {141},
      YEAR = {1991},
    NUMBER = {1},
     PAGES = {153--209},
      ISSN = {0010-3616},
   MRCLASS = {58G26 (14D20 32G13 58D27 81T13)},
  MRNUMBER = {1133264},
MRREVIEWER = {Dana S. Fine},
       URL = {http://projecteuclid.org/euclid.cmp/1104248198},
}

@article{Kapovich-proper,
	author = {Kapovich, Michael},
	title = {A note on properly discontinuous actions},
    JOURNAL = {San Paulo Journal of Math Sciences},
	year = {2024},
    PAGES = {807--836}}

@article {HJ,
    AUTHOR = {Ho, Nan-Kuo and Jeffrey, Lisa C.},
     TITLE = {The volume of the moduli space of flat connections on a
              nonorientable 2-manifold},
   JOURNAL = {Comm. Math. Phys.},
  FJOURNAL = {Communications in Mathematical Physics},
    VOLUME = {256},
      YEAR = {2005},
    NUMBER = {3},
     PAGES = {539--564},
      ISSN = {0010-3616},
   MRCLASS = {53D30 (58D27)},
  MRNUMBER = {2327950},
MRREVIEWER = {Christopher B. Willett},
       DOI = {10.1007/s00220-005-1344-3},
       URL = {https://doi.org/10.1007/s00220-005-1344-3},
}

@article {Minsky-primitive,
    AUTHOR = {Minsky, Yair N.},
     TITLE = {On dynamics of {$Out(F_n)$} on {$\mathsf{PSL}_2({\mathbb{C}})$}
              characters},
   JOURNAL = {Israel J. Math.},
  FJOURNAL = {Israel Journal of Mathematics},
    VOLUME = {193},
      YEAR = {2013},
    NUMBER = {1},
     PAGES = {47--70},
      ISSN = {0021-2172},
   MRCLASS = {37A15 (20E05 20E36 22E15)},
  MRNUMBER = {3038545},
MRREVIEWER = {Peter Ha\"{\i}ssinsky},
       DOI = {10.1007/s11856-012-0086-0},
       URL = {https://doi.org/10.1007/s11856-012-0086-0},
}

@article {pal-erg,
    AUTHOR = {Palesi, Fr\'{e}d\'{e}ric},
     TITLE = {Ergodic actions of mapping class groups on moduli spaces of
              representations of non-orientable surfaces},
   JOURNAL = {Geom. Dedicata},
  FJOURNAL = {Geometriae Dedicata},
    VOLUME = {151},
      YEAR = {2011},
     PAGES = {107--140},
      ISSN = {0046-5755},
   MRCLASS = {22D40 (20F34 57M05)},
  MRNUMBER = {2780741},
MRREVIEWER = {William Goldman},
       DOI = {10.1007/s10711-010-9522-7},
       URL = {https://doi.org/10.1007/s10711-010-9522-7},
}

@article{GLX,
    AUTHOR = {Goldman, William M. and Lawton, Sean and Xia, Eugene Z.},
     TITLE = {The mapping class group action on {$\mathsf{SU}(3)$}-character
              varieties},
   JOURNAL = {Ergodic Theory Dynam. Systems},
  FJOURNAL = {Ergodic Theory and Dynamical Systems},
    VOLUME = {41},
      YEAR = {2021},
    NUMBER = {8},
     PAGES = {2382--2396},
      ISSN = {0143-3857},
   MRCLASS = {22F50 (22D40 37A25)},
  MRNUMBER = {4283277},
MRREVIEWER = {Cheng Zheng},
       DOI = {10.1017/etds.2020.50},
       URL = {https://doi.org/10.1017/etds.2020.50},
}

@article {LR,
    AUTHOR = {Leininger, Christopher J. and Russell, Jacob},
     TITLE = {Pseudo-{A}nosov subgroups of general fibered 3-manifold
              groups},
   JOURNAL = {Trans. Amer. Math. Soc. Ser. B},
  FJOURNAL = {Transactions of the American Mathematical Society. Series B},
    VOLUME = {10},
      YEAR = {2023},
     PAGES = {1141--1172},
   MRCLASS = {57K20 (20E08 20F65 57K30)},
  MRNUMBER = {4632569},
MRREVIEWER = {Thomas Koberda},
       DOI = {10.1090/btran/157},
       URL = {https://doi.org/10.1090/btran/157},
}

@unpublished{Mu-Pen,
	author = {Mulase,Motohico  and  Penkava, Michael},
	note = {preprint, arXiv:math/0212012},
	title = {Volume of representation varieties},
	year = {2002}}

@article {kim-lee,
    AUTHOR = {Kim, Inkang and Lee, Michelle},
     TITLE = {Separable-stable representations of a compression body},
   JOURNAL = {Topology Appl.},
  FJOURNAL = {Topology and its Applications},
    VOLUME = {206},
      YEAR = {2016},
     PAGES = {171--184},
      ISSN = {0166-8641},
   MRCLASS = {57S25 (51M10)},
  MRNUMBER = {3494440},
MRREVIEWER = {Marja K. Kankaanrinta},
       DOI = {10.1016/j.topol.2016.03.029},
       URL = {https://doi.org/10.1016/j.topol.2016.03.029},
}

@article {TWZ-corr,
    AUTHOR = {Tan, Ser Peow and Wong, Yan Loi and Zhang, Ying},
     TITLE = {Corrigendum to ``{G}eneralized {M}arkoff maps and
              {M}c{S}hane's identity'' [{A}dv. {M}ath. 217 (2008) 761--813]
              [MR2370281]},
   JOURNAL = {Adv. Math.},
  FJOURNAL = {Advances in Mathematics},
    VOLUME = {222},
      YEAR = {2009},
    NUMBER = {6},
     PAGES = {2270--2271},
      ISSN = {0001-8708,1090-2082},
   MRCLASS = {57M50 (20H10)},
  MRNUMBER = {2562784},
       DOI = {10.1016/j.aim.2009.06.024},
       URL = {https://doi.org/10.1016/j.aim.2009.06.024},
}

@article {TWZ,
    AUTHOR = {Tan, Ser Peow and Wong, Yan Loi and Zhang, Ying},
     TITLE = {Generalized {M}arkoff maps and {M}c{S}hane's identity},
   JOURNAL = {Adv. Math.},
  FJOURNAL = {Advances in Mathematics},
    VOLUME = {217},
      YEAR = {2008},
    NUMBER = {2},
     PAGES = {761--813},
      ISSN = {0001-8708,1090-2082},
   MRCLASS = {57M50 (20H10)},
  MRNUMBER = {2370281},
MRREVIEWER = {Colin\ C.\ Adams},
       DOI = {10.1016/j.aim.2007.09.004},
       URL = {https://doi.org/10.1016/j.aim.2007.09.004},
}

@article {LMP,
    AUTHOR = {Lawton, Sean and Maloni, Sara and Palesi, Fr\'{e}d\'{e}ric},
     TITLE = {Dynamics on the {$\mathrm{SU}(2,1)$}-character variety of the one-holed torus},
   JOURNAL = {Moduli},
  FJOURNAL = {Moduli},
    VOLUME = {2},
      YEAR = {2025},
      PAGES = {1--38},
      ISSN = {2949-7647,2977-1382},
   MRCLASS = {57M60 (14M35 57K20 57M50)},
  MRNUMBER = {4892272},
MRREVIEWER = {Thilo\ Kuessner},
       DOI = {10.1112/mod.2025.3},
       URL = {https://doi.org/10.1112/mod.2025.3},
}

@article {MPY,
    AUTHOR = {Maloni, Sara and Palesi, Fr\'{e}d\'{e}ric and Yang, Tian},
     TITLE = {On type-preserving representations of thrice punctured
              projective plane group},
   JOURNAL = {J. Differential Geom.},
  FJOURNAL = {Journal of Differential Geometry},
    VOLUME = {119},
      YEAR = {2021},
    NUMBER = {3},
     PAGES = {421--457},
      ISSN = {0022-040X,1945-743X},
   MRCLASS = {20F65 (57K20)},
  MRNUMBER = {4333027},
MRREVIEWER = {Lvzhou\ Chen},
       DOI = {10.4310/jdg/1635368618},
       URL = {https://doi.org/10.4310/jdg/1635368618},
}

@article {MP,
    AUTHOR = {Maloni, Sara and Palesi, Fr\'{e}d\'{e}ric},
     TITLE = {On the character variety of the three-holed projective plane},
   JOURNAL = {Conform. Geom. Dyn.},
  FJOURNAL = {Conformal Geometry and Dynamics. An Electronic Journal of the
              American Mathematical Society},
    VOLUME = {24},
      YEAR = {2020},
     PAGES = {68--108},
      ISSN = {1088-4173},
   MRCLASS = {57K20 (20E05 37A15)},
  MRNUMBER = {4071233},
MRREVIEWER = {Samuel\ Aaron\ Ballas},
       DOI = {10.1090/ecgd/349},
       URL = {https://doi.org/10.1090/ecgd/349},
}

@article {MPT,
    AUTHOR = {Maloni, Sara and Palesi, Fr\'{e}d\'{e}ric and Tan, Ser Peow},
     TITLE = {On the character variety of the four-holed sphere},
   JOURNAL = {Groups Geom. Dyn.},
  FJOURNAL = {Groups, Geometry, and Dynamics},
    VOLUME = {9},
      YEAR = {2015},
    NUMBER = {3},
     PAGES = {737--782},
      ISSN = {1661-7207,1661-7215},
   MRCLASS = {57M50 (20F65)},
  MRNUMBER = {3420542},
MRREVIEWER = {Thomas\ Koberda},
       DOI = {10.4171/GGD/326},
       URL = {https://doi.org/10.4171/GGD/326},
}

@article {sch_the,
    AUTHOR = {Scharlemann, Martin},
     TITLE = {The complex of curves on nonorientable surfaces},
   JOURNAL = {J. London Math. Soc. (2)},
  FJOURNAL = {Journal of the London Mathematical Society. Second Series},
    VOLUME = {25},
      YEAR = {1982},
    NUMBER = {1},
     PAGES = {171--184},
      ISSN = {0024-6107,1469-7750},
   MRCLASS = {57R30 (58F18)},
  MRNUMBER = {645874},
MRREVIEWER = {Elmar\ Vogt},
       DOI = {10.1112/jlms/s2-25.1.171},
       URL = {https://doi.org/10.1112/jlms/s2-25.1.171},
}

@article{gol-erg,
	author = {Goldman, William M.},
	doi = {10.2307/2952454},
	fjournal = {Annals of Mathematics. Second Series},
	issn = {0003-486X},
	journal = {Ann. of Math. (2)},
	mrclass = {58D29 (57M99 58F11)},
	mrnumber = {1491446},
	mrreviewer = {Ambar N. Sengupta},
	number = {3},
	pages = {475--507},
	title = {Ergodic theory on moduli spaces},
	url = {https://doi.org/10.2307/2952454},
	volume = {146},
	year = {1997}}

@incollection {palesi-no,
    AUTHOR = {Palesi, Fr\'{e}d\'{e}ric},
     TITLE = {Connected components of {${\rm PGL}(2,\mathbb{R})$}-representation
              spaces of non-orientable surfaces},
 BOOKTITLE = {Geometry, topology and dynamics of character varieties},
    SERIES = {Lect. Notes Ser. Inst. Math. Sci. Natl. Univ. Singap.},
    VOLUME = {23},
     PAGES = {281--295},
 PUBLISHER = {World Sci. Publ., Hackensack, NJ},
      YEAR = {2012},
      ISBN = {978-981-4401-35-7; 981-4401-35-8},
   MRCLASS = {57M05},
  MRNUMBER = {2987621},
MRREVIEWER = {Paul\ A.\ Kirk},
       DOI = {10.1142/9789814401364\{_}0008}

@incollection {T-min,
    AUTHOR = {Thurston, William P.},
     TITLE = {Minimal stretch maps between hyperbolic surfaces},
 BOOKTITLE = {Collected works of {W}illiam {P}. {T}hurston with commentary.
              {V}ol. {I}. {F}oliations, surfaces and differential geometry},
     PAGES = {533--585},
      NOTE = {1986 preprint, 1998 eprint},
 PUBLISHER = {Amer. Math. Soc., Providence, RI},
      YEAR = {[2022] \copyright 2022},
      ISBN = {978-1-4704-6388-5; [9781470468330]; [9781470451646]},
   MRCLASS = {57K20},
  MRNUMBER = {4554454},
}

@article {DT-domi,
    AUTHOR = {Deroin, Bertrand and Tholozan, Nicolas},
     TITLE = {Dominating surface group representations by {F}uchsian ones},
   JOURNAL = {Int. Math. Res. Not. IMRN},
  FJOURNAL = {International Mathematics Research Notices. IMRN},
      YEAR = {2016},
    NUMBER = {13},
     PAGES = {4145--4166},
      ISSN = {1073-7928,1687-0247},
   MRCLASS = {53C20},
  MRNUMBER = {3544632},
MRREVIEWER = {Bruno\ P.\ Zimmermann},
       DOI = {10.1093/imrn/rnv275},
       URL = {https://doi.org/10.1093/imrn/rnv275},
}

@article {GKW,
    AUTHOR = {Gu\'{e}ritaud, Fran\c{c}ois and Kassel, Fanny and Wolff,
              Maxime},
     TITLE = {Compact anti--de {S}itter 3-manifolds and folded hyperbolic
              structures on surfaces},
   JOURNAL = {Pacific J. Math.},
  FJOURNAL = {Pacific Journal of Mathematics},
    VOLUME = {275},
      YEAR = {2015},
    NUMBER = {2},
     PAGES = {325--359},
      ISSN = {0030-8730,1945-5844},
   MRCLASS = {57M07 (20H10 30F45 30F60 32G15 53C50)},
  MRNUMBER = {3347373},
MRREVIEWER = {Anthony\ Weaver},
       DOI = {10.2140/pjm.2015.275.325},
       URL = {https://doi.org/10.2140/pjm.2015.275.325},
}

@article {KTZ,
    AUTHOR = {Kim, Sungwoon and Tan, Ser Peow and Zhang, Tengren},
     TITLE = {Weakly positive and directed {A}nosov representations},
   JOURNAL = {Adv. Math.},
  FJOURNAL = {Advances in Mathematics},
    VOLUME = {408},
      YEAR = {2022},
    NUMBER = {part B},
     PAGES = {Paper No. 108611, 72},
      ISSN = {0001-8708,1090-2082},
   MRCLASS = {20C99 (20E05 57M60)},
  MRNUMBER = {4462941},
MRREVIEWER = {Yasushi\ Yamashita},
       DOI = {10.1016/j.aim.2022.108611},
       URL = {https://doi.org/10.1016/j.aim.2022.108611},
}

@article {KK,
    AUTHOR = {Kim, Inkang and Kim, Sungwoon},
     TITLE = {Primitive stable representations in higher rank semisimple
              {L}ie groups},
   JOURNAL = {Rev. Mat. Complut.},
  FJOURNAL = {Revista Matem\'{a}tica Complutense},
    VOLUME = {34},
      YEAR = {2021},
    NUMBER = {3},
     PAGES = {715--745},
      ISSN = {1139-1138,1988-2807},
   MRCLASS = {22F30 (20E05 32G15)},
  MRNUMBER = {4302239},
MRREVIEWER = {Fr\'{e}d\'{e}ric\ Palesi},
       DOI = {10.1007/s13163-020-00372-w},
       URL = {https://doi.org/10.1007/s13163-020-00372-w},
}

@article {tsou,
    AUTHOR = {Tsouvalas, Konstantinos},
     TITLE = {Quasi-isometric embeddings inapproximable by {A}nosov
              representations},
   JOURNAL = {J. Inst. Math. Jussieu},
  FJOURNAL = {Journal of the Institute of Mathematics of Jussieu. JIMJ.
              Journal de l'Institut de Math\'{e}matiques de Jussieu},
    VOLUME = {22},
      YEAR = {2023},
    NUMBER = {5},
     PAGES = {2497--2514},
      ISSN = {1474-7480,1475-3030},
   MRCLASS = {22E40 (20H10)},
  MRNUMBER = {4624969},
MRREVIEWER = {O.\ V.\ Shvartsman},
       DOI = {10.1017/S1474748021000645},
       URL = {https://doi.org/10.1017/S1474748021000645},
}

@unpublished{gui-phd,
    author = {Guichard, Olivier},
    title = {D\'eformation de sous-groupes discrets de groupes de rang un},
    note = {PhD thesis, Universit\'e' Paris 7},
	year = {2004},
}

@unpublished{ser_pri,
	author = {Series, Caroline},
	note = {preprint, arXiv:1901.01396,  to appear in New Aspects of Teichm\"uller Theory, Eds K. Ohshika, N. Kawazumi S. Yamada., Advanced Studies in Pure Mathematics},
	title = {Primitive stability and {B}owditch's {BQ} conditions are equivalent},
	year = {2024}}

@article{KLP1,
	author = {Kapovich, Michael and Leeb, Bernhard and Porti, Joan},
	doi = {10.2140/gt.2018.22.157},
	fjournal = {Geometry \& Topology},
	issn = {1465-3060},
	journal = {Geom. Topol.},
	mrclass = {53C35 (22E40 37B05 51E24)},
	mrnumber = {3720343},
	number = {1},
	pages = {157--234},
	title = {Dynamics on flag manifolds: domains of proper discontinuity and cocompactness},
	url = {https://doi.org/10.2140/gt.2018.22.157},
	volume = {22},
	year = {2018}}

@article{lee_dyn,
	 AUTHOR = {Lee, Michelle},
     TITLE = {Dynamics on the {${\rm PSL}(2,\mathbb{C})$}-character variety of a
              compression body},
   JOURNAL = {Algebr. Geom. Topol.},
  FJOURNAL = {Algebraic \& Geometric Topology},
    VOLUME = {14},
      YEAR = {2014},
    NUMBER = {4},
     PAGES = {2149--2179},
      ISSN = {1472-2747,1472-2739},
   MRCLASS = {57M60 (57M50)},
  MRNUMBER = {3331612},
MRREVIEWER = {Charalampos\ Charitos},
       DOI = {10.2140/agt.2014.14.2149},
       URL = {https://doi.org/10.2140/agt.2014.14.2149},
}

@article{lee-twisted,
	  AUTHOR = {Lee, Michelle},
     TITLE = {Dynamics on the {${\rm PSL}(2,\mathbb{C})$}-character variety of
              a twisted {$I$}-bundle},
   JOURNAL = {Groups Geom. Dyn.},
  FJOURNAL = {Groups, Geometry, and Dynamics},
    VOLUME = {9},
      YEAR = {2015},
    NUMBER = {1},
     PAGES = {187--201},
      ISSN = {1661-7207,1661-7215},
   MRCLASS = {57M60},
  MRNUMBER = {3343351},
MRREVIEWER = {Michael\ Heusener},
       DOI = {10.4171/GGD/310},
       URL = {https://doi.org/10.4171/GGD/310},
}

@inproceedings {sullivan,
    AUTHOR = {Sullivan, Dennis},
     TITLE = {On the ergodic theory at infinity of an arbitrary discrete
              group of hyperbolic motions},
 BOOKTITLE = {Riemann surfaces and related topics: {P}roceedings of the 1978
              {S}tony {B}rook {C}onference ({S}tate {U}niv. {N}ew {Y}ork,
              {S}tony {B}rook, {N}.{Y}., 1978)},
    SERIES = {Ann. of Math. Stud., No. 97},
     PAGES = {465--496},
 PUBLISHER = {Princeton Univ. Press, Princeton, NJ},
      YEAR = {1981},
      ISBN = {0-691-08264-2},
   MRCLASS = {58F11 (22E40 30C60 53C30 57R99)},
  MRNUMBER = {624833},
MRREVIEWER = {M.\ Rees},
}

@incollection {thurston2,
    AUTHOR = {Thurston, William P.},
     TITLE = {Hyperbolic structures on 3-manifolds, {II}: {S}urface groups
              and 3-manifolds which fiber over the circle},
 BOOKTITLE = {Collected works of {W}illiam {P}. {T}hurston with commentary.
              {V}ol. {II}. 3-manifolds, complexity and geometric group
              theory},
     PAGES = {79--110},
      NOTE = {August 1986 preprint, January 1998 eprint},
 PUBLISHER = {Amer. Math. Soc., Providence, RI},
      YEAR = {[2022] \copyright 2022},
      ISBN = {978-1-4704-6389-2; [9781470468347]; [9781470451646]},
   MRCLASS = {57K32},
  MRNUMBER = {4556467},
}

@article{kas_coor,
	author = {Kashaev, Rinat},
	journal = {Math. Res. Lett.},
	number = {1},
	pages = {23--36},
	title = {Coordinates for the moduli space of flat {${\rm PSL}(2,\mathbb{R})$}-connections},
	volume = {12},
	year = {2005}}

@article {mar-wol2,
    AUTHOR = {March\'{e}, Julien and Wolff, Maxime},
     TITLE = {Six-point configurations in the hyperbolic plane and
              ergodicity of the mapping class group},
   JOURNAL = {Groups Geom. Dyn.},
  FJOURNAL = {Groups, Geometry, and Dynamics},
    VOLUME = {13},
      YEAR = {2019},
    NUMBER = {2},
     PAGES = {731--766},
      ISSN = {1661-7207},
   MRCLASS = {57M05 (20H10 30F45 30F60 58D29)},
  MRNUMBER = {3950649},
MRREVIEWER = {Thomas Koberda},
       DOI = {10.4171/GGD/503},
       URL = {https://doi.org/10.4171/GGD/503},
}

@article {bow_geo,
    AUTHOR = {Bowditch, Brian H.},
     TITLE = {Geometrical finiteness with variable negative curvature},
   JOURNAL = {Duke Math. J.},
  FJOURNAL = {Duke Mathematical Journal},
    VOLUME = {77},
      YEAR = {1995},
    NUMBER = {1},
     PAGES = {229--274},
      ISSN = {0012-7094,1547-7398},
   MRCLASS = {53C21 (53C20 57R99)},
  MRNUMBER = {1317633},
MRREVIEWER = {Boris\ N.\ Apanasov},
       DOI = {10.1215/S0012-7094-95-07709-6},
       URL = {https://doi.org/10.1215/S0012-7094-95-07709-6},
}

@inproceedings{atiyah-bott,
	author = {Michael F. Atiyah and Raoul Bott},
	booktitle = {Differential analysis},
	pages = {175-186},
	publisher = {Oxford},
	title = {The index theorem for manifolds with boundary},
	year = {1964}}

@article{bowditch,
	author = {Bowditch, Brian H.},
	coden = {JLMSAK},
	fjournal = {Journal of the London Mathematical Society. Second Series},
	issn = {0024-6107},
	journal = {J. London Math. Soc. (2)},
	mrclass = {57M50 (57N05)},
	mrnumber = {MR1149015 (93f:57014)},
	mrreviewer = {F. E. A. Johnson},
	number = {3},
	pages = {553--565},
	title = {Singular {E}uclidean structures on surfaces},
	volume = {44},
	year = {1991}}

@article {MM1,
    AUTHOR = {Masur, Howard A. and Minsky, Yair N.},
     TITLE = {Geometry of the complex of curves. {I}. {H}yperbolicity},
   JOURNAL = {Invent. Math.},
  FJOURNAL = {Inventiones Mathematicae},
    VOLUME = {138},
      YEAR = {1999},
    NUMBER = {1},
     PAGES = {103--149},
      ISSN = {0020-9910,1432-1297},
   MRCLASS = {57M50 (20F67 30F60 32G15)},
  MRNUMBER = {1714338},
MRREVIEWER = {Darryl\ McCullough},
       DOI = {10.1007/s002220050343},
       URL = {https://doi.org/10.1007/s002220050343},
}

@article{goldman-symplectic,
	author = {Goldman, William M.},
	coden = {ADMTA4},
	fjournal = {Advances in Mathematics},
	issn = {0001-8708},
	journal = {Adv. in Math.},
	mrclass = {32G15 (57M05)},
	mrnumber = {MR762512 (86i:32042)},
	mrreviewer = {C. Earle},
	number = {2},
	pages = {200--225},
	title = {The symplectic nature of fundamental groups of surfaces},
	volume = {54},
	year = {1984}}

@article{gui_ano,
	author = {Guichard, Olivier and Wienhard, Anna},
	coden = {INVMBH},
	doi = {10.1007/s00222-012-0382-7},
	fjournal = {Inventiones Mathematicae},
	issn = {0020-9910},
	journal = {Invent. Math.},
	mrclass = {22F30 (32G15 53C30 53D25)},
	mrnumber = {2981818},
	mrreviewer = {Pablo Su{\'a}rez-Serrato},
	number = {2},
	pages = {357--438},
	title = {Anosov representations: domains of discontinuity and applications},
	url = {http://dx.doi.org/10.1007/s00222-012-0382-7},
	volume = {190},
	year = {2012}}

@article{H,
	author = {R. S. Hamilton},
	journal = {Bull. Amer. Math. Soc. (N. S.)},
	pages = {65-222},
	title = {The Inverse Function Theorem of {Nash} and {Moser}},
	volume = {7},
	year = {1982}}

@article{volume,
	author = {Krasnov, Kirill and Schlenker, Jean-Marc},
	coden = {CMPHAY},
	doi = {10.1007/s00220-008-0423-7},
	fjournal = {Communications in Mathematical Physics},
	issn = {0010-3616},
	journal = {Comm. Math. Phys.},
	mrclass = {53C65 (32G81 53C80 58Dxx)},
	mrnumber = {MR2386723},
	number = {3},
	pages = {637--668},
	title = {On the renormalized volume of hyperbolic 3-manifolds},
	url = {http://dx.doi.org/10.1007/s00220-008-0423-7},
	volume = {279},
	year = {2008}}

@unpublished{tholozan-wang,
	author = {Tholozan, Nicolas and Wang, Tianqi},
	note = {preprint arXiv:2307.02934},
	title = {Simple Anosov representations of closed surface groups},
	year = {2023}}

@article {GLW,
    AUTHOR = {Guichard, Olivier and Labourie, Fran\c{c}ois and Wienhard,
              Anna},
     TITLE = {Positivity and representations of surface groups},
   JOURNAL = {Forum Math. Pi},
  FJOURNAL = {Forum of Mathematics. Pi},
    VOLUME = {14},
      YEAR = {2026},
     PAGES = {Paper No. e6, 38},
      ISSN = {2050-5086},
   MRCLASS = {22E40 (22E15 22F30 57S20)},
  MRNUMBER = {5029150},
       DOI = {10.1017/fmp.2025.10022},
       URL = {https://doi.org/10.1017/fmp.2025.10022},
}

@article {labourie-mcshane,
    AUTHOR = {Labourie, Fran\c{c}ois and McShane, Gregory},
     TITLE = {Cross ratios and identities for higher
              {T}eichm\"{u}ller-{T}hurston theory},
   JOURNAL = {Duke Math. J.},
  FJOURNAL = {Duke Mathematical Journal},
    VOLUME = {149},
      YEAR = {2009},
    NUMBER = {2},
     PAGES = {279--345},
      ISSN = {0012-7094,1547-7398},
   MRCLASS = {32G15 (32G07)},
  MRNUMBER = {2541705},
MRREVIEWER = {Javier\ Aramayona},
       DOI = {10.1215/00127094-2009-040},
       URL = {https://doi.org/10.1215/00127094-2009-040},
}

@article{N,
	author = {L. Nirenberg},
	journal = {Comm. Pure Appl. Math},
	pages = {337-394},
	title = {The {Weyl} and {Minkowski} Problem in Differential Geometry in the Large},
	volume = {6},
	year = {1953}}

@book{O,
	author = {B. O'Neill},
	publisher = {Academic Press},
	title = {Semi-Riemannian Geometry},
	year = {1983}}

@book{otal-hyperbolisation,
	    AUTHOR = {Otal, Jean-Pierre},
     TITLE = {The hyperbolization theorem for fibered 3-manifolds},
    SERIES = {SMF/AMS Texts and Monographs},
    VOLUME = {7},
      NOTE = {Translated from the 1996 French original by Leslie D. Kay},
 PUBLISHER = {American Mathematical Society, Providence, RI; Soci\'{e}t\'{e}
              Math\'{e}matique de France, Paris},
      YEAR = {2001},
     PAGES = {xiv+126},
      ISBN = {0-8218-2153-9},
   MRCLASS = {57M50 (20E08 30F40 30F60 57M07)},
  MRNUMBER = {1855976},
}

@book{Po,
	author = {Aleksei V. Pogorelov},
	note = {Translations of Mathematical Monographs. Vol. 35},
	publisher = {American Mathematical Society},
	title = {Extrinsic Geometry of Convex Surfaces},
	year = {1973}}

@book{S,
	author = {M. Spivak},
	publisher = {Publish or perish},
	title = {A comprehensive introduction to geometry, Vol.I-V},
	year = {1970-1975}}

@book {Knapp,
    AUTHOR = {Knapp, Anthony W.},
     TITLE = {Lie groups beyond an introduction},
    SERIES = {Progress in Mathematics},
    VOLUME = {140},
   EDITION = {Second},
 PUBLISHER = {Birkh\"{a}user Boston, Inc., Boston, MA},
      YEAR = {2002},
     PAGES = {xviii+812},
      ISBN = {0-8176-4259-5},
   MRCLASS = {22-01},
  MRNUMBER = {1920389},
}

@article{bow_mar,
	author = {Bowditch, Brian H.},
	coden = {PLMTAL},
	doi = {10.1112/S0024611598000604},
	fjournal = {Proceedings of the London Mathematical Society. Third Series},
	issn = {0024-6115},
	journal = {Proc. London Math. Soc. (3)},
	mrclass = {57M50 (11D25)},
	mrnumber = {1643429 (99f:57014)},
	mrreviewer = {Athanase Papadopoulos},
	number = {3},
	pages = {697--736},
	title = {Markoff triples and quasi-{F}uchsian groups},
	url = {http://dx.doi.org/10.1112/S0024611598000604},
	volume = {77},
	year = {1998}}

@article{bro_sel,
	author = {Bromberg, Kenneth W. and Holt, John},
	coden = {JDGEAS},
	fjournal = {Journal of Differential Geometry},
	issn = {0022-040X},
	journal = {J. Differential Geom.},
	mrclass = {57M50 (57N10)},
	mrnumber = {1871491 (2002i:57015)},
	mrreviewer = {James W. Anderson},
	number = {1},
	pages = {47--65},
	title = {Self-bumping of deformation spaces of hyperbolic 3-manifolds},
	url = {http://projecteuclid.org/getRecord?id=euclid.jdg/1090348089},
	volume = {57},
	year = {2001}}

@article {KP,
    AUTHOR = {Kassel, Fanny and Potrie, Rafael},
     TITLE = {Eigenvalue gaps for hyperbolic groups and semigroups},
   JOURNAL = {J. Mod. Dyn.},
  FJOURNAL = {Journal of Modern Dynamics},
    VOLUME = {18},
      YEAR = {2022},
     PAGES = {161--208},
      ISSN = {1930-5311,1930-532X},
   MRCLASS = {20M30 (20F67 22E40 37D25 37D30 37H15)},
  MRNUMBER = {4446005},
MRREVIEWER = {Khosro\ Tajbakhsh},
       DOI = {10.3934/jmd.2022008},
       URL = {https://doi.org/10.3934/jmd.2022008},
}

@article {BPS,
    AUTHOR = {Bochi, Jairo and Potrie, Rafael and Sambarino, Andr\'{e}s},
     TITLE = {Anosov representations and dominated splittings},
   JOURNAL = {J. Eur. Math. Soc. (JEMS)},
  FJOURNAL = {Journal of the European Mathematical Society (JEMS)},
    VOLUME = {21},
      YEAR = {2019},
    NUMBER = {11},
     PAGES = {3343--3414},
      ISSN = {1435-9855,1435-9863},
   MRCLASS = {22E40 (20F67 37B99 37D30 53C35)},
  MRNUMBER = {4012341},
MRREVIEWER = {Alejandro\ Ucan-Puc},
       DOI = {10.4171/JEMS/905},
       URL = {https://doi.org/10.4171/JEMS/905},
}

@article {GGKW,
    AUTHOR = {Gu\'{e}ritaud, Fran\c{c}ois and Guichard, Olivier and Kassel,
              Fanny and Wienhard, Anna},
     TITLE = {Anosov representations and proper actions},
   JOURNAL = {Geom. Topol.},
  FJOURNAL = {Geometry \& Topology},
    VOLUME = {21},
      YEAR = {2017},
    NUMBER = {1},
     PAGES = {485--584},
      ISSN = {1465-3060,1364-0380},
   MRCLASS = {37D40 (20F67 22E40 57S30)},
  MRNUMBER = {3608719},
       DOI = {10.2140/gt.2017.21.485},
       URL = {https://doi.org/10.2140/gt.2017.21.485},
}

@article {LaSi,
    AUTHOR = {Lawton, Sean and Sikora, Adam S.},
     TITLE = {Varieties of characters},
   JOURNAL = {Algebr. Represent. Theory},
  FJOURNAL = {Algebras and Representation Theory},
    VOLUME = {20},
      YEAR = {2017},
    NUMBER = {5},
     PAGES = {1133--1141},
      ISSN = {1386-923X,1572-9079},
   MRCLASS = {20C15 (14D20 14L30)},
  MRNUMBER = {3707908},
MRREVIEWER = {John\ T.\ Cullinan},
       DOI = {10.1007/s10468-017-9679-y},
       URL = {https://doi.org/10.1007/s10468-017-9679-y},
}

@article {CFLO,
    AUTHOR = {Casimiro, Ana and Florentino, Carlos and Lawton, Sean and
              Oliveira, Andr\'{e}},
     TITLE = {Topology of moduli spaces of free group representations in
              real reductive groups},
   JOURNAL = {Forum Math.},
  FJOURNAL = {Forum Mathematicum},
    VOLUME = {28},
      YEAR = {2016},
    NUMBER = {2},
     PAGES = {275--294},
      ISSN = {0933-7741,1435-5337},
   MRCLASS = {14D20 (14L30 14P25 20G20 32G13)},
  MRNUMBER = {3466569},
MRREVIEWER = {Sangjib\ Kim},
       DOI = {10.1515/forum-2014-0049},
       URL = {https://doi.org/10.1515/forum-2014-0049},
}

@book{fri_vor,
	author = {Fricke, Robert and Klein, Felix},
	publisher = {Teubner},
	title = {Vorlesungen {\"u}ber die Theorie der automorphen Funktionen},
	year = {1897}}

@article{goldman-flows,
	author = {Goldman, William M.},
	coden = {INVMBH},
	doi = {10.1007/BF01389091},
	fjournal = {Inventiones Mathematicae},
	issn = {0020-9910},
	journal = {Invent. Math.},
	mrclass = {32G15 (57M05)},
	mrnumber = {846929 (87j:32069)},
	mrreviewer = {C. Earle},
	number = {2},
	pages = {263--302},
	title = {Invariant functions on {L}ie groups and {H}amiltonian flows of surface group representations},
	url = {http://dx.doi.org/10.1007/BF01389091},
	volume = {85},
	year = {1986}}

@article{hit_lie,
	author = {Hitchin, N.J.},
	journal = {Topology},
	number = {3},
	pages = {449--473},
	publisher = {Elsevier},
	title = {Lie groups and Teichm{\"u}ller space},
	volume = {31},
	year = {1992}}

@article{lab_ano,
	author = {Labourie, Fran{\c{c}}ois},
	doi = {10.1007/s00222-005-0487-3},
	journal = {Inventiones Mathematicae},
	number = {1},
	pages = {51--114},
	title = {Anosov flows, surface groups and curves in projective space},
	volume = {165},
	year = {2006}}

@article{mil_ont,
    AUTHOR = {Milnor, John},
     TITLE = {On the existence of a connection with curvature zero},
   JOURNAL = {Comment. Math. Helv.},
  FJOURNAL = {Commentarii Mathematici Helvetici},
    VOLUME = {32},
      YEAR = {1958},
     PAGES = {215--223},
      ISSN = {0010-2571,1420-8946},
   MRCLASS = {53.00},
  MRNUMBER = {95518},
MRREVIEWER = {L.\ Auslander},
       DOI = {10.1007/BF02564579},
       URL = {https://doi.org/10.1007/BF02564579},
}

@unpublished{geom,
	author = {Soma, Teruhiko},
	note = {arXiv:math/0702725},
	title = {Geometric limits of quasi-Fuchsian groups},
	year = {2007}}

@book {thu_geometry,
    AUTHOR = {Thurston, William P.},
     TITLE = {The geometry and topology of three-manifolds. {V}ol. {IV}},
      NOTE = {Edited and with a preface by Steven P. Kerckhoff and a chapter
              by J. W. Milnor},
 PUBLISHER = {American Mathematical Society, Providence, RI},
      YEAR = {[2022] \copyright 2022},
     PAGES = {xvii+316},
      ISBN = {978-1-4704-6391-5; [9781470468361]; [9781470451646]},
   MRCLASS = {57K32 (53C15 57R30)},
  MRNUMBER = {4554426},
MRREVIEWER = {Thilo\ Kuessner},
}

@article{PS,
	author = {Potrie, Rafael and Sambarino, Andr{\'e}s},
	journal = {Inventiones mathematicae},
	number = {3},
	pages = {885--925},
	publisher = {Springer},
	title = {Eigenvalues and entropy of a Hitchin representation},
	volume = {209},
	year = {2017}}

@article {gelander,
    AUTHOR = {Gelander, Tsachik},
     TITLE = {{${\rm Aut}(F_n)$} actions on representation spaces},
   JOURNAL = {J. Algebra},
  FJOURNAL = {Journal of Algebra},
    VOLUME = {656},
      YEAR = {2024},
     PAGES = {206--225},
      ISSN = {0021-8693,1090-266X},
   MRCLASS = {20E36 (20C15 20E05 22C05 22E46)},
  MRNUMBER = {4759502},
MRREVIEWER = {Yassine\ Guerch},
       DOI = {10.1016/j.jalgebra.2024.04.005},
       URL = {https://doi.org/10.1016/j.jalgebra.2024.04.005},
}

@incollection{can_dyn,
	author = {Canary, Richard D.},
	booktitle = {Handbook of group actions. {V}ol. {II}},
	mrclass = {53C35 (20F28 20F67 37C85 57M50)},
	mrnumber = {3382028},
	mrreviewer = {Caglar Uyanik},
	pages = {175--200},
	publisher = {Int. Press, Somerville, MA},
	series = {Adv. Lect. Math. (ALM)},
	title = {Dynamics on character varieties: a survey},
	volume = {32},
	year = {2015}}

@article{yan_ont,
	author = {Yang, Tian},
	journal = {Geometry \& Topology},
	number = {2},
	pages = {1213--1255},
	publisher = {Mathematical Sciences Publishers},
	title = {On type-preserving representations of the four-punctured sphere group},
	volume = {20},
	year = {2016}}

@article {ohshika,
    AUTHOR = {Ohshika, Ken'ichi},
     TITLE = {Realising end invariants by limits of minimally parabolic,
              geometrically finite groups},
   JOURNAL = {Geom. Topol.},
  FJOURNAL = {Geometry \& Topology},
    VOLUME = {15},
      YEAR = {2011},
    NUMBER = {2},
     PAGES = {827--890},
      ISSN = {1465-3060,1364-0380},
   MRCLASS = {57M50 (30F40)},
  MRNUMBER = {2821565},
MRREVIEWER = {Bruno\ P.\ Zimmermann},
       DOI = {10.2140/gt.2011.15.827},
       URL = {https://doi.org/10.2140/gt.2011.15.827},
}

@article {namazi-souto,
    AUTHOR = {Namazi, Hossein and Souto, Juan},
     TITLE = {Non-realizability and ending laminations: proof of the density  conjecture},
   JOURNAL = {Acta Math.},
  FJOURNAL = {Acta Mathematica},
    VOLUME = {209},
      YEAR = {2012},
    NUMBER = {2},
     PAGES = {323--395},
      ISSN = {0001-5962,1871-2509},
   MRCLASS = {30F40 (20H10 57M50)},
  MRNUMBER = {3001608},
MRREVIEWER = {Majid\ Heydarpour},
       DOI = {10.1007/s11511-012-0088-0},
       URL = {https://doi.org/10.1007/s11511-012-0088-0},
}

@article {JKLO,
    AUTHOR = {Jeon, Woojin and Kim, Inkang and Lecuire, Cyril and Ohshika, Ken'ichi},
     TITLE = {Primitive stable representations of free {K}leinian groups},
   JOURNAL = {Israel J. Math.},
  FJOURNAL = {Israel Journal of Mathematics},
    VOLUME = {199},
      YEAR = {2014},
    NUMBER = {2},
     PAGES = {841--866},
      ISSN = {0021-2172,1565-8511},
   MRCLASS = {30F40 (20C15 20H10)},
  MRNUMBER = {3219560},
MRREVIEWER = {J.\ D.\ Dixon},
       DOI = {10.1007/s11856-013-0062-3},
       URL = {https://doi.org/10.1007/s11856-013-0062-3},
}

@article{BIW,
   AUTHOR = {Burger, Marc and Iozzi, Alessandra and Wienhard, Anna},
     TITLE = {Surface group representations with maximal {T}oledo invariant},
   JOURNAL = {Ann. of Math. (2)},
  FJOURNAL = {Annals of Mathematics. Second Series},
    VOLUME = {172},
      YEAR = {2010},
    NUMBER = {1},
     PAGES = {517--566},
      ISSN = {0003-486X,1939-8980},
   MRCLASS = {22E41 (20F67 57M07)},
  MRNUMBER = {2680425},
       DOI = {10.4007/annals.2010.172.517},
       URL = {https://doi.org/10.4007/annals.2010.172.517},
}

@article{wood_bundles,
    AUTHOR = {Wood, John W.},
     TITLE = {Bundles with totally disconnected structure group},
   JOURNAL = {Comment. Math. Helv.},
  FJOURNAL = {Commentarii Mathematici Helvetici},
    VOLUME = {46},
      YEAR = {1971},
     PAGES = {257--273},
      ISSN = {0010-2571,1420-8946},
   MRCLASS = {32L05 (55F25)},
  MRNUMBER = {293655},
MRREVIEWER = {R.\ L. E. Schwarzenberger},
       DOI = {10.1007/BF02566843},
       URL = {https://doi.org/10.1007/BF02566843},
}

@article{mar-wol,
	author = {March\'e, Julien and Wolff, Maxime},
	journal = {Duke Math. J.},
	number = {2},
	pages = {371--412},
	title = {The modular action on {${\rm PSL}(2,\mathbb{R})$}-characters in genus 2},
	volume = {165},
	year = {2016}}

@article{canary-storm,
    AUTHOR = {Canary, Richard D. and Storm, Peter A.},
     TITLE = {Moduli spaces of hyperbolic 3-manifolds and dynamics on
              character varieties},
   JOURNAL = {Comment. Math. Helv.},
  FJOURNAL = {Commentarii Mathematici Helvetici. A Journal of the Swiss
              Mathematical Society},
    VOLUME = {88},
      YEAR = {2013},
    NUMBER = {1},
     PAGES = {221--251},
      ISSN = {0010-2571,1420-8946},
   MRCLASS = {57M50},
  MRNUMBER = {3008919},
       DOI = {10.4171/CMH/284},
       URL = {https://doi.org/10.4171/CMH/284},
}

@book {CDP,
    AUTHOR = {Coornaert, Michel and Delzant, Thomas and Papadopoulos, Athanase},
     TITLE = {G\'{e}om\'{e}trie et th\'{e}orie des groupes},
    SERIES = {Lecture Notes in Mathematics},
    VOLUME = {1441},
      NOTE = {Les groupes hyperboliques de Gromov. [Gromov hyperbolic
              groups],
              With an English summary},
 PUBLISHER = {Springer-Verlag, Berlin},
      YEAR = {1990},
     PAGES = {x+165},
      ISBN = {3-540-52977-2},
   MRCLASS = {57M07 (20F32)},
  MRNUMBER = {1075994},
MRREVIEWER = {John\ Meier},
}

@book{bridson-haefliger,
    AUTHOR = {Bridson, Martin R. and Haefliger, Andr\'{e}},
     TITLE = {Metric spaces of non-positive curvature},
    SERIES = {Grundlehren der mathematischen Wissenschaften [Fundamental
              Principles of Mathematical Sciences]},
    VOLUME = {319},
 PUBLISHER = {Springer-Verlag, Berlin},
      YEAR = {1999},
     PAGES = {xxii+643},
      ISBN = {3-540-64324-9},
   MRCLASS = {53C23 (20F65 53C70 57M07)},
  MRNUMBER = {1744486},
MRREVIEWER = {Athanase\ Papadopoulos},
       DOI = {10.1007/978-3-662-12494-9},
       URL = {https://doi.org/10.1007/978-3-662-12494-9},
}

@article {marden_geo,
    AUTHOR = {Marden, Albert},
     TITLE = {The geometry of finitely generated kleinian groups},
   JOURNAL = {Ann. of Math. (2)},
  FJOURNAL = {Annals of Mathematics. Second Series},
    VOLUME = {99},
      YEAR = {1974},
     PAGES = {383--462},
      ISSN = {0003-486X},
   MRCLASS = {30A58 (22E40)},
  MRNUMBER = {349992},
MRREVIEWER = {C.\ Earle},
       DOI = {10.2307/1971059},
       URL = {https://doi.org/10.2307/1971059},
}

@article{wang2021anosov,
    AUTHOR = {Wang, Tianqi},
     TITLE = {Anosov representations over closed subflows},
   JOURNAL = {Trans. Amer. Math. Soc.},
  FJOURNAL = {Transactions of the American Mathematical Society},
    VOLUME = {376},
      YEAR = {2023},
    NUMBER = {9},
     PAGES = {6177--6214},
      ISSN = {0002-9947,1088-6850},
   MRCLASS = {53C30 (20H10 22E40 37D20 37D40)},
  MRNUMBER = {4630773},
MRREVIEWER = {Sadayoshi\ Kojima},
       DOI = {10.1090/tran/8920},
       URL = {https://doi.org/10.1090/tran/8920},
}

@article {lee-xu,
    AUTHOR = {Lee, Jaejeong and Xu, Binbin},
     TITLE = {Bowditch's {Q}-conditions and {M}insky's primitive stability},
   JOURNAL = {Trans. Amer. Math. Soc.},
  FJOURNAL = {Transactions of the American Mathematical Society},
    VOLUME = {373},
      YEAR = {2020},
    NUMBER = {2},
     PAGES = {1265--1305},
      ISSN = {0002-9947,1088-6850},
   MRCLASS = {20E05 (57M60)},
  MRNUMBER = {4068264},
MRREVIEWER = {Doron\ Puder},
       DOI = {10.1090/tran/7953},
       URL = {https://doi.org/10.1090/tran/7953},
}

@article {brown-anosov,
    AUTHOR = {Brown, Richard J.},
     TITLE = {Anosov mapping class actions on the {${\rm
              SU}(2)$}-representation variety of a punctured torus},
   JOURNAL = {Ergodic Theory Dynam. Systems},
  FJOURNAL = {Ergodic Theory and Dynamical Systems},
    VOLUME = {18},
      YEAR = {1998},
    NUMBER = {3},
     PAGES = {539--554},
      ISSN = {0143-3857,1469-4417},
   MRCLASS = {58F15 (57S15 58D29 58F11 58F27)},
  MRNUMBER = {1631712},
MRREVIEWER = {Jonathan\ A.\ Poritz},
       DOI = {10.1017/S0143385798108258},
       URL = {https://doi.org/10.1017/S0143385798108258},
}

@article {FGLM-nonerg,
    AUTHOR = {Forni, Giovanni and Goldman, William M. and Lawton, Sean and
              Matheus, Carlos},
     TITLE = {Non-ergodicity on {$\rm SU(2)$} and {$\rm SU(3)$} character
              varieties of the once-punctured torus},
   JOURNAL = {Ann. H. Lebesgue},
  FJOURNAL = {Annales Henri Lebesgue},
    VOLUME = {7},
      YEAR = {2024},
     PAGES = {1099--1130},
      ISSN = {2644-9463},
   MRCLASS = {14M35 (22D40 37A25 37J40 53D30 70H08)},
  MRNUMBER = {4799915},
MRREVIEWER = {Thomas\ Ward},
       DOI = {10.5802/ahl.216},
       URL = {https://doi.org/10.5802/ahl.216},
}

@article {kam-theory,
    AUTHOR = {Arnol'd, Vladimir I.},
     TITLE = {Small denominators and problems of stability of motion in
              classical and celestial mechanics},
   JOURNAL = {Uspehi Mat. Nauk},
  FJOURNAL = {Akademija Nauk SSSR i Moskovskoe Matemati\v ceskoe Ob\v s\v
              cestvo. Uspehi Matemati\v ceskih Nauk},
    VOLUME = {18},
      YEAR = {1963},
    NUMBER = {6(114)},
     PAGES = {91--192},
      ISSN = {0042-1316},
   MRCLASS = {85.57 (57.48)},
  MRNUMBER = {170705},
MRREVIEWER = {J.\ Moser},
}

@phdthesis{lee_thesis,
	author = {Lee, Michelle},
	school = {Universityof Michigan},
	title = {Dynamics on $\mathsf{PSL}(2, \mathbb{C})$--character varieties of certain
hyperbolic $3$--manifolds},
	year = {2012}}

@book {FM,
    AUTHOR = {Farb, Benson and Margalit, Dan},
     TITLE = {A primer on mapping class groups},
    SERIES = {Princeton Mathematical Series},
    VOLUME = {49},
 PUBLISHER = {Princeton University Press, Princeton, NJ},
      YEAR = {2012},
     PAGES = {xiv+472},
      ISBN = {978-0-691-14794-9},
   MRCLASS = {57M50 (20F36 20F65 57M07 57N05)},
  MRNUMBER = {2850125},
MRREVIEWER = {Stephen\ P.\ Humphries},
}

@article {mcmullen_complex,
    AUTHOR = {McMullen, Curtis T.},
     TITLE = {Complex earthquakes and {T}eichm\"{u}ller theory},
   JOURNAL = {J. Amer. Math. Soc.},
  FJOURNAL = {Journal of the American Mathematical Society},
    VOLUME = {11},
      YEAR = {1998},
    NUMBER = {2},
     PAGES = {283--320},
      ISSN = {0894-0347,1088-6834},
   MRCLASS = {32G15 (20H10 30F60 57N05)},
  MRNUMBER = {1478844},
MRREVIEWER = {Darryl\ McCullough},
       DOI = {10.1090/S0894-0347-98-00259-8},
       URL = {https://doi.org/10.1090/S0894-0347-98-00259-8},
}

@article {magid,
    AUTHOR = {Magid, Aaron D.},
     TITLE = {Deformation spaces of {K}leinian surface groups are not
              locally connected},
   JOURNAL = {Geom. Topol.},
  FJOURNAL = {Geometry \& Topology},
    VOLUME = {16},
      YEAR = {2012},
    NUMBER = {3},
     PAGES = {1247--1320},
      ISSN = {1465-3060,1364-0380},
   MRCLASS = {57M50 (30F40)},
  MRNUMBER = {2967052},
MRREVIEWER = {Yasushi\ Yamashita},
       DOI = {10.2140/gt.2012.16.1247},
       URL = {https://doi.org/10.2140/gt.2012.16.1247},
}

@article{saadi-nonerg,
	author = {Saadi, Fayssal},
	fjournal = {Commentarii Mathematici Helvetici},
	issn = {1420-8946},
	journal = {Comment. Math. Helv.},
	title = {Non-ergodicity on the {$\rm SU(2)$}-character varieties},
	year = {2025},
    note = {published online first, arXiv:2404.00372v2}}

@article {gelander-minsky,
    AUTHOR = {Gelander, Tsachik and Minsky, Yair},
     TITLE = {The dynamics of {${\rm Aut}(F_n)$} on redundant
              representations},
   JOURNAL = {Groups Geom. Dyn.},
  FJOURNAL = {Groups, Geometry, and Dynamics},
    VOLUME = {7},
      YEAR = {2013},
    NUMBER = {3},
     PAGES = {557--576},
      ISSN = {1661-7207,1661-7215},
   MRCLASS = {37A25 (20E05 20F28 22D40 57M50)},
  MRNUMBER = {3095709},
       DOI = {10.4171/GGD/197},
       URL = {https://doi.org/10.4171/GGD/197},
}

@incollection {lub-survey,
    AUTHOR = {Lubotzky, Alexander},
     TITLE = {Dynamics of {${\rm Aut}(F_N)$} actions on group presentations
              and representations},
 BOOKTITLE = {Geometry, rigidity, and group actions},
    SERIES = {Chicago Lectures in Math.},
     PAGES = {609--643},
 PUBLISHER = {Univ. Chicago Press, Chicago, IL},
      YEAR = {2011},
      ISBN = {978-0-226-23788-6; 0-226-23788-5},
   MRCLASS = {20F28 (37C85 57M07)},
  MRNUMBER = {2807845},
MRREVIEWER = {Alexander\ Fel\cprime shtyn},
}

@unpublished{schlich-equivalence,
	author = {Schlich, Suzanne},
	note = {preprint, arXiv:2511.10551},
	title = {Bowditch representations in {G}romov-hyperbolic spaces: characterizations, dynamics of $\mathrm{Out}({F}_2)$ and recognition},
	year = {2025}}

@book{dehn_papers_1987,
	address = {New York, NY},
	title = {Papers on {Group} {Theory} and {Topology}},
	copyright = {http://www.springer.com/tdm},
	isbn = {978-1-4612-9107-7 978-1-4612-4668-8},
	url = {http://link.springer.com/10.1007/978-1-4612-4668-8},
	doi = {10.1007/978-1-4612-4668-8},
	language = {en},
	urldate = {2026-06-12},
	publisher = {Springer New York},
	author = {Dehn, Max},
	year = {1987},
}

@phdthesis{lupi,
	author = {Lupi, Damiano},
	school = {University of Warwick},
	title = {Primitive stability and {B}owditch conditions for rank 2 free group representations},
	year = {2015}}

@article {bourdon,
    author = {Bourdon, Marc},
     TITLE = {Structure conforme au bord et flot g\'{e}od\'{e}sique d'un
              {${\rm CAT}(-1)$}-espace},
   JOURNAL = {Enseign. Math. (2)},
  FJOURNAL = {L'Enseignement Math\'{e}matique. Revue Internationale. 2e
              S\'{e}rie},
    VOLUME = {41},
      YEAR = {1995},
    NUMBER = {1-2},
     PAGES = {63--102},
      ISSN = {0013-8584},
   MRCLASS = {58F17 (20F05 20F32 57S30)},
  MRNUMBER = {1341941},
MRREVIEWER = {Michel\ Coornaert}}

@unpublished{remfort-aurat,
	author = {Remfort-Aurat, Ulysse},
    note = {preprint, arXiv:2605.28891},
    title = {Simple-stable representations of surface groups in $\mathsf{PU}(2,1)$},
	year = {2026}}

@phdthesis{bouilly,
	author = {Bouilly, Yohann},
	school = {Universit\'e de Strasbourg},
	title = {Ergodic actions of Torelli subgroup on compact character varieties and pure modular group on relative character varieties and topological dynamics of modular group},
	year = {2021}}

@article {chillingworth,
    AUTHOR = {Chillingworth, David R. J.},
     TITLE = {A finite set of generators for the homeotopy group of a
              non-orientable surface},
   JOURNAL = {Proc. Cambridge Philos. Soc.},
  FJOURNAL = {Proceedings of the Cambridge Philosophical Society},
    VOLUME = {65},
      YEAR = {1969},
     PAGES = {409--430},
      ISSN = {0008-1981},
   MRCLASS = {57.47},
  MRNUMBER = {235583},
MRREVIEWER = {H.\ R.\ Gluck},
       DOI = {10.1017/s0305004100044388},
       URL = {https://doi.org/10.1017/s0305004100044388},
}

@article {lickorish,
    AUTHOR = {Lickorish, W. B. R.},
     TITLE = {Homeomorphisms of non-orientable two-manifolds},
   JOURNAL = {Proc. Cambridge Philos. Soc.},
  FJOURNAL = {Proceedings of the Cambridge Philosophical Society},
    VOLUME = {59},
      YEAR = {1963},
     PAGES = {307--317},
      ISSN = {0008-1981},
   MRCLASS = {54.75},
  MRNUMBER = {145498},
MRREVIEWER = {H.\ R.\ Gluck},
       DOI = {10.1017/s0305004100036926},
       URL = {https://doi.org/10.1017/s0305004100036926},
}

@article {lickorish-note,
    AUTHOR = {Lickorish, W. B. R.},
     TITLE = {On the homeomorphisms of a non-orientable surface},
   JOURNAL = {Proc. Cambridge Philos. Soc.},
  FJOURNAL = {Proceedings of the Cambridge Philosophical Society},
    VOLUME = {61},
      YEAR = {1965},
     PAGES = {61--64},
      ISSN = {0008-1981},
   MRCLASS = {54.75},
  MRNUMBER = {169221},
MRREVIEWER = {H.\ R.\ Gluck},
       DOI = {10.1017/s0305004100038640},
       URL = {https://doi.org/10.1017/s0305004100038640},
}

@article {palesi-psl,
    AUTHOR = {Palesi, Fr\'{e}d\'{e}ric},
     TITLE = {Connected components of spaces of representations of
              non-orientable surfaces},
   JOURNAL = {Comm. Anal. Geom.},
  FJOURNAL = {Communications in Analysis and Geometry},
    VOLUME = {18},
      YEAR = {2010},
    NUMBER = {1},
     PAGES = {195--217},
      ISSN = {1019-8385,1944-9992},
   MRCLASS = {53D30 (57M10 57N05)},
  MRNUMBER = {2660463},
MRREVIEWER = {Lee-Peng\ Teo},
       DOI = {10.4310/CAG.2010.v18.n1.a8},
       URL = {https://doi.org/10.4310/CAG.2010.v18.n1.a8},
}

@book {erg-decomp,
    AUTHOR = {Furstenberg, Hillel},
     TITLE = {Recurrence in ergodic theory and combinatorial number theory},
      NOTE = {M. B. Porter Lectures},
 PUBLISHER = {Princeton University Press, Princeton, NJ},
      YEAR = {1981},
     PAGES = {xi+203},
      ISBN = {0-691-08269-3},
   MRCLASS = {28D05 (10K10 10L10 54H20)},
  MRNUMBER = {603625},
MRREVIEWER = {Michael\ Keane},
}

@article {nonor_mcg,
    AUTHOR = {Mangler, W.},
     TITLE = {Die {K}lassen von topologischen {A}bbildungen einer
              geschlossenen {F}l\"ache auf sich},
   JOURNAL = {Math. Z.},
  FJOURNAL = {Mathematische Zeitschrift},
    VOLUME = {44},
      YEAR = {1939},
    NUMBER = {1},
     PAGES = {541--554},
      ISSN = {0025-5874,1432-1823},
   MRCLASS = {99-04},
  MRNUMBER = {1545786},
       DOI = {10.1007/BF01210672},
       URL = {https://doi.org/10.1007/BF01210672},
}

@book {peter-weyl,
    AUTHOR = {Applebaum, David},
     TITLE = {Probability on compact {L}ie groups},
    SERIES = {Probability Theory and Stochastic Modelling},
    VOLUME = {70},
      NOTE = {With a foreword by Herbert Heyer},
 PUBLISHER = {Springer, Cham},
      YEAR = {2014},
     PAGES = {xxvi+217},
      ISBN = {978-3-319-07841-0; 978-3-319-07842-7},
   MRCLASS = {60B15 (22E30 43A05 43A30 43A77 62G07)},
  MRNUMBER = {3243650},
MRREVIEWER = {Maria\ Gordina},
       DOI = {10.1007/978-3-319-07842-7},
       URL = {https://doi.org/10.1007/978-3-319-07842-7},
}

@article {flo-law,
    AUTHOR = {Florentino, Carlos and Lawton, Sean},
     TITLE = {Topology of character varieties of {A}belian groups},
   JOURNAL = {Topology Appl.},
  FJOURNAL = {Topology and its Applications},
    VOLUME = {173},
      YEAR = {2014},
     PAGES = {32--58},
      ISSN = {0166-8641,1879-3207},
   MRCLASS = {14L35 (22E46)},
  MRNUMBER = {3227204},
MRREVIEWER = {Wen-Wei\ Li},
       DOI = {10.1016/j.topol.2014.05.009},
       URL = {https://doi.org/10.1016/j.topol.2014.05.009},
}

@article {MB-domination,
    AUTHOR = {Martin-Baillon, Florestan},
     TITLE = {Dominating {${\rm CAT}$} {$(-1)$} surface group
              representations by {F}uchsian ones},
   JOURNAL = {Geom. Dedicata},
  FJOURNAL = {Geometriae Dedicata},
    VOLUME = {217},
      YEAR = {2023},
    NUMBER = {3},
     PAGES = {Paper No. 45, 19},
      ISSN = {0046-5755,1572-9168},
   MRCLASS = {51F99 (20F34 57K20)},
  MRNUMBER = {4553987},
MRREVIEWER = {David\ Matthew\ Freeman},
       DOI = {10.1007/s10711-023-00782-2},
       URL = {https://doi.org/10.1007/s10711-023-00782-2},
}

@article {Tykhyy,
    AUTHOR = {Tykhyy, Yuriy},
     TITLE = {Finite orbits of monodromies of rank two {F}uchsian systems},
   JOURNAL = {Anal. Math. Phys.},
  FJOURNAL = {Analysis and Mathematical Physics},
    VOLUME = {12},
      YEAR = {2022},
    NUMBER = {5},
     PAGES = {Paper No. 122, 42},
      ISSN = {1664-2368,1664-235X},
   MRCLASS = {34M55 (15A30 20F36)},
  MRNUMBER = {4483398},
       DOI = {10.1007/s13324-022-00698-2},
       URL = {https://doi.org/10.1007/s13324-022-00698-2},
}

@article {Souto-Storm,
    AUTHOR = {Souto, Juan and Storm, Peter},
     TITLE = {Dynamics of the mapping class group action on the variety of
              {${\rm PSL}_2\mathbb{C}$} characters},
   JOURNAL = {Geom. Topol.},
  FJOURNAL = {Geometry and Topology},
    VOLUME = {10},
      YEAR = {2006},
     PAGES = {715--736},
      ISSN = {1465-3060,1364-0380},
   MRCLASS = {57M50},
  MRNUMBER = {2240903},
MRREVIEWER = {Kentaro\ Ito},
       DOI = {10.2140/gt.2006.10.715},
       URL = {https://doi.org/10.2140/gt.2006.10.715},
}

\end{document}